\documentclass[a4paper, 10pt, notitlepage]{article}

\usepackage{amsthm, amsmath, amssymb, latexsym}
\usepackage{mathrsfs,cite}
\usepackage[multiple]{footmisc}
\usepackage{graphicx}

\usepackage[ruled,vlined]{algorithm2e}

\theoremstyle{plain}
\newtheorem{theorem}{Theorem}
\newtheorem{lemma}{Lemma}
\newtheorem{corollary}{Corollary}
\newtheorem{proposition}{Proposition}

\theoremstyle{definition}
\newtheorem{definition}{Definition}

\newtheorem{assumption}{Assumption}

\theoremstyle{remark}
\newtheorem{remark}{Remark}

\usepackage[a4paper]{geometry}	

\DeclareMathOperator*{\esssup}{ess\,sup}
\DeclareMathOperator{\co}{co}

\DeclareMathOperator*{\minimise}{minimise}

\author{M.V. Dolgopolik\footnote{Institute for Problems in Mechanical Engineering, Russian Academy of Sciences, Saint
Petersburg, Russia}\footnote{This work was performed in IPME RAS and supported by the Russian Science Foundation (Grant
No. 20-71-10032).}}
\title{An adaptive exact penalty method for nonsmooth optimal control problems with nonsmooth nonconvex state and
control constraints}

\begin{document}

\maketitle

\begin{abstract}
A class of exact penalty-type local search methods for optimal control problems with nonsmooth cost functional,
nonsmooth (but continuous) dynamics, and nonsmooth state and control constraints is presented, in which the the penalty
parameter and several line search parameters are adaptively adjusted during the optimisation process. This class of
methods is applicable to problems having a known DC (Difference-of-Convex functions) structure and in its core is based
on the classical DCA method, combined with the steering exact penalty rules for updating the penalty parameter and an
adaptive nonmonotone line search procedure. Under the assumption that all auxiliary subproblems are solved only
approximately (that is, with finite precision), we prove the correctness of the proposed family of methods and present
its detailed convergence analysis. The performance of several different versions of the method is illustrated by means
of a numerical example, in which the methods are applied to a semi-academic optimal control problem with a nonsmooth
nonconvex state constraint.
\end{abstract}

\section{Introduction}

Nonsmooth optimal control problems for continuous (as opposed discontinuous/switching) systems naturally arise in many
applied fields, including economics and engineering (for particular examples see, e.g. Clarke\cite{Clarke}, Outrata et
al.\cite{Outrata83,OutrataSchindler,Outrata88}, Dyer and McReynolds\cite{DyerMcReynolds}, and Krylov\cite{Krylov},
etc.). Although such nonsmooth optimal control problems have been an object of active theoretical research for many
years (see, e.g. Clarke\cite{Clarke}, Vinter\cite{Vinter}, Loewen\cite{Loewen}, Mordukhovich\cite{Mordukhovich}, etc.),
relatively little effort has been put into development of numerical methods for solving such problems for ODE systems.
Smooth penalty methods for some minimax optimal control problems were proposed by Gorelik and
Tarakanov\cite{GorelikTarakanov89,GorelikTarakanov92}. An efficient numerical method for optimal control of piecewise
smooth systems was recently developed by Nurkanovi\'{c} and Diehl\cite{NurkanovicDiehl}, while a method based on
discretisation of nonsmooth optimal control problems with state constraints with the use of the so-called pseudospectral
knotting technique was presented by Ross and Fahroo\cite{RossFahroo}. An approach to some nonsmooth optimal control
problems based on smoothing approximations was considered by Noori Skandari et al.\cite{Noori2015,Noori2016}. Finally,
the quasidifferential descent method for optimal control problems with nonsmooth objective functional was developed by
Fominyh\cite{Fominyh} (see also the references therein). 

The main goal of this paper is to present a class of exact penalty-type local search methods for general nonsmooth
optimal control problems for continuous systems with nonsmooth nonconvex state and control constraints. Exact penalty
functions and exact penalty methods for various optimal control problems have attracted a significant attention of
researchers. Local properties of exact penalty functions for optimal control problems with state constraints were first
studied by Lasserre\cite{Lasserre}, and Xing et al.\cite{Xing89,Xing94}. Global properties of exact penalty functions
were studied for free-endpoint problems by Demyanov et
al.\cite{DemyanovKarelin98,DemyanovKarelin2000_InCollect,DemyanovKarelin2000}, problems with state inequality constrains
by Hammoudi and Benharrat\cite{HammoudiBenharrat}, and problems for some implicit control systems by Demyanov et
al.\cite{DemyanovTamasyan2005}. A general theory of exact penalty functions for optimal control problems was developed
by Dolgopolik et al.\cite{DolgopolikFominyh,Dolgopolik_OptControl}.

Exact penalty methods for smooth optimal control problems, including problems with state constraints, have been
developed by Maratos\cite{MaratosPHD}, Mayne et al.\cite{MaynePolak80,MaynePolak80_2,MaynePolak87,SmithMayne88},
Polak\cite{Polak_book}, and Fominyh et al.\cite{FominyhKarelin2018}. An exact penalty method for optimal control 
problems for time-delay systems was analysed by Wong and Teo\cite{WongTeo91}. Exact penalty methods based on the Huyer
and Neumaier penalty function\cite{HuyerNeumaier,Dolgopolik_ExPen_II} for state constrained optimal control
problems were studied in Refs.\cite{LiYu2011,JiangLin2012,LinLoxton2014}, while exact penalty methods for general
\textit{nonsmooth} optimal control problem, based on bundle methods from nonsmooth optimisation, were developed by
Outrata et al.\cite{Outrata83,OutrataSchindler,Outrata88}.

In this paper, we develop a class of exact penalty methods for constrained nonsmooth optimal control problems based on
DC (Difference-of-Convex functions) optimisation techniques. DC optimisation is based on the use of decompositions of
the objective function and constraints as the difference of convex functions, which allows one to utilise the
well-developed tools of convex analysis and optimisation to tackle DC optimisation problems. This approach proved to be
very fruitful and has found a wide range of applications in various
fields\cite{LeThiDinh2005,LeThiDinh2018,Nouiehed,LippBoyd,Tuy95_Collect,HorstThoai,Tuy_2000}.

One of the most well-known smooth and nonsmooth DC optimisation methods is the so-called DCA (Difference-of-Convex
functions Algorithm) method, developed and applied to various problems in the works of Pham Dinh and Le Thi et
al.\cite{PhamDinhLeThi96,DinhLeThi1997,LeThiPhamDinh2014,LeThiPhamDinh2014b} (for a detailed survey of DCA, its
modifications, and applications see Refs.\cite{LeThiDinh2005,LippBoyd,LeThiDinh2018,AckooijDeOliveira2019}). Recently,
an improved version of the DCA, called the boosted
DCA\cite{AragonArtachoFlemingVuong,AragonArtachoVuong,FerreiraSantosSouza}, was
proposed. In contrast to the original DCA, boosted DCA utilises line search (\textit{nonmonotone} line search in the
nonsmooth case; see Ref.\cite{FerreiraSantosSouza}) to improve its overall performance. As was demonstrated by
numerical experiments in the aforementioned references, the boosted DCA significantly outperforms the standard DCA on
various test problems.

Recently, DCA was extended to the case of finite dimensional nonsmooth DC optimisation problems with equality and
inequality constraints by Strekalovsky\cite{Strekalovsky2020} and Dolgopolik\cite{Dolgopolik_StExPenDCA} with the use
of the so-called steering exact penalty technique developed by Byrd et al.\cite{ByrdNocedalWaltz,ByrdLopezCalvaNocedal}
for nonlinear programming problems. Steering exact penalty technique provides simple rules for adjusting the penalty
parameter in a way that ensures global convergence of the corresponding exact penalty methods. Let us also note that DC
optimisation methods has been applied to \textit{smooth} optimal control problems in the works of Strekalovsky et
al.\cite{Strekalovsky2013,StrekalovskyYanulevich2008,StrekalovskyYanulevich2013,StrY2016,Strekalovsky2021}. However, to
the best of the author's knowledge, DC optimisation approach to \textit{nonsmooth} optimal control problems has not been
considered earlier.

Our goal is to develop a class of adaptive exact penalty methods, called the boosted steering exact penalty DCA 
(B-STEP-DCA), for constrained nonsmooth optimal control problems, whose DC structure is explicitly known, based on the
combination of steering exact penalty
techniques\cite{ByrdNocedalWaltz,ByrdLopezCalvaNocedal,Strekalovsky2020,Dolgopolik_StExPenDCA} and the boosted DCA with
nonmonotone line search\cite{AragonArtachoFlemingVuong,AragonArtachoVuong,FerreiraSantosSouza} under the assumption
that all computations are performed with finite precision. The use of the steering exact penalty techniques allows the
method to automatically and adaptively adjust the penalty parameter to ensure convergence of the constructed sequence
and sufficient decrease of the infeasibility measure on each iteration. In turn, the use of nonmonotone line search with
adaptive choice of the trial step size and line search tolerance, proposed in Ref.~\cite{FerreiraSantosSouza} for
unconstrained problems, is aimed at improving the overall efficiency of the method and reducing the number of iterations
before termination. Moreover, the nonmonotone line search also has an innertial effect that in some cases might help the
method to ``jump over'' some critical points and find a better solution (cf.~Refs.\cite{deOliveiraTcheou}).

The paper is organised as follows. The problem formulation and main assumptions on problem data are presented in
Section~\ref{sect:ProblemFormulation}. Section~\ref{sect:OptimalityConditions} is devoted to a discussion of optimality
conditions and approximate criticality notions for nonsmooth DC optimal control problems. A detailed description of
B-STEP-DCA, including the description of stopping criteria, strategies for choosing nonmonotone line search tolerances,
and rules for adjusting the trial step sizes, is given in Section~\ref{sect:MethodDescription}. The correctness of the
method, that is the fact that all auxiliary subproblems are solved only a finite number of times on each iteration of
the method, is proved in Section~\ref{sect:Correctness}, while global convergence of the method is studied in
Section~\ref{sect:ConvergenceToCriticalPoints}. Some elementary convergence properties of the method, as well as
justification of stopping criteria, are presented in Subsection~\ref{subsect:ElementaryAnalysis}. The convergence of
sequences generated by B-STEP-DCA to approximately critical points is proved in
Subsection~\ref{subsect:GlobalConvergence}, the convergence of trajectories vs. the convergence of controls is studied
in Subsection~\ref{subsect:Control_vs_traj}, while the convergence of infeasibility measure is analysed in
Subsection~\ref{subsect:InfeasMeas}. Finally, some results of numerical experiments demonstrating the robustness of
B-STEP-DCA with respect to computational errors and the effects of the nonmonotone line search on its performance are
presented in Section~\ref{sect:NumericalExperiments}.

\section{Problem formulation}
\label{sect:ProblemFormulation}

Let $W^{1, s}([0, T], \mathbb{R}^m)$ with $1 \le s \le \infty$ be the Banach space of all absolutely continuous
functions $x \colon [0, T] \to \mathbb{R}^m$ such that $\dot{x} \in L^s([0, T]; \mathbb{R}^m)$ equipped with the norm 
\[
  \| x \|_{1, s} = \| x \|_s + \| \dot{x} \|_s, \quad
  \| x \|_s = \bigg( \int_0^T |x(t)|^s \, dt \bigg)^{\frac{1}{s}} \quad \text{if } s < +\infty, \quad
  \| x \|_{\infty} = \esssup_{t \in [0, T]} |x(t)|,
\]
where $|\cdot|$ is the Euclidean norm. Recall that this space can be identified with the cooresponding Sobolev space 
(see, e.g. Leoni\cite{Leoni}). Denote 
$X := W^{1, \infty}([0, T], \mathbb{R}^n) \times L^{\infty}([0, T], \mathbb{R}^m)$, and let the space $X$ be endowed
with the norm $\| (x, u) \| = \| x \|_{1, \infty} + \| u \|_{\infty}$. To include several different problem
formulations into one setting, as well as to unite several versions of the method developed in this article into one
theoretical scheme, suppose that a convex set $X_0 \subseteq X$ closed in the topology of the space 
$W^{1, 1}([0, T], \mathbb{R}^n) \times L^1([0, T], \mathbb{R}^m)$ is given. 

The set $X_0$ can be viewed as one describing the nonfunctional constraint $(x, u) \in X_0$, that is, describing
the \textit{convex} constraints that are \textit{not} included into the penalty function and are taken into account with
the use of some other techniques. One can opt to include all constraints (both nonconvex and convex ones) into the
penalty function and define $X_0 = X$ or to include some (or all) convex constraints into the definition of the set
$X_0$ and take them into account within the convex optimisation subroutine used on each iteration of the method
presented in this article. Each choice leads to a different version of the method. The effect of such choice on the
overall performance of the method is discussed below and is studied numerically for a particular optimal control
problem in Section~\ref{sect:NumericalExperiments}. In addition to defining nonfunctional convex constraints, the set
$X_0$ can also describe a control parametrisation\cite{TeoGohWong,GohTeo,RehbocktTeoJenningsLee}/discretisation scheme.

Throughout this article we consider the following nonsmooth nonconvex optimal control problem with mixed
state-control constraints:
\begin{align*}
  &\minimise_{(x, u) \in X_0} \enspace J(x, u) = \int_0^T F_0(x(t), u(t), t) \, dt + f_0(x(0), x(T)),
  \\
  &\text{subject to} \enspace 
  \begin{aligned}[t]
    &\dot{x}(t) = F(x(t), u(t), t) \quad \text{for a.e. } t \in [0, T], \hspace{5cm} (\mathcal{P})
    \\
    &f_i(x(0), x(T)) \le 0, \quad i \in \mathcal{I}, \quad f_j(x(0), x(T)) = 0, \quad j \in \mathcal{E},
    \\
    &\Xi_s(x(t), u(t), t) \le 0 \quad \text{for a.e. } t \in [0, T], \quad s \in \mathcal{M},
  \end{aligned}
\end{align*}
Here $x(t) \in \mathbb{R}^n$ is the system state at time $t$, $u(t) \in \mathbb{R}^m$ is a control input, 
and the functions $F_0,\Xi_s \colon \mathbb{R}^n \times \mathbb{R}^m \times [0, T] \to \mathbb{R}$, 
$F \colon \mathbb{R}^n \times \mathbb{R}^m \times [0, T] \to \mathbb{R}^n$, and 
$f_i \colon \mathbb{R}^{2n} \to \mathbb{R}$ are assumed to be DC (Difference-of-Convex functions) with respect to 
$(x, u)$ (or simply DC in the case of $f_i$). We also suppose that DC decompositions of these functions of the form
\begin{gather*}
  F_0(x, u, t) = G_0(x, u, t) - H_0(x, u, t), \quad F(x, u, t) = G(x, u, t) - H(x, u, t), 
  \\
  f_i(x, y) = g_i(x, y) - h_i(x, y), \quad \Xi_s(x, u, t) = p_s(x, u, t) - q_s(x, u, t),
\end{gather*}
are known. Here the real-valued functions $G_0, H_0, p_s$, and $q_s$ are convex in $(x, u)$, the vector-valued functions
$G$ and $H$ are coordinate-wise convex in $(x, u)$, while the real-valued functions $g_i$ and $h_j$ are convex. No
smoothness assumptions on these functions are imposed (that is, they can be nonsmooth). We also
suppose that $\mathcal{I} = \{ 1, \ldots, \ell_I \}$, $\mathcal{E} = \{ \ell_I + 1, \ldots, \ell_E \}$, and
$\mathcal{M} = \{ 1, \ldots, \ell_M \}$ with $\ell_I, \ell_E, \ell_M \in \mathbb{N}$ are finite index sets, any one of
which can be empty. 

\begin{remark}
Let us note that the description of the exact penalty method, the proof of its correctness, and its convergence analysis
presented in this article are \textit{identical} for problems with only pure state, only mixed state-control, as well as
both pure and mixed, constraints. Therefore, below we consider only mixed state constraints, since they include pure
ones as a particular case. However, one should point out that an exhaustive convergence analysis of the method
(involving, e.g. an analysis of conditions ensuring the boundedness of the penalty parameter, convergence of dual
variables/Lagrange multipliers, etc.) would require differentiation between pure and mixed state constraints. We leave
an analysis of these advanced topics as an interesting problem for future research and instead concentrate on the
description of the method, and analysis of its correctness and basic convergence properties.
\end{remark}

\begin{remark}
Although constraints on control of the form $u(t) \in U(t)$ for a.e. $t \in [0, T]$ and some convex-valued multifunction
$U \colon [0, T] \rightrightarrows \mathbb{R}^m$ are not explicitly mentioned in the formulation of the problem 
$(\mathcal{P})$, they can be included into the problem via the set $X_0$, namely, by defining 
$X_0 = \{ (x, u) \in X \mid u(t) \in U(t) \text{ for a.e. } t \in [0, T] \}$. Alternatively, if the constraints
have the form $\underline{u}_i \le u_i(\cdot) \le \overline{u}_i$ for some vectors 
$\underline{u}, \overline{u} \in \mathbb{R}^m$, one can consider them as mixed constraints by putting
$\Xi_i(x, u, t) = u_i - \overline{u}_i$ and $\Xi_{m + i}(x, u, t) = \underline{u}_i - u_i$, $i \in \{ 1, \ldots, m \}$.
These two options lead to two significantly different versions of the exact penalty method (see
Remark~\ref{rmrk:BSEP_DCA_DifferentVersions} below for more details).
\end{remark}

Denote by $\partial f(x)$ the subdifferential (in the sense of convex analysis\cite{Rockafellar}) of a function $f$ at a
point $x$, and denote by $\partial_x g(x, y)$ the subdifferential (in the sense of convex analysis) of the function 
$g(\cdot, y)$ at a point $x$. Hereinafter, we impose the following assumption on measurability, continuity, and
generalised differentiability properties on the problem data.

\begin{assumption} \label{assumpt:ContDiff}
The following conditions hold true:
\begin{enumerate}
\item{$G_0, H_0, G, H, p_s$, and $q_s$ are Carath\'{e}odory functions (i.e. they are continuous in $(x, u)$ 
for a.e. $t \in [0, T]$, and measurable in $t$ for any $(x, u)$) such that for any $R > 0$ there exists $\gamma_R > 0$
for which
\begin{equation} \label{eq:GrowthCondition}
\begin{split}
  |G_0(x, u, t)| \le \gamma_R, \enspace &|H_0(x, u, t)| \le \gamma_R, \enspace
  |G(x, u, t)| \le \gamma_R, \enspace |H(x, u, t)| \le \gamma_R, 
  \\
  &|p_s(x, u, t)| \le \gamma_R, \enspace |q_s(x, u, t)| \le \gamma_R
\end{split}
\end{equation}
for a.e. $t \in [0, T]$ and for any $(x, u) \in \mathbb{R}^n \times \mathbb{R}^m$ with $|x| + |u| \le R$;
}

\item{for any $(x, u) \in X$ the subdifferential mappings 
\begin{gather*}
  \partial_{x, u} G_0(x(\cdot), u(\cdot), \cdot), \quad \partial_{x, u} H_0(x(\cdot), u(\cdot), \cdot), \quad
  \partial_{x, u} G(x(\cdot), u(\cdot), \cdot), \quad \partial_{x, u} H(x(\cdot), u(\cdot), \cdot), 
  \\
  \partial_{x, u} p_s(x(\cdot), u(\cdot), \cdot), \quad \partial_{x, u} q_s(x(\cdot), u(\cdot), \cdot)
\end{gather*}
are measurable and for any $R > 0$ there exists $\gamma_R > 0$ such that
\begin{gather*}
  |v_i| \le \gamma_R \quad 
  \forall v_i \in \partial_{x, u} G_i(x(t), u(t), t) \cup \partial_{x, u} H_i(x(t), u(t), t), \quad
  \forall i \in \{ 0, 1, \ldots, n \},
  \\
  |v| \le \gamma_R \quad 
  \forall  v \in \partial_{x, u} p_s(x(t), u(t), t) \cup \partial_{x, u} q_s(x(t), u(t), t), \quad
  \forall s \in \mathcal{M}
\end{gather*}
for a.e. $t \in [0, T]$ and for any $(x, u) \in \mathbb{R}^n \times \mathbb{R}^m$ with $|x| + |u| \le R$.
}
\end{enumerate}
\end{assumption}

\begin{remark}
It should be mentioned that assumption~\ref{assumpt:ContDiff} ensures that all integrals throughout this article are
correctly defined and finite. Moreover, if the maps $G_0, G, H_0, H, p_s$, and $q_s$ do not depend on $t$, then the
assumption on the existence of $\gamma_R$ satisfying the corresponding inequalities for these functions and their
subgradients is in actuality redundant. In this case, the existence of $\gamma_R > 0$ satisfying these assumptions
follows directly from well-known properties of convex functions (see, e.g. Crlr.~10.1.1 and Thm.~24.7 in
Ref.\cite{Rockafellar}).
\end{remark}

\section{Optimality conditions}
\label{sect:OptimalityConditions}

Our goal is to develop a class of local search methods for the problem $(\mathcal{P})$ that combines the boosted DCA
with nonmonotone line search\cite{AragonArtachoFlemingVuong,AragonArtachoVuong,FerreiraSantosSouza} and the exact
penalty technique based on the use of an $L_1$ exact penalty function and the steering exact penalty 
methodology\cite{ByrdNocedalWaltz,ByrdLopezCalvaNocedal,Strekalovsky2020,Dolgopolik_StExPenDCA}.
In the case of the problem $(\mathcal{P})$ this penalty function has the form 
$\Phi_c(x, u) = J(x, u) + c \varphi(x, u)$, where
\begin{equation} \label{eq:PenTerm}
\begin{split}
  \varphi(x, u) = \sum_{i = 1}^n \int_0^T |\dot{x}_i(t) - F_i(x(t), u(t), t)| \, dt
  &+ \sum_{i = 1}^{\ell_I} \max\{ 0, f_i(x(0), x(T) \} 
  \\
  &+ \sum_{j = \ell_I + 1}^{\ell_E} |f_j(x(0), x(T))|
  + \sum_{s = 1}^{\ell_M} \int_0^T \max\{ 0, \Xi_s(x(t), u(t), t) \} \, dt.
\end{split}
\end{equation}
is the $L_1$ penalty term and $c > 0$ is penalty parameter. 

\begin{remark} \label{rmrk:PenaltyTerms}
One can replace the $L_1$ penalty term with the corresponding $L_{\infty}$ penalty term
\begin{align*}
  \varphi(x, u)_{\infty} 
  = &\esssup_{t \in [0, T]} \max_{i \in \{ 1, \ldots, n \}} \big| \dot{x}_i(t) - F_i(x(t), u(t), t) \big|
  + \max\{ 0, f_1(x(0), x(T), \ldots, f_{\ell_I}(x(0), x(T) \} 
  \\
  &+ \max\big\{ |f_{\ell_I + 1}(x(0), x(T))|, \ldots, |f_{\ell_E}(x(0), x(T))| \big\}
  \\
  &+ \esssup_{t \in [0, T]} \max\big\{ 0, \Xi_1(x(t), u(t), t), \ldots, \Xi_{\ell_M}(x(t), u(t), t) \big\}
\end{align*}
or any mixture of both, that is, one can choose which constraints are penalised via the $L_1$ penalty term and which
constraints are penalised via the $L_{\infty}$ penalty term (in particular, one can penalise a part of mixed
constraints with the use of $L_1$ penalty term and another part with the use of $L_{\infty}$ penalty term). Let us
underline that both the description of the exact penalty method for the problem $(\mathcal{P})$ and its convergence
analysis presented below are \textit{identical} for any such choice of the penalty term $\varphi(x, u)$. We opted to use
the $L_1$ penalty term for the sake of definiteness and simplicity.
\end{remark}

The so-called \textit{local exactness} of a penalty function plays one the central roles both in the theory of exact
penalty functions and analysis of exact penalty methods for constrained optimisation
(see
Refs.~\cite{HanMangasarian,DiPilloGrippo86,DiPilloGrippo89,Zaslavski,Dolgopolik_ExPen_I,Dolgopolik_ExPen_II,
Strekalovsky2019}). Recall that the penalty function $\Phi_c$ is called locally
exact at a locally optimal solution $(x_*, u_*)$ of the problem $(\mathcal{P})$, if there exists a penalty parameter
$c_* \ge 0$ such that for any $c \ge c_*$ the pair $(x_*, u_*)$ is a locally optimal solution of the penalised problem
\begin{equation} \label{eq:LocalExactnessDef}
  \minimise_{(x, u) \in X_0} \: \Phi_c(x, u).
\end{equation}
Any such $c_* > 0$ is called \textit{an exact penalty parameter} of the penalty function $\Phi_c$ at the point 
$(x_*, u_*)$.

Various necessary and/or sufficient conditions for the local exactness of penalty functions for optimal control
problems and some closely related results can be found in
Refs.\cite{Lasserre,Xing89,Xing94,DemyanovKarelin98,DemyanovKarelin2000_InCollect,DemyanovKarelin2000,DolgopolikFominyh,
Dolgopolik_OptControl,HammoudiBenharrat}. Typically, one needs to impose a suitable constraint
qualification at a given point to ensure that the corresponding penalty function for an optimal control problem under
consideration is locally exact. In other words, the assumption on local exactness of a penalty function can be viewed as
an implicit constraint qualification.

Under the assumption that the penalty function $\Phi_c$ is locally exact, we can obtain convenient local optimality
conditions for the problem $(\mathcal{P})$ that can be used as a foundation for a local search DCA-type method for
this problem.

\begin{proposition} \label{prp:OptCond}
Let $(x_*, u_*)$ be a locally optimal solution of the problem $(\mathcal{P})$ and the penalty function $\Phi_c$ be
locally exact at the point $(x_*, u_*)$. Then there exists $c_* \ge 0$ such that for any $c \ge c_*$ and for any
subgradients 
\begin{gather} \notag
  V_l(t) \in \partial_{x, u} H_l(x_*(t), u_*(t), t), \quad l \in \{ 0, 1, \ldots, n \}, \quad
  W_l(t) \in \partial_{x, u} G_l(x_*(t), u_*(t), t), \quad l \in \{ 1, \ldots, n \}, 
  \\ \label{eq:SubgradientCollection}
  v_i \in \partial h_i(x_*(0), x_*(T)), \quad i \in \{ 0, 1, \ldots, \ell_E \}, \quad
  w_j \in \partial g_j(x_*(0), x_*(T)), \quad j \in \{ \ell_I + 1, \ldots, \ell_E \}, 
  \\ \notag
  z_s(t) \in \partial_{x, u} q_s(x_*(t), u_*(t), t), \quad s \in \{ 1, \ldots, \ell_M \},
\end{gather}
(the maps $V_l(\cdot)$, $W_l(\cdot)$, and $z_s(\cdot)$ are assumed to be measurable) the pair $(x_*, u_*)$ is a globally
optimal solution of the convex variational problem
\begin{equation} \label{eq:CriticalProblem}
  \minimise_{(x, u) \in X_0} \enspace Q_c(x, u; x_*, u_*; V),
\end{equation}
where $V = (V_0(\cdot), V_1(\cdot), \ldots, V_n(\cdot), W_1(\cdot), \ldots, W_n(\cdot), v_0, v_1, \ldots, v_{\ell_E}, 
w_{\ell_I + 1}, \ldots, w_{\ell_E}, z_1(\cdot), \ldots z_{\ell_m}(\cdot))$, and
\begin{align} \notag
  Q_c&(x, u; x_*, u_*; V) 
  := \int_0^T \Big( G_0(x(t), u(t), t) - \langle V_0(t), (x(t), u(t)) \rangle \Big) \, dt 
  + g_0(x(0), x(T)) - \langle v_0, (x(0), x(T)) \rangle
  \\ \notag
  &+ c \sum_{i = 1}^n \int_0^T \max\Big\{ 
  \begin{aligned}[t]
    &\dot{x}^{(i)}(t) + H_i(x(t), u(t), t) - G_i(x_*(t), u_*(t), t) 
    - \langle W_i(t), (x(t) - x_*(t), u(t) - u_*(t)) \rangle,
    \\ 
    &- \dot{x}^{(i)}(t) + G_i(x(t), u(t), t) - H_i(x_*(t), u_*(t), t) 
    - \langle V_i(t), (x(t) - x_*(t), u(t) - u_*(t)) \rangle \Big\} \, dt
  \end{aligned}
  \\ \notag
  &+ c \sum_{i = 1}^{\ell_I}
  \max\big\{ 0, g_i(x(0), x(T)) - h_i(x_*(0), x_*(T)) - \langle v_i, (x(0) - x_*(0), x(T) - x_*(T)) \rangle \big\}
  \\ \notag
  &+ c \sum_{j = \ell_I + 1}^{\ell_E}
  \begin{aligned}[t]
    \max\big\{ &g_i(x(0), x(T)) - h_i(x_*(0), x_*(T)) - \langle v_i, (x(0) - x_*(0), x(T) - x_*(T)) \rangle,
    \\
    &h_i(x(0), x(T)) - g_i(x_*(0), x_*(T)) - \langle w_i, (x(0) - x_*(0), x(T) - x_*(T)) \rangle \big\}
  \end{aligned}
  \\ \label{eq:ExPenConvexMajorant}
  &+ c \sum_{s = 1}^{\ell_M} \int_0^T \max\Big\{ 0, p_s(x(t), u(t), t) - q_s(x_*(t), u_*(t), t) 
  - \langle z_s(t), (x(t) - x_*(t), u(t) - u_*(t)) \rangle \Big\} \, dt.
\end{align}
\end{proposition}

\begin{proof}
Let $c_*$ be an exact penalty parameter of the penalty function $\Phi_c$ at $(x_*, u_*)$. Then for any $c \ge c_*$ 
the point $(x_*, u_*)$ is a locally optimal solution of problem \eqref{eq:LocalExactnessDef}. 

Observe that by the definition of subgradient for any $(x, u) \in X$ one has 
\begin{align} \notag
  Q_c(x, u; x_*, u_*; V) 
  &- \int_0^T \Big( H_0(x_*(t), u_*(t), t) - \langle V_0(t), (x_*(t), u_*(t)) \rangle \Big) \, dt
  - h_0(x_*(0), x_*(T)) + \langle v_0, (x_*(0), x_*(T)) \rangle  
  \\ \notag
  &\ge \int_0^T F_0(x(t), u(t), t) \, dt + f_0(x(0), x(T)) +   
  c \sum_{i = 1}^n \int_0^T |\dot{x}_i(t) - F_i(x(t), u(t), t)| \, dt
  \\ \notag
  &+ c \sum_{i = 1}^{\ell_I} \max\{ 0, f_i(x(0), x(T) \} + c \sum_{j = \ell_I + 1}^{\ell_E} |f_j(x(0), x(T))|
  + c \sum_{s = 1}^{\ell_M} \int_0^T \max\{ 0, \Xi_s(x(t), u(t), t) \} \, dt
  \\ \label{eq:GlobalConvexMajorant}
  &= \Phi_c(x, u)
\end{align}
and this inequality turns into equality for $(x, u) = (x_*, u_*)$. Therefore for any $c \ge c_*$ the pair $(x_*, u_*)$
is a locally optimal solution of problem \eqref{eq:CriticalProblem}, which thanks to the obvious convexity of 
the function $(x, u) \mapsto Q_c(x, u; x_*, u_*; V)$ and the fact that the set $X_0$ is by definition convex implies
that $(x_*, u_*)$ is globally optimal solution of problem \eqref{eq:CriticalProblem}.
\end{proof}

\begin{remark}
From the proof of the proposition above it follows that the function 
\[
  (x, u) \mapsto Q(x, u; x_*, u_*; V_*) - \int_0^T \Big( H_0(x_*(t), u_*(t), t) 
  - \langle V_0(t), (x_*(t), u_*(t)) \rangle \Big) \, dt - h_0(x_*(0), x_*(T)) + \langle v_0, (x_*(0), x_*(T))
\]
is a global convex majorant of the penalty function $\Phi_c$. Thus, the optimality conditions from the previous
propositions are expressed in terms of a family of global convex majorants $Q_c(\cdot; x_*, u_*, V_*)$ with $V_*$
running through the direct product of the corresponding subdifferentials. In turn, the optimality conditions state that
the point $(x_*, u_*)$ at which this family of majorants is constructed is a global minimiser of each of these majorants
on the convex set $X_0$.
\end{remark}

In many cases, the optimality conditions from the proposition above are too restrictive and cumbersome for applications,
since they require the computation of the entire subdifferentials of all corresponding convex functions and verification
of the optimality conditions \textit{for all} subgradients of these functions (that is, for \textit{all} convex
majorants $Q_c(\cdot; x_*, u_*, V)$), which is often either too computationally expensive or simply impossible. This
fact motivates us to introduce the well-known in DC optimisation 
(cf. Refs.~\cite{LeThiDinh2018,AckooijDeOliveira2019,JokiBagirov2020}) notion of \textit{criticality} for the problem
$(\mathcal{P})$. 

\begin{definition} \label{def:Criticality}
A feasible point $(x_*, u_*)$ of the problem $(\mathcal{P})$ is called \textit{critical} for this problem, if
\textit{there exists} a collection of subgradients $V$ of the corresponding convex functions at $(x_*, u_*)$, defined
in Eq.~\eqref{eq:SubgradientCollection}, and a penalty parameter $c_* > 0$ such that for any $c \ge c_*$ the pair 
$(x_*, u_*)$ is a globally optimal solution of the convex problem \eqref{eq:CriticalProblem}. 
A pair $(x_*, u_*) \in X_0$ is called \textit{a generalized critical point} for the problem $(\mathcal{P})$ for a given
value $c > 0$ of the penalty parameter, if there exist a collection of subgradients $V$ of the corresponding convex
functions, defined in Eq.~\eqref{eq:SubgradientCollection}, such that the pair $(x_*, u_*)$ is a globally optimal
solution of the convex problem \eqref{eq:CriticalProblem}.
\end{definition}

Clearly, any feasible point $(x_*, u_*)$ satisfying optimality conditions from Proposition~\ref{prp:OptCond} is critical
for the problem $(\mathcal{P})$; however, the converse statement is not true in the general nonsmooth case. The
criticality of $(x_*, u_*)$ is equivalent to the validity of optimality conditions from Proposition~\ref{prp:OptCond},
for example, if all functions $H_0$, $G$, $H$, $g_i$, $i \in \mathcal{E}$, $h_j$, $j \in \mathcal{I} \cup \mathcal{E}$,
and $q_s$, $s \in \mathcal{M}$, are smooth. For a more detailed discussion of interrelations between various optimality
conditions and criticality notions for DC optimisation problems see
Refs.\cite{LeThiDinh2018,AckooijDeOliveira2019,JokiBagirov2020}.

Generalized criticality extends the notion of criticality to the case of \textit{infeasible} points. Note, however, that
the generalized criticality of a pair $(x_*, u_*)$ depends on the choice of penalty parameter $c$. A pair $(x_*, u_*)$
might be a generalized critical point of the problem $(\mathcal{P})$ for some value $c$ of the penalty parameter, and
not be a generalized critical point for another value $c'$ of the penalty parameter, even if $c' \ge c$.

\begin{remark} \label{rmrk:ConvexProblem}
One can readily see that the criticality of a pair $(x_*, u_*)$ implies (and under some additional assumptions is
equivalent to) its global optimality in the convex case, that is, in the case when the mapping $F$ is affine in 
$(x, u)$, $F_0$ and $\Xi_s$, $s \in \mathcal{M}$, are convex in $(x, u)$, $f_i$, $i \in \mathcal{I} \cup \{ 0 \}$ are
convex, and $f_j$, $j \in \mathcal{J}$, are affine. In the nonconvex case, the criticality of $(x_*, u_*)$ can, roughly
speaking, be viewed as optimality of this pair for the convex problem  
\begin{align*}
  &\minimise_{(x, u) \in X_0} \int_0^T \Big( G_0(x(t), u(t), t) - \langle V_0(t), (x(t), u(t)) \Big) \, dt 
  + g_0(x(0), x(T)) - \langle v_0, (x(0), x(T)) \rangle, 
  \\
  &\text{s.t. } 
  \begin{aligned}[t]
    &\dot{x}(t) \in \mathcal{F}(x(t), u(t), x_*(t), u_*(t), t) \quad \text{for a.e. } t \in [0, T]
    \\
    & g_i(x(0), x(T)) - h_i(x_*(0), x_*(T)) - \langle v_i, (x(0) - x_*(0), x(T) - x_*(T)) \rangle \le 0, \quad 
    i \in \mathcal{I}, 
    \\
    &g_j(x(0), x(T)) - h_j(x_*(0), x_*(T)) - \langle v_j, (x(0) - x_*(0), x(T) - x_*(T)) \rangle \le 0, \quad
    j \in \mathcal{E},
    \\
    &h_j(x(0), x(T)) - g_j(x_*(0), x_*(T)) - \langle w_j, (x(0) - x_*(0), x(T) - x_*(T)) \rangle \le 0, \quad 
    j \in \mathcal{E},
    \\
    & p_s(x(t), u(t), t) - q_s(x_*(t), u_*(t), t) - \langle z_s(t), (x(t) - x_*(t), u(t) - u_*(t)) \rangle \le 0 
    \quad \text{for a.e. } t \in [0, T], \quad k \in \mathcal{M},
  \end{aligned}
\end{align*}
where
\begin{align*}
  \mathcal{F}_i(x(t), u(t), x_*(t), u_*(t), t) = \co\Big\{ 
  &G_i(x(t), u(t), t) - H_i(x_*(t), u_*(t), t) - \langle V_i(t), (x(t) - x_*(t), u(t) - u_*(t)) \rangle,
  \\
  &G_i(x_*(t), u_*(t), t) + \langle W_i(t), (x(t) - x_*(t), u(t) - u_*(t)) \rangle - H_i(x(t), u(t), t)
  \Big\}
\end{align*}
and ``$\co$'' stands for the convex hull. Note that 
$F(x(\cdot), u(\cdot), \cdot) \in \mathcal{F}(x(\cdot), u(\cdot), x_*(\cdot), u_*(\cdot), \cdot)$
for any $(x, u) \in X$ and $(x_*, u_*) \in X$, and the multifunction 
$(x, u) \mapsto \mathcal{F}(x, u, x_*(\cdot), u_*(\cdot), \cdot)$ is convex in the sense that
\[
  \mathcal{F}(\alpha x_1 + (1 - \alpha) x_2, \alpha u_1 + (1 - \alpha) u_2, x_*(\cdot), u_*(\cdot), \cdot)
  \subseteq \alpha \mathcal{F}(x_1, u_1, x_*(\cdot), u_*(\cdot), \cdot)
  + (1 - \alpha) \mathcal{F}(x_2, u_2, x_*(\cdot), u_*(\cdot), \cdot)
\]
for any $\alpha \in [0, 1]$, $x_1, x_2 \in \mathbb{R}^n$, and $u_1, u_2 \in \mathbb{R}^m$. Thus, the differential
inclusion $\dot{x}(t) \in \mathcal{F}(x(t), u(t), x_*(t), u_*(t), t)$ is a \textit{convex relaxation} of the
corresponding differential equation $\dot{x}(t) = F(x(t), u(t), t)$ at any given point $(x_*, u_*) \in X$.

The convex problem above is obtained from the original nonconvex problem $(\mathcal{P})$ by convexifing it
with the use of the DC decompositions of the cost function and all constraints, and the function 
$Q_c(\cdot; x_*, u_*; V)$ is a penalty function for this convex problem. Hence, in particular, under some additional
assumptions one can reformulate the criticality condition as the Pontryagin Maximum Principle for the convex problem
above.
\end{remark}

In the context of numerical methods (and from the practical point of view), one cannot deal with optimal solutions of
auxiliary subproblems directly, since only approximate solutions of such subproblems can be constructed numerically.
Therefore, to bridge the gap between a description and a theoretical analysis of a numerical method on the one hand, and
its practical implementation on the other hand, it is natural to replace optimality with approximate optimality 
($\varepsilon$-optimality) of corresponding solutions. Such replacement, motivated by the issues of practical
implementation (in particular, motivated by effects of discretisation), leads us the notion of \textit{approximate
criticality} of feasible points of the problem $(\mathcal{P})$ (cf. a similar notion of approximate optimality for
inequality constrained finite dimensional DC optimisation problem in Ref.\cite{AckooijDeOliveira2019}).

\begin{definition} \label{def:EpsilonCriticality}
A feasible point $(x_*, u_*)$ of the problem $(\mathcal{P})$ is called $\varepsilon$-\textit{critical} for this
problem for some $\varepsilon \ge 0$, if there exist a collection of subgradients $V$ of the corresponding convex
functions at $(x_*, u_*)$, defined in Eq.~\eqref{eq:SubgradientCollection}, and a penalty parameter $c_* > 0$ such that
for any $c \ge c_*$ the pair $(x_*, u_*)$ is an $\varepsilon$-optimal solution of the convex problem
\eqref{eq:CriticalProblem}, that is,
\[
  Q_c(x_*, u_*; x_*, u_*; V) \le \inf_{(x, u) \in X_0} Q(x, u; x_*, u_*; V) + \varepsilon.
\] 
A pair $(x_*, u_*) \in X_0$ is called a \textit{generalized $\varepsilon$-critical point} for the problem 
$(\mathcal{P})$ for a given value $c > 0$ of the penalty parameter and some $\varepsilon > 0$, if there exist a
collection of subgradients $V$ of the corresponding convex functions, defined in Eq.~\eqref{eq:SubgradientCollection},
such that the pair $(x_*, u_*)$ is a an $\varepsilon$-optimal solution of the convex problem \eqref{eq:CriticalProblem}.
\end{definition}

Clearly, a (generalised) $\varepsilon$-critical point of the problem $(\mathcal{P})$ with $\varepsilon = 0$ is a 
(generalised) critical point of this problem. Moreover, any \textit{feasible} generalised $\varepsilon$-critical point
of the problem $(\mathcal{P})$ with $\varepsilon \ge 0$ is an $\varepsilon$-critical point of this problem.

\section{Boosted steering exact penalty DCA for nonsmooth optimal control problems}
\label{sect:MethodDescription}

This section is devoted to a detailed description of a class of local search exact penalty DCA-type methods for 
the problem $(\mathcal{P})$. This class of methods can be viewed as an extension of the steering exact penalty
DCA\cite{Strekalovsky2020,Dolgopolik_StExPenDCA} to the case of optimal control problems combined with the boosted
DCA\cite{AragonArtachoFlemingVuong,AragonArtachoVuong,FerreiraSantosSouza} to improve its overall
performance.

\subsection{A description of the method}

The method is motivated by optimality conditions from Proposition~\ref{prp:OptCond} and in its core is based on
consecutively solving convex variational problems of the form
\begin{equation} \label{prob:MainStep}
  \minimise_{(x, u) \in X_0} Q_{c_k}(x, u; x_k, u_k; V_k),
\end{equation}
where the function $Q_c$ is defined as in Eq.~\eqref{eq:ExPenConvexMajorant}, $(x_k, u_k)$ is the current iterate, $c_k$
is the current value of the penalty parameter, while 
\begin{equation} \label{eq:CollectedSubgradients}
  V_k = (V_{k0}(\cdot), V_{k1}(\cdot), \ldots, V_{kn}(\cdot), W_{k1}(\cdot), \ldots, W_{kn}(\cdot), 
  v_{k0}, v_{k1}, \ldots, v_{k \ell_E}, w_{k(\ell_I + 1)}, \ldots, w_{k\ell_E}, z_{k1}(\cdot), \ldots 
  z_{k \ell_m}(\cdot))
\end{equation}
with
\begin{gather} \notag
  V_{kl}(t) \in \partial_{x, u} H_l(x_k(t), u_k(t), t), \quad l \in \{ 0, 1, \ldots, n \}, \quad
  W_{kl}(t) \in \partial_{x, u} G_l(x_k(t), u_k(t), t), \quad l \in \{ 1, \ldots, n \}, 
  \\ \label{eq:Subgradients}
  v_{ki} \in \partial h_i(x_k(0), x_k(T)), \quad i \in \{ 0, 1, \ldots, \ell_E \}, \quad
  w_{kj} \in \partial g_j(x_k(0), x_k(T)), \quad j \in \{ \ell_I + 1, \ldots, \ell_E \}, 
  \\ \notag
  z_{ks}(t) \in \partial_{x, u} q_s(x_k(t), u_k(t), t), \quad s \in \{ 1, \ldots, \ell_M \}
\end{gather}
for a.e. $t \in [0, T]$ and the maps $V_{kl}(\cdot)$, $W_{kl}(\cdot)$, and $z_{ks}(\cdot)$ being measurable. Roughly
speaking, this method can be viewed as the boosted
DCA\cite{AragonArtachoFlemingVuong,AragonArtachoVuong,FerreiraSantosSouza} applied to the penalised problem
\eqref{eq:LocalExactnessDef}, in which one adjusts the penalty parameter $c$ between iterations.

Our rules for updating the penalty parameter $c_k$ largely follow the steering exact penalty methodology originally
developed by Byrd et al.\cite{ByrdNocedalWaltz,ByrdLopezCalvaNocedal} for sequential linear and quadratic programming
methods and further modified to the case of nonsmooth constrained DC optimisation problems in
Refs.\cite{Strekalovsky2020,Dolgopolik_StExPenDCA}. To describe these rules, note that $Q_{c_k}(\cdot, x_k, u_k; V_k)$
is the $L_1$ penalty function for the convex variational problem defined in Remark~\ref{rmrk:ConvexProblem} with 
$(x_*, u_*) = (x_k, v_k)$ and $V = V_k$. If the feasible region of this problem is nonempty, then we would like the
penalty parameter $c_k$ to be large enough to ensure that an optimal solution $(x_k[c_k], u_k[c_k])$ of the convex
problem \eqref{prob:MainStep} is ``almost'' feasible for this problem. 

If the feasible region of this convex problem is empty, then to determine an appropriate value of the penalty parameter
$c_k$, first, one needs to determine the optimal value of infeasibility for this problem by solving the auxiliary convex
optimisation problem
\begin{equation} \label{prob:OptimalFeasibility}
  \minimise_{(x, u) \in X_0} \enspace \Gamma(x, u; x_k, u_k; V_k),
\end{equation}
where
\begin{align} \notag
  \Gamma&(x, u; x_k, u_k; V_k) 
  \\ \notag
  &:=
  \sum_{i = 1}^n \int_0^T \max\Big\{ 
  \begin{aligned}[t]
    &\dot{x}^{(i)}(t) + H_i(x(t), u(t), t) - G_i(x_k(t), u_k(t), t) 
    - \langle W_{ki}(t), (x(t) - x_k(t), u(t) - u_k(t)) \rangle,
    \\
    &- \dot{x}^{(i)}(t) + G_i(x(t), u(t), t) - H_i(x_k(t), u_k(t), t) 
    - \langle V_{ki}(t), (x(t) - x_k(t), u(t) - u_k(t)) \rangle \Big\} \, dt
  \end{aligned}
  \\ \notag
  &+ \sum_{i = 1}^{\ell_I} 
  \max\big\{ 0, g_i(x(0), x(T)) - h_i(x_k(0), x_k(T)) - \langle v_{ki}, (x(0) - x_k(0), x(T) - x_k(T)) \rangle \big\}
  \\ \notag
  &+ \sum_{j = \ell_I + 1}^{\ell_E}
  \begin{aligned}[t]
    \max\big\{ &g_i(x(0), x(T)) - h_i(x_k(0), x_k(T)) - \langle v_{ki}, (x(0) - x_k(0), x(T) - x_k(T)) \rangle,
    \\
    &h_i(x(0), x(T)) - g_i(x_k(0), x_k(T)) - \langle w_{ki}, (x(0) - x_k(0), x(T) - x_k(T)) \rangle \big\}
  \end{aligned}
  \\ \label{def:Gamma}
  &+ \sum_{s = 1}^{\ell_M} \int_0^T \max\Big\{ 0, 
  p_s(x(t), u(t), t) - q_s(x_k(t), u_k(t), t) - \langle z_{ks}(t), (x(t) - x_k(t), u(t) - u_k(t)) \rangle \Big\} \, dt
\end{align}
is an infeasibility measure for the convex problem from Remark~\ref{rmrk:ConvexProblem}. Note also that 
$\Gamma(\cdot; x_k, u_k; V_k)$ is a convex majorant of the penalty term $\varphi$ (see~Eq.~\eqref{eq:PenTerm}), that
is, 
\begin{equation} \label{eq:PenTermConvexMajorant}
  \Gamma(x, u; x_k, u_k; V_k) \ge \varphi(x, u) \quad \forall (x, u) \in X, \quad
  \Gamma(x_k, u_k; x_k, u_k; V_k) = \varphi(x_k, u_k).
\end{equation}
This fact can be verified by applying the inequality from the definition of subgradient to the right-hand side of
equality \eqref{def:Gamma}.

Let $(\widehat{x}_k, \widehat{u}_k)$ be an optimal solution of the optimal feasibility problem
\eqref{prob:OptimalFeasibility} (recall that this problem is convex). 
Then the difference $\Gamma(x_k, u_k; x_k, u_k; V_k) - \Gamma(\widehat{x}_k, \widehat{u}_k; x_k, u_k; V_k)$ represents
the optimal (i.e. maximum possible) decrease of the infeasibility measure at the current iteration. We would like the
penalty parameter $c_k$ to be large enough to ensure that the actual decrease of the infeasibility measure at the
current iteration is proportional to the optimal one, that is,
\[
  \Gamma(x_k[c_k], u_k[c_k]; x_k, u_k; V_k) - \Gamma(x_k, u_k; x_k, u_k; V_k)
  \le \eta_1 \Big( \Gamma(\widehat{x}_k, \widehat{u}_k; x_k, u_k; V_k) - \Gamma(x_k, u_k; x_k, u_k; V_k) \Big)
\] 
for some $\eta_1 \in (0, 1)$. Note, however, that in the case 
$\Gamma(x_k, u_k; x_k, u_k; V_k) = \Gamma(\widehat{x}_k, \widehat{u}_k; x_k, u_k; V_k) > 0$ 
(in this case the constraints are, in some sense, degenerate at the point $(x_k, u_k)$; see
Def.~\ref{def:PenTermCriticality} below) there might not exist finite $c_k > 0$ satisfying the inequality above.
Therefore, one must take into account the case 
$\Gamma(x_k, u_k; x_k, u_k; V_k) = \Gamma(\widehat{x}_k, \widehat{u}_k; x_k, u_k; V_k) > 0$ separately, and determine
$c_k$ in this case in a way that would ensure that the infeasibility measure $\Gamma(x_k[c_k], u_k[c_k]; x_k, u_k; V_k)$
is sufficiently close to the optimal value of the infeasibility measure 
$\Gamma(\widehat{x}_k, \widehat{u}_k; x_k, u_k; V_k)$.

Finally, following the steering exact penalty rules\cite{ByrdNocedalWaltz,ByrdLopezCalvaNocedal,Dolgopolik_StExPenDCA},
we would like to ensure balanced progress towards \textit{both} feasibility and optimality by requiring that 
the improvement of the value of the penalty function $Q_{c_k}(\cdot; x_k, u_k; V_k)$ is large, when the improvement of
the infeasibility measure is large. More precisely, we would like to ensure that the corresponding improvements are
proportional, that is,
\[
  Q_{c_k}(x_k[c_k], u_k[c_k]; x_k, u_k; V_k) - Q_{c_k}(x_k, u_k; x_k, u_k; V_k)
  \le c_k \eta_2 \Big( \Gamma(x_k[c_k], u_k[c_k]; x_k, u_k; V_k) - \Gamma(x_k, u_k; x_k, u_k; V_k) \Big)
\]
for some $\eta_2 \in (0, 1)$.

After a value $c_k$ of the penalty parameter, satisfying the conditions discussed above, is computed along with the
corresponding point $(x_k[c_k], u_k[c_k])$, we employ a nonmonotone line search along the direction 
$(x_k[c_k] - x_k, u_k[c_k] - u_k)$ as in the boosted DCA for nonsmooth DC functions\cite{FerreiraSantosSouza} to further
improve the current point $(x_k[c_k], u_k[c_k])$ and accelerate convergence of the sequence constructed by the method.
The usage of \textit{nonmonotone} line search is dictated by the fact that the direction 
$(x_k[c_k] - x_k, u_k[c_k] - u_k)$ might not be a descent direction of the penalty function $\Phi_c$, since this penalty
function is always nonsmooth, even if the original problem is smooth (see 
Refs.\cite{AragonArtachoVuong,FerreiraSantosSouza}).

Let us note that in practice convex subproblems \eqref{prob:MainStep} and \eqref{prob:OptimalFeasibility} can be solved
only approximately due to the finite precision of computations, effects of discretisation, etc. Therefore, for the
theoretical scheme of the method to closer resemble its practical implementation, we will assume that 
$(x_k[c_k], u_k[c_k])$ and $(\widehat{x}_k, \widehat{u}_k)$ are not optimal, but $\varepsilon_k$-optimal solutions 
of the corresponding convex problems for some $\varepsilon_k > 0$. This assumption has an additional benefit of
resolving the problem of the \textit{existence} of corresponding optimal solutions.  

Thus, we arrive at the following scheme of the boosted steering exact penalty DCA (B-STEP-DCA) for the optimal control
problem $(\mathcal{P})$ given in Algorithmic Pattern~\ref{alg:Boosted_SEP_DCA}. Following van Ackooij and 
de Oliveira\cite{AckooijDeOliveira2019}, we use the term \textit{algorithmic pattern}, since
B-STEP-DCA is not a local search method the problem $(\mathcal{P})$ per se, but rather a pattern
that can be used to define a whole family of local search methods for this problem. By specifying methods 
(in particular, discretisation schemes) for solving auxiliary convex subproblems, rules for increasing the penalty
parameter $c_+$ and $c_{k + 1}$ on Steps~2-4, a rule for choosing line search tolerances $\nu_k$, a rule for choosing
trial step size $\overline{\alpha}_k > 0$, and, finally, stopping criteria, one can define a particular local search
method for the problem $(\mathcal{P})$ that follows the pattern described by the B-STEP-DCA.

\begin{algorithm}
\caption{Boosted steering exact penalty DCA (B-STEP-DCA) for nonsmooth optimal control problems}

\textbf{Initialization.} Choose an initial guess $(x_0, u_0) \in X_0$ and an initial value of the penalty parameter
$c_0 > 0$. Choose parameters $\eta_1, \eta_2, \zeta, \sigma \in (0, 1)$, infeasibility tolerances 
$\varepsilon_{\varphi}, \varepsilon_{feas} > 0$, and optimality tolerances $\{ \varepsilon_k \} \subset (0, + \infty)$.
Set $k = 0$.

\textbf{Step 1.} Put $c_+ = c_k$. For all $l \in \{ 0, 1, \ldots, n \}$, 
$i \in \mathcal{I} \cup \mathcal{E} \cup \{ 0 \}$, $j \in \mathcal{J}$, $s \in \mathcal{M}$, compute subgradients
\eqref{eq:Subgradients} and denote by $V_k$ the vector from Eq.~\eqref{eq:CollectedSubgradients}. Compute an
$\varepsilon_k$-optimal solution $(x_k[c_+], u_k[c_+])$ of the convex penalised subproblem
\begin{equation} \label{subprob:MainStep}
  \minimise_{(x, u) \in X_0} \enspace Q_c(x, u; x_k, u_k; V_k)
\end{equation}
with $c = c_+$. If $\Gamma(x_k[c_+], u_k[c_+]; x_k, u_k; V_k) \le \varepsilon_{\varphi}$ (that is, 
$(x_k[c_+], u_k[c_+])$ is approximately feasible for the corresponding convexified problem and the problem
$(\mathcal{P})$), go to \textbf{Step 4}. Otherwise, go to \textbf{Step 2}.

\textbf{Step 2.} Compute an $\varepsilon_k$-optimal solution $(\widehat{x}_k, \widehat{u}_k)$ of the optimal
infeasibility subproblem
\begin{equation} \label{subprob:OptInfeasMeas}
  \minimise_{(x, u) \in X_0} \enspace \Gamma(x, u; x_k, u_k; V_k).
\end{equation}
If $\Gamma(\widehat{x}_k, \widehat{u}_k; x_k, u_k; V_k) < \Gamma(x_k, u_k; x_k, u_k; V_k)$, go to \textbf{Step 3}.
Otherwise, $(x_k, u_k)$ is an $\varepsilon_k$-critical point of the penalty term $\varphi$. \textbf{While} the
inequality \begin{equation} \label{eq:NearOptimalInfeasMeas}
  \Gamma(x_k[c_+], u_k[c_+]; x_k, u_k; V_k) \le \Gamma(x_k, u_k; x_k, u_k; V_k) + \varepsilon_{feas}
\end{equation}
is \textbf{not} satisfied, increase $c_+$ and compute an $\varepsilon_k$-optimal solution $(x_k[c_+], u_k[c_+])$ of
problem \eqref{subprob:MainStep} with $c = c_+$. Once inequality \eqref{eq:NearOptimalInfeasMeas} is satisfied, go to
Step~4.

\textbf{Step 3.} \textbf{While} inequality
\begin{equation} \label{eq:InfeasMeasDecay}
  \Gamma(x_k[c_+], u_k[c_+]; x_k, u_k; V_k) - \Gamma(x_k, u_k; x_k, u_k; V_k)
  \le \eta_1 \Big( \Gamma(\widehat{x}_k, \widehat{u}_k; x_k, u_k; V_k) - \Gamma(x_k, u_k; x_k, u_k; V_k) \Big)
\end{equation}
is \textbf{not} satisfied, increase $c_+$ and compute an $\varepsilon_k$-optimal solution $(x_k[c_+], u_k[c_+])$ of
problem \eqref{subprob:MainStep} with $c = c_+$. Once inequality \eqref{eq:InfeasMeasDecay} is satisfied, go to
\textbf{Step~4}.

\textbf{Step 4.} Put $c_{k + 1} = c_+$. \textbf{While} inequality
\begin{multline} \label{eq:PenFuncDecay}
  Q_{c_{k + 1}}(x_k[c_{k + 1}], u_k[c_{k + 1}]; x_k, u_k; V_k) - Q_{c_{k + 1}}(x_k, u_k; x_k, u_k; V_k)
  \\
  \le c_{k + 1} \eta_2 
  \Big( \Gamma(x_k[c_{k + 1}], u_k[c_{k + 1}]; x_k, u_k; V_k) - \Gamma(x_k, u_k; x_k, u_k; V_k) \Big)
\end{multline}
is \textbf{not} satisfied, increase $c_{k + 1}$ and compute an $\varepsilon_k$-optimal solution 
$(x_k[c_{k + 1}], u_k[c_{k + 1}])$ of problem \eqref{subprob:MainStep} with $c = c_{k + 1}$. Once inequality
\eqref{eq:PenFuncDecay} is satisfied, go to Step~5.

\textbf{Step 5.} Choose line search tolerance $\nu_k > 0$ and trial step size $\overline{\alpha}_k \ge 0$. Find minimal
$j_k \in \mathbb{N}$ satisfying the inequality
\begin{multline} \label{eq:LineSearchIneq}
  \Phi_{c_{k + 1}}\Big( x_k[c_{k + 1}] + \zeta^{j_k} \overline{\alpha_k} \big( x_k[c_{k + 1}] - x_k \big), 
  u_k[c_{k + 1}] + \zeta^{j_k} \overline{\alpha_k} \big( u_k[c_{k + 1}] - u_k \big) \Big) 
  - \Phi_{c_{k + 1}}(x_k[c_{k + 1}], u_k[c_{k + 1}]) 
  \\
  \le - \sigma \big( \zeta^{j_k} \overline{\alpha_k} \big)^2 
  \| (x_k[c_{k + 1}] - x_k, u_k[c_{k + 1}] - u_k) \|_2^2 + \nu_k
\end{multline}
and define $\alpha_k = \zeta^{j_k} \overline{\alpha}_k$. Set 
$(x_{k + 1}, u_{k + 1}) = (x_k[c_{k + 1}], u_k[c_{k + 1}]) + \alpha_k (x_k[c_{k + 1}] - x_k, u_k[c_{k + 1}] - u_k)$, 
and check a \textbf{stopping criterion} for $(x_{k + 1}, u_{k + 1})$. If it is satisfied, \textbf{Stop}. Otherwise,
put $k \leftarrow k + 1$, and go to \textbf{Step~1}. \label{alg:Boosted_SEP_DCA}.
\end{algorithm}

Let us briefly discuss the way B-STEP-DCA works. On each iteration of this algorithmic pattern one solves the auxiliary
penalised convex subproblem \eqref{subprob:MainStep}. If a computed $\varepsilon_k$-optimal solution of this problem
satisfies the corresponding constraints, then the method jumps to Step~4 and checks inequality
\eqref{eq:InfeasMeasDecay}, ensuring sufficient decay of the penalty function. If the inequality is satisfied, then the
current value of the penalty parameter $c_k$ is adequate, the method performs nonmonotone line search on Step~5 and
moves to the next iteration.

If an $\varepsilon_k$-optimal solution computed on Step~1 is infeasible, then one solves auxiliary convex subproblem
\eqref{subprob:OptInfeasMeas} to determine the optimal level of infeasibility on the current iteration and then,
potentially, solves subproblem \eqref{subprob:MainStep} multiple times for increasing values of the penalty parameter
to find a value that would guarantee sufficient decrease of both infeasibility measure and the value of the penalty
function. After an adequate value of the penalty parameter is found, the method performs nonmonotone line search on
Step~5 and moves to the next iteration.

Thus, on each iteration of Algorithmic Pattern~\ref{alg:Boosted_SEP_DCA} one has to potentially solve several convex
optimisation subproblems to ensure the validity of all corresponding conditions. As was clearly demonstrated by multiple
numerical experiments in Refs.\cite{ByrdNocedalWaltz,ByrdLopezCalvaNocedal,Dolgopolik_StExPenDCA}, despite seeming
computationally expensive and inefficient, the use of such strategy completely pays off in the end, since it leads to
accelerated convergence and improved robustness of the corresponding exact penalty methods.

For the sake of convenience, below we impose the following \textit{nonrestrictive} assumption on the point 
$(x_k[c_+], u_k[c_+])$ and $(x_k[c_{k + 1}], u_k[c_{k + 1}])$ {subprob:OptInfeasMeas}computed on Steps~1--4 of
Algorithmic
Pattern~\ref{alg:Boosted_SEP_DCA} that simplifies convergence analysis and allows one to strengthen some theoretical
results.

\begin{assumption} \label{assumpt:SubOptSolCorrectness}
For any $k \in \mathbb{N}$ the point $(x_k[c_+], u_k[c_+])$ on Steps~1--3 of Algorithmic
Pattern~\ref{alg:Boosted_SEP_DCA} satisfies the inequality
\begin{equation} \label{eq:SubOptSolCorrectness}
  Q_c(x_k[c], u_k[c]; x_k, u_k; V_k) \le Q_c(x_k, u_k; x_k, u_k; V_k)
\end{equation}
with $c = c_+$, while the point $(x_k[c_{k + 1}], u_k[c_{k + 1}])$ on Step~4 of 
Algorithmic Pattern~\ref{alg:Boosted_SEP_DCA} satisfies inequality \eqref{eq:SubOptSolCorrectness} with 
$c = c_{k + 1}$.
\end{assumption}

Let us underline that this assumption is not restrictive. If a convex optimisation method that one uses to solve
subproblem~\eqref{subprob:MainStep} fails to improve the current value $Q_{c_+}(x_k, u_k; x_k, u_k; V_k)$ of the convex
majorant $Q_{c_+}(\cdot; x_k, u_k; V_k)$ of the penalty function $\Phi_{c_+}(\cdot)$ for the problem $(\mathcal{P})$,
then by definitions the point $(x_k, u_k)$ is a $\varepsilon_k$-optimal solution of this problem and, therefore, an
$\varepsilon_k$-critical point of the problem $(\mathcal{P})$. In this case it is natural to terminate the algorithmic
pattern, since it has found an approximately critical problem of the problem under consideration. Thus, after solving
subproblem~\eqref{subprob:MainStep} on Steps~1--4 of Algorithmic Pattern~\ref{alg:Boosted_SEP_DCA} one should
check inequality \eqref{eq:SubOptSolCorrectness}. If it is no satisfied, then the method must terminate and return
the point $(x_k, u_k)$.

We also impose the following assumption on the trial step size $\overline{\alpha}_k \ge 0$ on Step~5 of Algorithmic
Pattern~\ref{alg:Boosted_SEP_DCA} that allows one to correctly unify several different versions of B-STEP-DCA for 
the problem $(\mathcal{P})$ into one theoretical scheme.

\begin{assumption} \label{assumpt:TrialStepSize}
The trial step size $\overline{\alpha}_k \ge 0$ on Step~5 of Algorithmic Pattern~\ref{alg:Boosted_SEP_DCA} satisfies
the condition
\[
  \Big( x_k[c_{k + 1}] + \zeta^j \overline{\alpha_k} \big( x_k[c_{k + 1}] - x_k \big), 
  u_k[c_{k + 1}] + \zeta^j \overline{\alpha_k} \big( u_k[c_{k + 1}] - u_k \Big) \in X_0
  \quad \forall j \in \mathbb{N},
\]
which implies that $(x_{k + 1}, u_{k + 1}) \in X_0$.
\end{assumption}

If $X_0 = X$ or $X_0$ is an affine subspace of $X$, then any $\overline{\alpha}_k \ge 0$ satisfies
Assumption~\ref{assumpt:TrialStepSize}. However, if $X_0$ is a non-affine convex set (that might have a complicated
structure), then the only safe choice is $\overline{\alpha}_k = 0$. Such choice means that Algorithmic
Pattern~\ref{alg:Boosted_SEP_DCA} does not perform line search (Step~5), and in this case it is natural to call the
corresponding method simply STEP-DCA (cf.~Refs.\cite{Strekalovsky2020,Dolgopolik_StExPenDCA}).
Assumption~\ref{assumpt:TrialStepSize} guarantees that one can simultaneously consider Algorithmic
Pattern~\ref{alg:Boosted_SEP_DCA} with and without line search and present a unified convergence analysis for both
versions of the method.

\begin{remark} \label{rmrk:BSEP_DCA_DifferentVersions}
{(i)~The rather general formulations of the problem $(\mathcal{P})$ and Algorithmic Pattern~\ref{alg:Boosted_SEP_DCA}
(see also Remark~\ref{rmrk:PenaltyTerms}) allow one to flexibly adjust B-STEP-DCA to any given optimal control problem
to improve its overall performance. In particular, they allow one to choose whether to use line search or not. For
example, if the end-points are fixed, that is, $x(0) = x^0$ and $x(T) = x^T$ for some given $x^0, x^T \in \mathbb{R}^n$,
then it is natural to not include the end-point constraints into the definition of the problem $(\mathcal{P})$
explicitly and instead put $X_0 = \{ (x, u) \in X \mid x(0) = x^0, \: x(T) = x^T \}$. Similarly, if the system's
dynamics is linear and has the form $\dot{x}(t) = A(t) x(t) + B(t) u(t)$, then one can define
\[
  X_0 = \Big\{ (x, u) \in X \Bigm| \dot{x}(t) = A(t) x(t) + B(t) u(t) \text{ for a.e. } t \in [0, T] \Big\}
\]
and not include the corresponding terms into the penalty functions $\Phi_c$ and $Q_c$ (for such choice of $X_0$ these
terms are identically zero on $X_0$). In this case, the convex subproblems \eqref{subprob:MainStep} and
\eqref{subprob:OptInfeasMeas} become \textit{convex} optimal control problems that can be solved with the use of
standard methods. Note that for both choices of $X_0$ mentioned above Assumption~\ref{assumpt:TrialStepSize} is
satisfied for any $\overline{\alpha}_k \ge 0$ and one can safely perform line search on Step~5 of Algorithmic
Pattern~\ref{alg:Boosted_SEP_DCA}.

In the case when some of the end-point constraints or state constraints are convex, one can choose to either (a) include
them into the set $X_0$ and avoid line search or (b) include them into the formulation of the problem $(\mathcal{P})$
explicitly to allow the method to perform line search in order to potentially accelerate its convergence (see numerical
examples in Section~\ref{sect:NumericalExperiments}).
}

\noindent{(ii)~Let us note that Algorithmic Pattern~\ref{alg:Boosted_SEP_DCA} allows a straighforward extension to the
case of optimal control problem with isoperimetric equality and inequality constraints of the form
\[
  \int_0^T \Psi_i(x(t), u(t), t) \, dt + f_i(x(0), x(T)) \le 0, \quad
  \int_0^T \Psi_j(x(t), u(t), t) \, dt + f_j(x(0), x(T)) = 0.
\]
One simply needs to add the corresponding terms to the penalty functions $\Phi_c$ and $Q_c$, and the infeasibility
measure $\Gamma$. Although such constraints can be easily converted into equivalent end-points constraints by
introducing additional variables $\dot{y}_i(t) = \Psi_i(x(t), u(t), t)$, it seems more reasonable to take them into
account directly, since this way one does not need to increase the dimension of the problem and the terms 
$|\int_0^T \Psi_j(x(t), u(t), t) \, dt + f_j(x(0), x(T))|$ are easier to deal with both theoretically and numerically
than $\int_0^T |\dot{y}_j(t) - \Psi_j(x(t), u(t), t)| \, dt$. Moroever, Algorithmic Pattern~\ref{alg:Boosted_SEP_DCA}
can also be extended to the case of problems with pure state and mixed state-control \textit{equality} constraints of
the form $\Xi_s(x(t), u(t), t) = 0$ by including the terms $\int_0^T |\Xi_s(x(t), u(t), t)| \, dt$ into the penalty
term $\varphi$ and including the corresponding convex majorants of these terms into the convex majorant $Q_c$ of the
penalty function $\Phi_c$ for the problem $(\mathcal{P})$. We omit a detailed description of the fairly obvious
extension of Algorithmic Pattern~\ref{alg:Boosted_SEP_DCA} to the case of problems with isoperimetric constraints and
mixed equality constraints for the sake of shortness, and only note here that \textit{all} theoretical results presented
below hold true for this extension.
}
\end{remark}

\begin{remark} \label{rmrk:MaxPenParamSafeguard}
Although theoretical analysis of Algorithmic Pattern~\ref{alg:Boosted_SEP_DCA} presented below is valid in a fairly
general setting, in practice the algorithmic pattern might fail to converge or find a critical point of the problem
$(\mathcal{P})$ in the case when this problem is in some sense degenerate. Therefore, a practical implementation of
B-STEP-DCA must include certain safeguards to ensure its proper behaviour in degenerate cases. In particular, one must
set a maximal value of the penalty parameter $c_{max} > 0$. If this value is reached and the algorithmic pattern fails
to find a point satisfying a termination criterion after a fixed number of iterations, then the computations must be
stopped with a warning message. In such case one can try restarting the method from a different starting point or
consider using a different method altogether (or at least a different discretisation scheme/method for auxiliary convex
subproblems), since B-STEP-DCA might be unsuitable for solving the problem due to some kind of degeneracy.
\end{remark}

\subsection{Stopping criteria}
\label{subsect:StoppingCriteria}

The following inequalities with sufficiently small $\varepsilon_f > 0$, $\varepsilon_x > 0$, and 
$\varepsilon_{\varphi} > 0$ can be used as stopping criteria for B-STEP-DCA:
\begin{equation} \label{eq:StoppingCriterion1}
  \Big| \Phi_{c_{k + 1}}(x_{k + 1}, u_{k + 1}) - \Phi_{c_{k + 1}}(x_k, u_k) \Big| < \varepsilon_f \quad 
  \Big( \text{and/or } \| (x_k[c_{k + 1}] - x_k, u_k[c_{k + 1}] - u_k) \|_2 < \varepsilon_x \Big), \quad 
  \varphi(x_{k + 1}, u_{k + 1}) < \varepsilon_{\varphi}
\end{equation}
(note that $\varepsilon_{\varphi} > 0$ is a parameter from Step~1 of Algorithmic Pattern~\ref{alg:Boosted_SEP_DCA}).
A theoretical justification for the stopping criteria \eqref{eq:StoppingCriterion1} is provided in
Theorems~\ref{thrm:StoppingCriteria1} and \ref{thrm:StoppingCriteria2}.

Let us note that to avoid potentially expensive computations of the values 
$\Phi_{c_{k + 1}}(x_{k + 1}, u_{k + 1})$, $\Phi_{c_{k + 1}}(x_k, u_k)$, and $\varphi(x_{k + 1}, u_{k + 1})$, one can use
similar inequalities
\begin{equation} \label{eq:StoppingCriterion2}
  Q_{c_{k + 1}}(x_k, u_k; x_k, u_k; V_k) - Q_{c_{k + 1}}(x_k[c_{k + 1}], u_k[c_{k + 1}]; x_k, u_k; V_k) 
  < \varepsilon_f, \quad
  \Gamma(x_k[c_{k + 1}], u_k[c_{k + 1}]; x_k, u_k; V_k) < \varepsilon_{\varphi}
\end{equation}
involving values that are always computed on each iteration of Algorithmic Pattern~\ref{alg:Boosted_SEP_DCA} as a
stopping criterion.\footnote{The author is sincerely grateful to prof. A.S. Strekalovsky for pointing out this fact to
him.} The fact that the first inequality is, roughly speaking, equivalent to the corresponding inequality for 
the penalty function $\Phi_{c_{k + 1}}$ is shown in Corollary~\ref{crlr:StoppingCriteria}. Here we also note that the
validity of the inequality for $Q_{c_{k + 1}}$ by definition means that $(x_k, u_k)$ is a generalised
$(\varepsilon_f + \varepsilon_k)$-critical point of the problem $(\mathcal{P})$. 

The equivalence of the corresponding inequalities for the penalty terms follows from the fact that
\[
  \Gamma(x_k[c_{k + 1}], u_k[c_{k + 1}]; x_k, u_k; V_k) \ge \varphi(x_k[c_{k + 1}], u_k[c_{k + 1}]), \quad
  \Gamma(x_k, u_k; x_k, u_k; V_k) = \varphi(x_k, u_k),
\]
since $\Gamma(\cdot; x_k, u_k; V_k)$ is a global convex majorant of $\varphi(\cdot)$.

\subsection{Strategies for choosing line search tolerances $\nu_k$}

As the convergence analysis presented below reveals, the main condition that the line search tolerances $\nu_k$ must
satisfy to ensure the overall convergence of the method is summability:
\begin{equation} \label{eq:NuKSummable}
  \sum_{k = 0}^{\infty} \nu_k < + \infty. 
\end{equation}
The validity of this condition can be ensured in several different ways. Namely, the following strategies for choosing
$\nu_k$ can be used (see Ref.~\cite{FerreiraSantosSouza}):
\begin{itemize}
\item[\textbf{(S1)}]{A priori strategy: choose any sequence $\{ \nu_k \} \subset (0, + \infty)$ such that
$\sum_{k = 0}^{\infty} \nu_k < + \infty$.}

\item[\textbf{(S2)}]{Strategy based on function values: Fix any $\delta_{\min} \in (0, 1)$ and $\nu_0 > 0$. For any 
$k \in \mathbb{N}$ choose any $\delta_{k + 1} \in [\delta_{\min}, 1)$ and $\nu_{k + 1}$ satisfying the inequalities
\begin{equation} \label{eq:S2Strategy}
  0 < \nu_{k + 1} \le (1 - \delta_{k + 1}) 
  \Big( \Phi_{c_{k + 1}}(x_k, u_k) - \Phi_{c_{k + 1}}(x_{k + 1}, u_{k + 1}) + \nu_k \Big).
\end{equation}
}

\vspace{-5mm}

\item[\textbf{(S3)}]{Strategy based on the sequence of iterates: choose any sequence $\{ \nu_k \} \subset (0, + \infty)$
such that for every $\delta > 0$ one can find $k_0 \in \mathbb{N}$ such that 
$\nu_k \le \delta \| (x_k[c_{k + 1}] - x_k, u_k[c_{k + 1}] - u_k) \|_2^2$ for all $k \ge k_0$.}
\end{itemize}
Note that in the case of strategy $(S3)$ one can define 
$\nu_k = \gamma_k \| (x_k[c_{k + 1}] - x_k, u_k[c_{k + 1}] - u_k) \|_2^2$ for any sequence 
$\{ \gamma_k \} \subset (0, + \infty)$ such that $\gamma_k \to 0$ as $k \to \infty$. Let us also note that if 
$\| (x_k[c_{k + 1}] - x_k, u_k[c_{k + 1}] - u_k) \|_2 > 0$ (that is, if $(x_k, u_k)$ is not critical for the problem
$(\mathcal{P})$), then one can always find $\nu_{k + 1}$ satisfying inequalities \eqref{eq:S2Strategy} (see
Crlr.~\ref{crlr:S2Strategy} below).

In Subsection~\ref{subsect:ElementaryAnalysis} we will show that both strategies $(S2)$ and $(S3)$ satisfy condition
\eqref{eq:NuKSummable} under some additional assumptions. For a more detailed discussion of various strategies
for choosing $\nu_k$ see Ref.\cite{FerreiraSantosSouza}.

\subsection{Strategies for choosing trial step lengths $\overline{\alpha}_k$}
\label{subsect:TrialStepSize}

The following strategies for choosing trial step length $\overline{\alpha}_k$ can be used:
\begin{enumerate}
\item{Constant trial step size: $\overline{\alpha}_k = \overline{\alpha} > 0$ for all $k \in \mathbb{N}$ and some
$\overline{\alpha} > 0$.}

\item{Step size from the previous iteration: $\overline{\alpha}_k = \alpha_{k - 1}$ for all $k \in \mathbb{N}$.}

\item{Self-adaptive trial step size: for any $k \ge 2$, if $\overline{\alpha}_{k - 2} = \alpha_{k - 2}$ and
$\overline{\alpha}_{k - 1} = \alpha_{k - 1}$, then put $\overline{\alpha}_k = \gamma \alpha_{k - 1}$ for some fixed
$\gamma > 1$; else set $\overline{\alpha}_k = \alpha_{k - 1}$.}
\end{enumerate}
The first strategy was proposed by Arag\'{o}n Artacho et al.\cite{AragonArtachoFlemingVuong}. The second strategy was
considered in Ferreira et al.\cite{FerreiraSantosSouza}, while the third one was verified numerically by
Arag\'{o}n Artacho and Vuong\cite{AragonArtachoVuong}. The main benefit of the self-adaptive strategy consists in
the fact that it is the only strategy of the three that can increase the trial step size $\overline{\alpha}_k$ to allow
the method to make bigger steps and, as a result, reach critical points faster. It was shown in
Ref.~\cite{AragonArtachoVuong} that on average the boosted DCA with self-adaptive strategy was $1.7$ times
faster than the boosted DCA with the constant trial step size.

Let us note that the convergence analysis of Algorithmic Pattern~\ref{alg:Boosted_SEP_DCA} presented below is valid for
an \textit{arbitrary} choice of trial step lengths $\overline{\alpha_k} \ge 0$, 
including $\overline{\alpha_k} \equiv 0$.

\section{Correctness of the method}
\label{sect:Correctness}

We start our analysis of Algorithmic Pattern~\ref{alg:Boosted_SEP_DCA} by showing that this algorithmic pattern is
correctly defined. To this end, we need to show that (i) the optimal value of the convex subproblem
\eqref{subprob:MainStep} is finite, so that an $\varepsilon$-solution of this problem with $\varepsilon > 0$ always
exists, (ii) inequalities \eqref{eq:NearOptimalInfeasMeas}, \eqref{eq:InfeasMeasDecay}, and \eqref{eq:PenFuncDecay} can
always be satisfied by sufficiently increasing the penalty parameter, and (iii) inequality \eqref{eq:LineSearchIneq}
is satisfied for some $j_k \in \mathbb{N}$ for any choice of nonmonotone line search tolerance $\nu_k > 0$ and trial
step size $\overline{\alpha}_k \ge 0$. Then one can conclude that each step of Algorithmic
Pattern~\ref{alg:Boosted_SEP_DCA} is correctly defined and to perform one iteration of this algorithmic pattern one
needs to approximately solve only a finite number of convex subproblems~\eqref{subprob:MainStep}.

Hereinafter we suppose that the following assumption ensuring the correctness of 
Algorithmic Pattern~\ref{alg:Boosted_SEP_DCA} holds true.

\begin{assumption} \label{assumpt:PenaltyFuncBoundedBelow}
The penalty function $\Phi_c(x, u)$ with $c = c_0$ is bounded below on the set $X_0$. 
\end{assumption}

With the use of this assumption we can immediately prove the correctness of the definition of $(x_k[c], u_k[c])$ for any
$c \ge c_0$. Note that the point $(\widehat{x}_k, \widehat{u}_k)$ is always correctly defined, since the function
$\Gamma(\cdot; x_k, u_k; V_k)$ is nonnegative by definition.

\begin{proposition}
For any $k \in \mathbb{N}$ and $c \ge c_0$ the optimal value of problem \eqref{subprob:MainStep} is finite.
\end{proposition}

\begin{proof}
Let $k = 0$. Taking into account the definition $Q_c$ (see Eqs.~\eqref{eq:ExPenConvexMajorant} and
\eqref{eq:GlobalConvexMajorant}) and applying the definition of subgradient one can easily check that
\begin{align*}
  Q_c(x, u; x_k, u_k; V_k) \ge \Phi_c(x, u)
  &+ \int_0^T \Big( H_0(x_k(t), u_k(t), t) - \langle V_{k0}(t), (x_k(t), u_k(t)) \rangle \Big) \, dt
  \\
  &+ h_0(x_k(0), x_k(T)) + \langle v_{k0}, (x_k(0), x_k(T)) \rangle
\end{align*}
for any $(x, u) \in X$ and $c > 0$. Hence bearing in mind Assumption~\ref{assumpt:PenaltyFuncBoundedBelow} and the fact
that the penalty function $\Phi_c$ is obviously nondecreasing in $c$ one can conclude that for any $c \ge c_0$ the
function $Q_c(\cdot; x_k, u_k; V_k)$ is bounded below on the set $X_0$. Therefore, the optimal value of problem
\eqref{subprob:MainStep} is finite for $k = 0$ and any $c \ge c_0$. Now, arguing by induction one can easily show that
the same statement holds true for any $k \in \mathbb{N}$. 
\end{proof}

Let us now prove a useful auxiliary result on behaviour of the function 
$c \mapsto \Gamma(x_k[c], u_k[c]; x_k, u_k; V_k)$ that is crucial for the proof of correctness of Algorithmic
Pattern~\ref{alg:Boosted_SEP_DCA}. For the sake of convenience, for any $k \in \mathbb{N}$ denote by
\[
  \omega_k(x, u) = \int_0^T \Big( G_0(x(t), u(t), t) - \langle V_{k0}(t), (x(t), u(t)) \rangle \Big) \, dt 
  + g_0(x(0), x(T)) - \langle v_{k0}, (x(0), x(T)) \rangle
\]
the convex majorant (up to a constant) of the cost functional $J(x, u)$ (see \eqref{eq:ExPenConvexMajorant}).

\begin{lemma} \label{lem:InfeasMeasBehaviour}
Let $(x_k[c], u_k[c])$ be an $\varepsilon_k$-optimal solution of problem \eqref{subprob:MainStep}. Then for any 
$k \in \mathbb{N}$ the following statements hold true:
\begin{enumerate}
\item{for any $t > c > 0$ one has $\Gamma(x_k[t], u_k[t]; x_k, u_k; V_k) \le \Gamma(x_k[c], u_k[c]; x_k, u_k; V_k) + 2
\varepsilon_k / (t - c)$;
}

\item{for any $t > c > 0$ one has $\omega_k(x_k[t], u_k[t]) \ge \omega_k(x_k[c], u_k[c]) - 2 c \varepsilon_k / (t - c) -
\varepsilon_k$;
}

\item{$\Gamma(x_k[c], u_k[c]; x_k, u_k; V_k)$ converges to the optimal value of problem \eqref{subprob:OptInfeasMeas} as
$c \to + \infty$.
}
\end{enumerate}
\end{lemma}

\begin{proof}
Fix any $k \in \mathbb{N}$ and $t > c > 0$. Observe that 
$Q_c(x_k[c], u_k[c]; x_k, u_k; V_k) \le Q_c(x_k[t], u_k[t]; x_k, u_k; V_k) + \varepsilon_k$ 
by the definition $(x_k[c], u_k[c])$, which yields
\begin{equation} \label{eq:IncreasedPenParamDiff}
  \omega_k(x_k[c], u_k[c]) - \omega_k(x_k[t], u_k[t]) 
  \le c \Big( \Gamma(x_k[t], u_k[t]; x_k, u_k; V_k) - \Gamma(x_k[c], u_k[c]; x_k, u_k; V_k) \Big) + \varepsilon_k.
\end{equation}
In turn, by the definition of $(x_k[t], u_k[t])$ one has
\[
  \omega_k(x_k[t], u_k[t]) - \omega_k(x_k[c], u_k[c]) 
  \le t \Big( \Gamma(x_k[c], u_k[c]; x_k, u_k; V_k) - \Gamma(x_k[t], u_k[t]; x_k, u_k; V_k) \Big) + \varepsilon_k.
\]
By summing up these two inequalities one obtains
\[
  (t - c) \Big( \Gamma(x_k[c], u_k[c]; x_k, u_k; V_k) - \Gamma(x_k[t], u_k[t]; x_k, u_k; V_k) \Big) + 2
\varepsilon_k \ge 0,
\]
which implies that 
\[
  \Gamma(x_k[c], u_k[c]; x_k, u_k; V_k) - \Gamma(x_k[t], u_k[t]; x_k, u_k; V_k) 
  \ge -\frac{2 \varepsilon_k}{t - c},
\]
that is, the first statement of the lemma holds true. Moreover, combining the inequality above with 
Ineq.~\eqref{eq:IncreasedPenParamDiff} one gets that
\[
  \omega_k(x_k[c], u_k[c]) \le \omega_k(x_k[t], u_k[t]) + \frac{2 c \varepsilon_k}{t - c} + \varepsilon_k.
\]
Thus, the second statement of the lemma is proved as well.

Let us prove the third statement. Arguing by reductio ad absurdum suppose that this statement is false. Then there exist
an increasing unbounded sequence of penalty parameters $\{ \xi_l \} \subset (c_0; + \infty)$ and $\delta > 0$ such that
\[
  \Gamma(x_k[\xi_l], u_k[\xi_l]; x_k, u_k; V_k) \ge 
  \inf_{(x, u) \in X_0} \Gamma(x, u; x_k, u_k; V_k) + \delta
  \quad \forall l \in \mathbb{N}.
\]
Let $(x'_k, u'_k) \in X_0$ be a $(\delta/2)$-optimal solution of problem \eqref{subprob:OptInfeasMeas}. Then
\begin{equation} \label{eq:InfMeasNonConvergToOptimV}
  \Gamma(x_k[\xi_l], u_k[\xi_l]; x_k, u_k; V_k) \ge 
  \Gamma(x'_k, u'_k; x_k, u_k; V_k) + \frac{\delta}{2}
  \quad \forall l \in \mathbb{N}.
\end{equation} 
By the definition of $(x_k[c], u_k[c])$ one has 
$\varepsilon_k + Q_c(x'_k, u'_k; x_k, u_k; V_k) \ge Q_c(x_k[c], u_k[c]; x_k, u_k; V_k)$ for any $c \ge c_0$, which with
the use of the second statement of the lemma and inequality \eqref{eq:InfMeasNonConvergToOptimV} implies that
\begin{align*}
  \varepsilon_k + \omega_k(x'_k, u'_k) &\ge \omega_k(x_k[\xi_l], u_k[\xi_l]) + \xi_l 
  \Big( \Gamma(x_k[\xi_l], u_k[\xi_l]; x_k, u_k; V_k) - \Gamma(x'_k, u'_k; x_k, u_k; V_k) \Big) 
  \\
  &\ge \omega_k(x_k[c_0], u_k[c_0]) - \frac{2 c_0 \varepsilon_k}{\xi_l - c_0} - \varepsilon_k + \xi_l \frac{\delta}{2}
\end{align*}
for any $l \in \mathbb{N}$. Note, however, that the right-hand side of this inequality increases unboundedly as 
$l \to + \infty$, since the sequence $\xi_l$ increases unboundedly, which leads to an obvious contradiction.
\end{proof}

With the use of the lemma above we can show that inequalities \eqref{eq:NearOptimalInfeasMeas},
\eqref{eq:InfeasMeasDecay}, and \eqref{eq:PenFuncDecay} on Steps 2--4 of Algorithmic Pattern~\ref{alg:Boosted_SEP_DCA}
hold true for any sufficiently large value of the penalty parameter, and, therefore, on each iteration of Algorithmic
Pattern~\ref{alg:Boosted_SEP_DCA} one has to solve only a finite number of convex subproblems \eqref{prob:MainStep}.

\begin{theorem} \label{thrm:MethodCorrectness}
For any $k \in \mathbb{N}$ and $\varepsilon_k \ge 0$ the following statements hold true:
\begin{enumerate}
\item{if $\Gamma(\widehat{x}_k, \widehat{u}_k; x_k, u_k; V_k) \ge \Gamma(x_k, u_k; x_k, u_k; V_k)$, then for any
$\varepsilon_{feas} > 0$ there exists $c_* > 0$ such that for all $c_+ \ge c_*$ inequality
\eqref{eq:NearOptimalInfeasMeas} holds true;
}

\item{if $\Gamma(\widehat{x}_k, \widehat{u}_k; x_k, u_k; V_k) < \Gamma(x_k, u_k; x_k, u_k; V_k)$, then for any 
$\eta_1 \in (0, 1)$ there exists $c_* > 0$ such that for all $c_+ \ge c_*$ inequality \eqref{eq:InfeasMeasDecay} holds
true;
}

\item{for any $\eta_2 \in (0, 1)$ there exists $c_* \ge c_+$ (here $c_+$ is from Step~4 of the Algorithmic
Pattern~\ref{alg:Boosted_SEP_DCA}) such that for all $c_{k + 1} \ge c_*$ inequality \eqref{eq:PenFuncDecay} holds true.
}
\end{enumerate}
\end{theorem}

\begin{proof}
The validity of the first and second statements of the theorem follows directly from the facts that
$\Gamma(x_k[c], u_k[c]; x_k, u_k; V_k)$ converges to the optimal value $\Gamma^*_k$ of problem
\eqref{subprob:OptInfeasMeas} as $c \to + \infty$ by Lemma~\ref{lem:InfeasMeasBehaviour}, and 
$\Gamma^*_k \le \Gamma(\widehat{x}_k, \widehat{u}_k; x_k, u_k; V_k) \le \Gamma^*_k + \varepsilon_k$ 
by definition.

Let us prove the third statement of the theorem. We will split the proof of this statement into two parts corresponding
to two different cases.

\textbf{Case I.} Suppose that $\Gamma(x_k, u_k; x_k, u_k; V_k) = \Gamma^*_k$. The for any $c \ge c_0$ one has
$\Gamma(x_k[c], u_k[c]; x_k, u_k; V_k) \ge \Gamma(x_k, u_k; x_k, u_k; V_k)$ and the right-hand side of inequality
\eqref{eq:PenFuncDecay} is greater than or equal to zero. In turn, by Assumption~\ref{assumpt:SubOptSolCorrectness} for
any $c_{k + 1} \ge c_0$ one has
\[
  Q_{c_{k + 1}}(x_k[c_{k + 1}], u_k[c_{k + 1}]; x_k, u_k; V_k) \le Q_{c_{k + 1}}(x_k, u_k; x_k, u_k; V_k),
\]
that is, the left-hand side of inequality \eqref{eq:PenFuncDecay} is less than or equal to zero. Thus, 
inequality \eqref{eq:PenFuncDecay} holds true for any $c_{k + 1} \ge c_+$.

\textbf{Case II.} Suppose that $\Gamma(x_k, u_k; x_k, u_k; V_k) > \Gamma^*_k$. Then by the third statement of
Lemma~\ref{lem:InfeasMeasBehaviour} there exist $\widehat{c} \ge c_+$ and $\delta > 0$ such that for any 
$c \ge \widehat{c}$ one has
\begin{equation} \label{eq:FeasibilityDecraseGap}
  \Gamma(x_k[c], u_k[c]; x_k, u_k; V_k) \le \Gamma(x_k, u_k; x_k, u_k; V_k) - \delta 
  \quad \forall c \ge \widehat{c}.
\end{equation}
Choose any $\lambda \in (0, 1 - \eta_2)$ and let $(\overline{x}_k, \overline{u}_k)$ be a $\lambda \delta$-optimal
solution of the optimal infeasibility problem \eqref{subprob:OptInfeasMeas}. By definition 
$Q_c(x_k[c], u_k[c]; x_k, u_k; V_k) \le Q_c(\overline{x}_k, \overline{u}_k; x_k, u_k; V_k) + \varepsilon_k$
for any $c \ge c_0$ or, equivalently,
\begin{multline*}
  \omega_k(x_k[c], u_k[c]) \le \omega_k(\overline{x}_k, \overline{u}_k) 
  + c \Big( \Gamma(\overline{x}_k, \overline{u}_k; x_k, u_k; V_k) - \Gamma(x_k[c], u_k[c]; x_k, u_k; V_k) \Big) 
  + \varepsilon_k
  \\
  \le \omega_k(\overline{x}_k, \overline{u}_k) + c \Big( \Gamma(\overline{x}_k, \overline{u}_k; x_k, u_k; V_k) 
  - \inf_{(x, u) \in X_0}\Gamma(x_k, u_k; x_k, u_k; V_k) \Big) + \varepsilon_k
  \le \omega_k(\overline{x}_k, \overline{u}_k) 
  + c \left( \lambda + \frac{\varepsilon_k}{c \delta} \right) \delta.
\end{multline*}
Pick any $\gamma \in (0, 1 - \eta_2 - \lambda)$. Then applying inequality~\eqref{eq:FeasibilityDecraseGap} one obtains
that for any
\[
  c \ge c_* := \max\left\{ \widehat{c}, \frac{\varepsilon_k}{\delta \gamma},
  \frac{\omega_k(\overline{x}_k, \overline{u}_k) - \omega_k(x_k, u_k)}{(1 - \eta_2 - \lambda - \gamma) \delta} \right\}
\]
one has
\begin{align*}
  \omega_k(x_k[c], u_k[c]) - \omega_k(x_k, u_k) 
  &\le \omega_k(\overline{x}_k, \overline{u}_k) - \omega_k(x_k, u_k) 
  + c \Big( \lambda + \frac{\varepsilon_k}{c \delta} \Big) \delta
  \le c \left( 1 - \eta_2 - \lambda - \gamma \right) \delta 
  + c \left( \lambda + \gamma \right) \delta
  \\
  &\le c (1 - \eta_2) \Big( \Gamma(x_k, u_k; x_k, u_k; V_k) - \Gamma(x_k[c], u_k[c]; x_k, u_k; V_k) \Big).
\end{align*}
Adding $c (\Gamma(x_k[c], u_k[c]; x_k, u_k; V_k) - \Gamma(x_k, u_k; x_k, u_k; V_k))$ to both sides of this inequality
one obtains that
\[
  Q_{c}(x_k[c], u_k[c]; x_k, u_k; V_k) - Q_{c}(x_k, u_k; x_k, u_k; V_k)
  \le c \eta_2 \Big( \Gamma(x_k[c], u_k[c]; x_k, u_k; V_k) - \Gamma(x_k, u_k; x_k, u_k; V_k) \Big),
\]
that is, inequality \eqref{eq:PenFuncDecay} is satisfied for any $c_{k + 1} \ge c_*$.
\end{proof}

Let us finally prove the correctness of the line search procedure on Step~5 of 
Algorithmic Pattern~\ref{alg:Boosted_SEP_DCA}.

\begin{proposition} \label{prp:PenFuncContinuity}
For any $k \in \mathbb{N}$ and for any values of parameters $\zeta \in (0, 1)$, $\overline{\alpha}_k \ge 0$, 
$\sigma > 0$, and $\nu_k > 0$ there exists $j_k \in \mathbb{N}$ satisfying inequality \eqref{eq:LineSearchIneq}.
\end{proposition}

\begin{proof}
Since $\nu_k > 0$, the validity of inequality \eqref{eq:LineSearchIneq} for any $j_k \in \mathbb{N}$ large enough
follows directly from the continuity of the penalty function $\Phi_c$ on the normed space 
$X = W^{1, \infty}([0, T]; \mathbb{R}^n) \times L^{\infty}([0, T])$. Therefore, let us prove that this function is
continuous.

Choose any $(x, u) \in X$ and any sequence $\{ (\overline{x}_k, \overline{u}_k) \} \subset X$ converging to $(x, u)$.
Then this sequence is bounded, $\overline{x}_k(t) \to x(t)$, $\dot{\overline{x}}_k(t) \to \dot{x}(t)$, and 
$\overline{u}_k(t)\to u(t)$ as $k \to \infty$ for a.e. $t \in [0, T]$, and 
$(\overline{x}_k(0), \overline{x}_k(T)) \to (x(0), x(T))$ as $k \to \infty$. Therefore
$f_i(\overline{x}_k(0), \overline{x}_k(T)) \to f_i(x(0), x(T))$ as $k \to \infty$ for any 
$i \in \{ 0 \} \cup \mathcal{I} \cup \mathcal{E}$, and for a.e. $t \in [0, T]$ one has
\begin{gather*}
  F_0(\overline{x}_k(t), \overline{u}(t), t) \to F_0(x(t), u(t), t), \quad
  \dot{\overline{x}}_k(t) - F(\overline{x}_k(t), \overline{u}(t), t) \to \dot{x}(t) - F(x(t), u(t), t), 
  \\
  \Xi_s(\overline{x}_k(t), \overline{u}_k(t), t) \to \Xi_s(x(t), u(t), t), \quad s \in \mathcal{M},
\end{gather*} 
as $k \to \infty$, since by our assumption all these maps are Carath\'{e}odory functions. Moreover, all these functions
are uniformly (in $k$) bounded due to the growth condition \eqref{eq:GrowthCondition} (see 
Assumption~\ref{assumpt:ContDiff}) and the the boundedness of the sequence $\{ (\overline{x}_k, \overline{u}_k) \}$ 
in $X$. Therefore applying Lebesgue's dominated convergence theorem one obtains that $\Phi_c(\overline{x}_k,
\overline{u}_k) \to \Phi_c(x, u)$ as $k \to \infty$ for any
$c > 0$. Thus, the penalty function $\Phi_c$ is continuous on $X$.
\end{proof}

\begin{remark} \label{rmrk:PenFuncContinuity}
Almost literally repeating the proof of the previous proposition and taking into account the fact that by 
the Sobolev imbedding theorem (Adams\cite{Adams}, Thm.~5.4) any convergent sequence in $W^{1, 1}[0, T]$ converges in
$C[0, T]$, one can readily verify that if a sequence $\{ (\overline{x}_k, \overline{u}_k) \} \subset X$ converges to
some $(x, u) \in X$ in the topology of the space $W^{1, 1}([0, T]; \mathbb{R}^n) \times L^1([0, T]; \mathbb{R}^m)$ and
the sequence $\{ \overline{u}_k \}$ is uniformly bounded, then 
$\varphi(\overline{x}_k, \overline{u}_k) \to \varphi(x, u)$ and
$\Phi_c(\overline{x}_k, \overline{u}_k) \to \Phi_c(x, u)$ as $k \to \infty$. Moroever, if the sequence 
$\{ \overline{u}_k \}$ is not uniformly bounded, then 
$\liminf_{k \to \infty} \varphi(\overline{x}_k, \overline{u}_k) \ge \varphi(x, u)$ by Fatou's lemma (see, e.g.
Ref.\cite{DunfordSchwartz}, Thm.~III.6.19), that is, the penalty term $\varphi$ is lower semicontinuous on the space $X$
endowed with the topology of $W^{1, 1}([0, T]; \mathbb{R}^n) \times L^1([0, T]; \mathbb{R}^m)$.
\end{remark}

\section{Convergence analysis}
\label{sect:ConvergenceToCriticalPoints}

Let us now turn to convergence analysis of B-STEP-DCA. We will analyse convergence of the infeasibility measure in
Subsections~\ref{subsect:Control_vs_traj} and \ref{subsect:InfeasMeas} in the general case, but for the sake of
simplicity study the behaviour of the sequence $\{ \Phi_{c_k}(x_k, u_k) \}$ and convergence to critical points in the
case when the penalty parameter $c_k$ remains bounded throughout iterations. Although the case when the penalty
parameter increases unboundedly is very important from the theoretical point of view, it is not particularly relevant
for a practical implementation of Algorithmic Pattern~\ref{alg:Boosted_SEP_DCA}, since, as was noted above (see
Remark~\ref{rmrk:MaxPenParamSafeguard}), an implementation of this method must include an upper bound $c_{max} > 0$ for
the penalty parameter, above which $c_k$ cannot be increased. Nonetheless, let us mention that it seems possible to
extend the cumbersome analysis of conditions ensuring the boundedness of the penalty parameter for the steering exact
penalty DCA\cite{Dolgopolik_StExPenDCA} in the finite dimensional case to the case of Algorithmic
Pattern~\ref{alg:Boosted_SEP_DCA}. We leave this extension as an interesting open problem for future research.

\subsection{Behaviour of the sequence $\{ \Phi_{c_k}(x_k, u_k) \}$ and correctness of stopping criteria}
\label{subsect:ElementaryAnalysis}

In order too facilitate convergence analysis, we make the following natural assumption on the way the penalty parameter
$c_k$ is increased in B-STEP-DCA.

\begin{assumption} \label{assumpt:PenaltyIncrease}
There exists $\varkappa > 0$ such that if the penalty parameter $c_+$ (or $c_{k + 1}$) is increased on Steps 2, 3 (or 4)
of Algorithmic Pattern~\ref{alg:Boosted_SEP_DCA}, then it is increased at least by $\varkappa$.
\end{assumption}

\begin{remark}
Let us note that typical rules for increasing the penalty parameter, such as $c_+ \leftarrow c_+ + \varkappa$ or 
$c_+ \leftarrow \theta c_+$ for some $\theta > 1$, obviously satisfy the above assumption.
\end{remark}

Let us first analyse behaviour of the sequence $\{ \Phi_{c_k}(x_k, u_k) \}$. Denote 
$\rho_k = \| (x_k[c_{k + 1}] - x_k, u_k[c_{k + 1}] - u_k) \|_2$ for any $k \in \mathbb{N}$.

\begin{lemma} \label{lem:PenFunctionDecay}
Let the sequence $\{ (x_k, u_k) \}$ be generated by Algorithmic Pattern~\ref{alg:Boosted_SEP_DCA}. Then for any 
$k \in \mathbb{N}$ one has
\begin{equation} \label{eq:PenFunctionDecay}
  \Phi_{c_{k + 1}}(x_{k + 1}, u_{k + 1}) \le \Phi_{c_{k + 1}}(x_k, u_k) - \sigma \alpha_k^2 \rho_k^2 + \nu_k.
\end{equation}
Moreover, if $H_0(x, u, t)$ is strongly convex in $(x, u)$ with modulus $\mu > 0$ for a.e. $t \in [0, T]$, then for any
$k \in \mathbb{N}$ one has
\begin{equation} \label{eq:PenFunctionDecay_StronglyConvex}
  \Phi_{c_{k + 1}}(x_{k + 1}, u_{k + 1}) 
  \le \Phi_{c_{k + 1}}(x_k, u_k) - \left( \frac{\mu}{2} + \sigma \alpha_k^2 \right) \rho_k^2 + \nu_k.
\end{equation}
If, in addition, the function $h_0$ is strongly convex with modulus $\mu' > 0$, then for any $k \in \mathbb{N}$ one has
\[
  \Phi_{c_{k + 1}}(x_{k + 1}, u_{k + 1}) 
  \le \Phi_{c_{k + 1}}(x_k, u_k) - \left( \frac{\mu}{2} + \sigma \alpha_k^2 \right) \rho_k^2 
  - \frac{\mu'}{2} \Big( |x_k[c_{k + 1}](0) - x_k(0)|^2 + |x_k[c_{k + 1}](T) - x_k(T)|^2 \Big) + \nu_k.
\]
\end{lemma}

\begin{proof}
Fix any $k \in \mathbb{N}$. In accordance with the scheme of Algorithmic Pattern~\ref{alg:Boosted_SEP_DCA}, the point 
$(x_k[c_{k + 1}], u_k[c_{k + 1}])$ is defined as an $\varepsilon_k$-optimal solution of the subproblem
\eqref{subprob:MainStep} with $c = c_{k + 1}$ satisfying Assumption~\ref{assumpt:SubOptSolCorrectness}. Therefore
\begin{multline} \label{eq:MainStepIneq}
  Q_{c_{k + 1}}(x_k[c_{k + 1}], u_k[c_{k + 1}]; x_k, u_k; V_k) 
  \le Q_{c_{k + 1}}(x_k, u_k; x_k, u_k; V_k)
  = \Phi_{c_{k + 1}}(x_k, u_k) 
  \\
  + \int_0^T \Big( H_0(x_k(t), u_k(t), t) - \langle V_{k0}(t), (x_k(t), u_k(t)) \rangle \Big) \, dt
  + h_0(x_k(0), x_k(T)) - \langle v_{k0}, (x_k(0), x_k(T)) \rangle
\end{multline}
(see Eqs.~\eqref{eq:ExPenConvexMajorant} and \eqref{eq:GlobalConvexMajorant}). Subtracting 
\begin{equation} \label{eq:CostFunctionCorrection}
  \int_0^T \Big( H_0(x_k(t), u_k(t), t) - \langle V_{k0}(t), (x_k(t), u_k(t)) \rangle \Big) \, dt
  + h_0(x_k(0), x_k(T)) - \langle v_{k0}, (x_k(0), x_k(T)) \rangle
\end{equation}
from both sides of this inequality and applying the definition of subgradient of a convex function one obtains that
\begin{multline*}
  \Phi_{c_{k + 1}}(x_k, u_k) \ge Q_{c_{k + 1}}(x_k[c_{k + 1}], u_k[c_{k + 1}]; x_k, u_k; V_k)
  - \int_0^T \Big( H_0(x_k(t), u_k(t), t) - \langle V_{k0}(t), (x_k(t), u_k(t)) \rangle \Big) \, dt
  \\
  - h_0(x_k(0), x_k(T)) + \langle v_{k0}, (x_k(0), x_k(T)) \rangle
  \ge \Phi_{c_{k + 1}}(x_k[c_{k + 1}], u_k[c_{k + 1}]).
\end{multline*}
Hence by the definition of $x_{k + 1}$ (see Step~5 of Algorithmic Pattern~\ref{alg:Boosted_SEP_DCA}) one gets that
inequality \eqref{eq:PenFunctionDecay} holds true.

Let us now prove inequality \eqref{eq:PenFunctionDecay_StronglyConvex}. By our assumption the function $H_0$ is strongly
convex in $(x, u)$ with modulus $\mu$. Consequently, for a.e. $t \in [0, T]$ one has
\begin{align*}
  H_0(x_k[c_{k + 1}](t), u_k[c_{k + 1}](t), t) - H_0(x_k(t), u_k(t), t) 
  &\ge \langle V_{k0}(t), (x_k[c_{k + 1}](t) - x_k(t), u_k[c_{k + 1}](t) - u_k(t)) \rangle
  \\
  &+ \frac{\mu}{2} \big| (x_k[c_{k + 1}](t) - x_k(t), u_k[c_{k + 1}](t) - u_k(t) \big|^2.
\end{align*}
Subtracting \eqref{eq:CostFunctionCorrection} from both sides of inequality \eqref{eq:MainStepIneq}, applying the
inequality above and the definition of subgradient one gets that
\begin{align*}
  \Phi_{c_{k + 1}}(x_k, u_k) &\ge Q_{c_{k + 1}}(x_k[c_{k + 1}], u_k[c_{k + 1}]; x_k, u_k; V_k) 
  - \int_0^T \Big( H_0(x_k(t), u_k(t), t) - \langle V_{k0}(t), (x_k(t), u_k(t)) \rangle \Big) \, dt
  \\
  &- h_0(x_k(0), x_k(T)) + \langle v_{k0}, (x_k(0), x_k(T)) \rangle
  \\
  &\ge Q_{c_{k + 1}}(x_k[c_{k + 1}], u_k[c_{k + 1}]; x_k, u_k; V_k)
  - h_0(x_k(0), x_k(T)) + \langle v_{k0}, (x_k(0), x_k(T)) \rangle + \frac{\mu}{2} \rho_k^2
  \\
  &- \int_0^T \Big( H_0(x_k[c_{k + 1}](t), u_k[c_{k + 1}](t), t) 
  - \langle V_{k0}(t), (x_k[c_{k + 1}](t), u_k[c_{k + 1}](t)) \rangle \Big) \, dt
  \\
  &= \int_0^T F_0(x_k[c_{k + 1}](t), u_k[c_{k + 1}](t), t) \, dt 
  + g_0(x_k[c_{k + 1}](0), x_k[c_{k + 1}](T)) 
  \\
  &- h_0(x_k(0), x_k(T)) + \langle v_{k0}, (x_k(0), x_k(T)) \rangle
  + c \Gamma(x_k[c_{k + 1}], u_k[c_{k + 1}]; x_k, u_k; V_k) + \frac{\mu}{2} \rho_k^2
  \\
  &\ge \Phi_{c_{k + 1}}(x_k[c_{k + 1}], u_k[c_{k + 1}]) + \frac{\mu}{2} \rho_k^2.
\end{align*}
It remains to note that inequality \eqref{eq:PenFunctionDecay_StronglyConvex} follows directly from the inequality above
and the definition of $x_{k + 1}$ (see Step~5 of Algorithmic Pattern~\ref{alg:Boosted_SEP_DCA}). The validity of the
corresponding inequality in the case when the function $h_0$ is strongly convex is proved in exactly the same way.
\end{proof}

\begin{corollary} \label{crlr:S2Strategy}
Let the sequence $\{ (x_k, u_k) \}$ be generated by Algorithmic Pattern~\ref{alg:Boosted_SEP_DCA}, the sequence of
penalty parameters $\{ c_k \}$ be bounded, and the sequence $\{ \nu_k \}$ be chosen according to strategy $(S2)$.
Suppose also that $\rho_k > 0$ (i.e. $(x_k, u_k)$ is not a critical point of the problem $(\mathcal{P})$) for any 
$k \in \mathbb{N}$. Then the sequence $\{ \nu_k \}$ is correctly defined and $\sum_{k = 0}^{\infty} \nu_k < + \infty$.
\end{corollary}

\begin{proof}\footnote{The proof largely repeats the observation made in Ref.~\cite{FerreiraSantosSouza}, Remark~3.5.}
By the previous lemma one has
\[
   0 < \sigma \alpha_k^2 \rho_k^2 \le \Phi_{c_{k + 1}}(x_k, u_k) - \Phi_{c_{k + 1}}(x_{k + 1}, u_{k + 1}) + \nu_k.
\]
Therefore, for any $\delta_{k + 1} \in (0, 1)$ there exists $\nu_{k + 1} > 0$ satisfying inequality
\eqref{eq:S2Strategy}. In other words, the sequence $\{ \nu_k \}$ is correctly defined.

By virtue of Assumption~\ref{assumpt:PenaltyIncrease}, the boundedness of the sequence $\{ c_k \}$ implies that there
exists $k_0 \in \mathbb{N}$ such that $c_k = c_{k_0}$ for all $k \ge k_0$. From inequalities \eqref{eq:S2Strategy} it
follows that
\[
  0 \le \delta_{k + 1} \Big( \Phi_{c_{k_0}}(x_k, u_k) - \Phi_{c_{k_0}}(x_{k + 1}, u_{k + 1}) + \nu_k \Big)
  \le \Big( \Phi_{c_{k_0}}(x_k, u_k) + \nu_k \Big) 
  - \Big( \Phi_{c_{k_0}}(x_{k + 1}, u_{k + 1}) + \nu_{k + 1} \Big)
\]
for any $k \ge k_0$. Hence bearing in mind the fact that $\delta_{k + 1} \ge \delta_{\min} > 0$ one gets that
\begin{multline*}
  \delta_{\min} \sum_{k = k_0}^N 
  \Big( \Phi_{c_{k_0}}(x_k, u_k) - \Phi_{c_{k_0}}(x_{k + 1}, u_{k + 1}) + \nu_k \Big)
  \le \sum_{k = k_0}^N 
  \Big( \big(\Phi_{c_{k_0}}(x_k, u_k) + \nu_k\big) - \big(\Phi_{c_{k_0}}(x_{k + 1}, u_{k + 1}) + \nu_{k + 1} \big) \Big)
  \\
  = \Phi_{c_{k_0}}(x_{k_0}, u_{k_0}) + \nu_{k_0} - \Phi_{c_{k_0}}(x_{N + 1}, u_{N + 1}) - \nu_{N + 1}
  \le \Phi_{c_{k_0}}(x_{k_0}, u_{k_0}) - \Phi_{c_{k_0}}^* + \nu_{k _0}
\end{multline*}
for any $N \ge k_0$, where $\Phi_{c_{k_0}}^* = \inf_{(x, u) \in X_0} \Phi_{c_{k_0}}(x, u)$. Note that this infimum is
finite by Assumption~\ref{assumpt:PenaltyFuncBoundedBelow}. Thus, by inequalities \eqref{eq:S2Strategy} one has
\[
  \sum_{k = k_0}^{\infty} \nu_k 
  \le (1 - \delta_{\min}) \sum_{k = k_0}^{\infty} 
  \Big( \Phi_{c_{k_0}}(x_k, u_k) - \Phi_{c_{k_0}}(x_{k + 1}, u_{k + 1}) + \nu_k \Big)
  \le \frac{1 - \delta_{\min}}{\delta_{\min}} 
  \Big( \Phi_{c_{k_0}}(x_{k_0}, u_{k_0}) - \Phi_{c_{k_0}}^* + \nu_{k _0} \Big),
\]
and the proof is complete.
\end{proof}

\begin{corollary}
Let the sequence $\{ (x_k, u_k) \}$ be generated by Algorithmic Pattern~\ref{alg:Boosted_SEP_DCA}, the sequence of
penalty parameters $\{ c_k \}$ be bounded, and the sequence $\{ \nu_k \}$ be chosen according to strategy $(S3)$.
Suppose also that the function $H_0(x, u, t)$ is strongly convex in $(x, u)$ with modulus $\mu > 0$ for a.e. 
$t \in [0, T]$. Then $\sum_{k = 0}^{\infty} \nu_k < + \infty$.
\end{corollary}

\begin{proof}
Since the sequence $\{ c_k \}$ is bounded, by Assumption~\ref{assumpt:PenaltyIncrease} there exists $k_1 \in \mathbb{N}$
such that $c_k = c_{k_1}$ for all $k \ge k_1$. In turn, by the definition of strategy $(S3)$ there exists 
$k_2 \in \mathbb{N}$ such that $\nu_k \le (\mu/4) \rho_k^2$ for all $k \ge k_2$. Consequently, by
Lemma~\ref{lem:PenFunctionDecay} one has
\[
  \Phi_{c_{k_1}}(x_{k + 1}, u_{k + 1}) 
  \le \Phi_{c_{k_1}}(x_k, u_k) - \left( \frac{\mu}{4} + \sigma \alpha_k^2 \right) \rho_k^2
  \quad \forall k \ge \max\{ k_1, k_2 \},
\]
which due to Assumption~\ref{assumpt:PenaltyFuncBoundedBelow} implies that $\sum_{k = 0}^{\infty} \rho_k^2 < + \infty$
and, therefore, $\sum_{k = 0}^{\infty} \nu_k < + \infty$ as well.
\end{proof}

\begin{remark}
{(i)~Note that if one chooses $\nu_k = \overline{\nu} \rho_k^2$ for some fixed 
$0 < \overline{\nu} < \mu/2$, then arguing in the same way as in the proof of the corollary above one gets that 
$\sum_{k = 0}^{\infty} \rho_k^2 < + \infty$ and $\sum_{k = 0}^{\infty} \nu_k < + \infty$, provided the function 
$H_0(x, u, t)$ is strongly convex in $(x, u)$ with modulus $\mu > 0$. Thus, in the strongly convex case one can choose 
$\nu_k = \overline{\nu} \rho_k^2$ for all $k \in \mathbb{N}$, provided $0 < \overline{\nu} < \mu/2$ (cf. strategy
$(S3')$ in Ref.~\cite{FerreiraSantosSouza}).
}

\noindent{(ii)~It should be mentioned that from the theoretical point of view the assumption on the strong convexity
of the function $H_0$ in $(x, u)$ is not restrictive, since one can always replace the DC decomposition 
$F_0(x, u, t) = G_0(x, u, t) - H_0(x, u, t)$ of the function $F_0$ with the following one
\[
  F_0(x, u, t) = \Big( G_0(x, u, t) + \frac{\mu}{2} \big( |x|^2 + |u|^2 \big) \Big) 
  - \Big( H_0(x, u, t) + \frac{\mu}{2} \big( |x|^2 + |u|^2 \big) \Big)
\]
for some $\mu > 0$. Note, however, that methods of DC optimisation, including (boosted) DCA, are \textit{not} invariant
with respect to the choice of DC decompositions of the objective functions and constraints 
(cf. Refs.~\cite{FerreiraSantosSouza,FerrerMartinzeLegaz}), and a particular choice of $\mu > 0$ can significantly
affect the performance of the method (in particular, significantly slow down/accelerate convergence). Therefore, from
the practical point of view a proper choice of $\mu > 0$ is a very challenging problem.
}
\end{remark}

With the use of Lemma~\ref{lem:PenFunctionDecay} we can justify the stopping criteria discussed in
Subsection~\ref{subsect:StoppingCriteria}.

\begin{theorem} \label{thrm:StoppingCriteria1}
Let the sequence $\{ (x_k, u_k) \}$ be generated by Algorithmic Pattern~\ref{alg:Boosted_SEP_DCA}, the sequence of
penalty parameters $\{ c_k \}$ be bounded, and $\sum_{k = 0}^{\infty} \nu_k < + \infty$.  Then the following statements
hold true:
\begin{enumerate}
\item{$\sum_{k = 0}^{\infty} \alpha_k^2 \rho_k^2 < + \infty$;}

\item{$|\Phi_{c_{k + 1}}(x_{k + 1}, u_{k + 1}) - \Phi_{c_{k + 1}}(x_k, u_k)| \to 0$ as $k \to \infty$;}

\item{the sequence $\{ \Phi_{c_k}(x_k, u_k) \}$ converges.}
\end{enumerate}
\end{theorem}

\begin{proof}
We split the proof of the theorem into three parts corresponding to each statement of the theorem. Before we proceed to
the proofs of the statements, let us first make an observation. Namely, from the fact that the sequence $\{ c_k \}$ is
bounded and Assumption~\ref{assumpt:PenaltyIncrease} it follows that the penalty parameter $c_k$ is increased only a
finite number of times, that is, there exists $k_0 \in \mathbb{N}$ such that $c_k = c_{k_0}$ for all $k \ge k_0$.

\textbf{Part 1.} By Lemma~\ref{lem:PenFunctionDecay} for any $k \in \mathbb{N}$ one has 
$\Phi_{c_{k + 1}}(x_{k + 1}, u_{k + 1}) - \Phi_{c_{k + 1}}(x_k, u_k) \le - \sigma \alpha_k^2 \rho_k^2 + \nu_k$.
Consequently, for any $s > k_0$ one has
\[
  \Phi_{c_{k_0}}(x_s, u_s) - \Phi_{c_{k_0}}(x_{k_0}, u_{k_0}) 
  = \sum_{k = k_0}^{s - 1} \big( \Phi_{c_{k_0}}(x_{k + 1}, u_{k + 1}) - \Phi_{c_{k_0}}(x_k, u_k) \big) 
  \le - \sum_{k = k_0}^{s - 1} \alpha_k^2 \rho_k^2 + \sum_{k = k_0}^{s - 1} \nu_k.
\]
If $\sum_{k = 0}^{\infty} \alpha_k^2 \rho_k^2 = + \infty$, then from the inequality above and the assumption
$\sum_{k = 0}^{\infty} \nu_k < + \infty$ it follows that $\Phi_{c_{k_0}}(x_s, u_s) \to - \infty$ as $s \to \infty$,
which contradicts Assumption~\ref{assumpt:PenaltyFuncBoundedBelow} on the boundedness below of the penalty function
$\Phi_c$ for $c = c_0$. Therefore, the first statement of the theorem holds true.

\textbf{Part 2.} Note that the difference $\Phi_{c_{k + 1}}(x_{k + 1}, u_{k + 1}) - \Phi_{c_{k + 1}}(x_k, u_k)$ is
bounded above by $- \alpha_k^2 \rho_k^2 + \nu_k$ by Lemma~\ref{lem:PenFunctionDecay}, and this upper bound converges to
zero as $k \to \infty$. Let us show that this difference is also bounded below and the lower bound converges to zero as 
$k \to \infty$ as well. Then one can conclude that 
$|\Phi_{c_{k + 1}}(x_{k + 1}, u_{k + 1}) - \Phi_{c_{k + 1}}(x_k, u_k)| \to 0$ as $k \to \infty$.

Indeed, for any $s > k_0$ one has
\begin{multline*}
  \Phi_{c_{k_0}}(x_s, u_s) - \Phi_{c_{k_0}}(x_{k_0}, u_{k_0}) 
  = \sum_{k = k_0}^{s - 1} \big( \Phi_{c_{k_0}}(x_{k + 1}, u_{k + 1}) - \Phi_{c_{k_0}}(x_k, u_k) \big)
  \\
  = \sum_{k = k_0}^{s - 1} \min\big\{ 0, \Phi_{c_{k_0}}(x_{k + 1}, u_{k + 1}) - \Phi_{c_{k_0}}(x_k, u_k) \big\}
  + \sum_{k = k_0}^{s - 1} \max\big\{ 0, \Phi_{c_{k_0}}(x_{k + 1}, u_{k + 1}) - \Phi_{c_{k_0}}(x_k, u_k) \big\}
  \\
  \le \sum_{k = k_0}^{s - 1} \min\big\{ 0, \Phi_{c_{k_0}}(x_{k + 1}, u_{k + 1}) - \Phi_{c_{k_0}}(x_k, u_k) \big\}
  + \sum_{k = k_0}^{\infty} \nu_k.
\end{multline*}
If $\sum_{k = k_0}^{\infty} \min\big\{ 0, \Phi_{c_{k_0}}(x_{k + 1}, u_{k + 1}) - \Phi_{c_{k_0}}(x_k, u_k) \big\} =
- \infty$, then the inequality above and the assumption $\sum_{k = 0}^{\infty} \nu_k < + \infty$ imply that
$\Phi_{c_{k_0}}(x_s, u_s) \to - \infty$ as $s \to \infty$, which once again contradicts
Assumption~\ref{assumpt:PenaltyFuncBoundedBelow}. Consequently, the sequence 
$\{ \min\big\{ 0, \Phi_{c_{k_0}}(x_{k + 1}, u_{k + 1}) - \Phi_{c_{k_0}}(x_k, u_k) \big\} \}$ is summable, which implies
that it converges to zero as $k \to \infty$. Hence taking into account the inequalities
\[
  \min\big\{ 0, \Phi_{c_{k_0}}(x_{k + 1}, u_{k + 1}) - \Phi_{c_{k_0}}(x_k, u_k) \big\} 
  \le \Phi_{c_{k + 1}}(x_{k + 1}, u_{k + 1}) - \Phi_{c_{k + 1}}(x_k, u_k) 
  \le - \alpha_k^2 \rho_k^2 + \nu_k \quad \forall k \ge k_0
\]
one can conclude that the second statement of the theorem holds true.

\textbf{Part 3.} For any $s_1 > s_2 > k_0$ one has
\begin{multline*}
  \big| \Phi_{c_{k_0}}(x_{s_2}, u_{s_2}) - \Phi_{c_{k_0}}(x_{s_1}, u_{s_1}) \big|
  = \Big| \sum_{k = s_1}^{s_2 - 1} \Phi_{c_{k_0}}(x_{k + 1}, u_{k + 1}) - \Phi_{c_{k_0}}(x_k, u_k) \Big| 
  \\
  = \Big| \sum_{k = s_1}^{s_2 - 1} \min\big\{ 0, \Phi_{c_{k_0}}(x_{k + 1}, u_{k + 1}) - \Phi_{c_{k_0}}(x_k, u_k) \big\}
  + \sum_{k = s_1}^{s_2 - 1} \max\big\{ 0, \Phi_{c_{k_0}}(x_{k + 1}, u_{k + 1}) - \Phi_{c_{k_0}}(x_k, u_k) \big\} \Big|
  \\
  \le \sum_{k = s_1}^{s_2 - 1} 
  \Big| \min\big\{ 0, \Phi_{c_{k_0}}(x_{k + 1}, u_{k + 1}) - \Phi_{c_{k_0}}(x_k, u_k) \big\} \Big|
  + \sum_{k = s_1}^{s_2 - 1} \max\big\{ 0, \Phi_{c_{k_0}}(x_{k + 1}, u_{k + 1}) - \Phi_{c_{k_0}}(x_k, u_k) \big\}
  \\
  \le \Big| \sum_{k = s_1}^{s_2 - 1} 
  \min\big\{ 0, \Phi_{c_{k_0}}(x_{k + 1}, u_{k + 1}) - \Phi_{c_{k_0}}(x_k, u_k) \big\} \Big|
  + \sum_{k = s_1}^{s_2 - 1} \nu_k
\end{multline*}
(here we used the fact that $\sum |a_k| = |\sum a_k|$, if $a_k \le 0$ for all $k$). By the assumption of the theorem and
the second part of the proof the series
\[
  \sum_{k = 0}^{\infty} \nu_k, \quad 
  \sum_{k = 0}^{\infty} \min\big\{ 0, \Phi_{c_{k_0}}(x_{k + 1}, u_{k + 1}) - \Phi_{c_{k_0}}(x_k, u_k) \big\}
\]
are summable. Therefore by Cauchy's convergence test for any $\varepsilon > 0$ there exists $N \in \mathbb{N}$ such that
for all $s_2 > s_1 \ge N$ one has
\[
  \sum_{k = s_1}^{s_2 - 1} \nu_k < \frac{\varepsilon}{2}, \quad
  \Big| \sum_{k = s_1}^{s_2 - 1} 
  \min\big\{ 0, \Phi_{c_{k_0}}(x_{k + 1}, u_{k + 1}) - \Phi_{c_{k_0}}(x_k, u_k) \big\} \Big| < \frac{\varepsilon}{2}.
\]
Hence for any $s_2 > s_1 > \max\{ N, k_0 \}$ one has 
$|\Phi_{c_{k_0}}(x_{s_2}, u_{s_2}) - \Phi_{c_{k_0}}(x_{s_1}, u_{s_1})| < \varepsilon$. 
Thus, $\{ \Phi_{c_k}(x_k, u_k) \}$ is a Cauchy sequence, which implies that it converges.
\end{proof}

\begin{corollary} \label{crlr:StoppingCriteria}
Under the assumptions of the previous theorem for any $k \in \mathbb{N}$ one has
\begin{equation} \label{eq:MajorantDiffEstimate}
  0 \le Q_{c_{k + 1}}(x_k, u_k; x_k, u_k; V_k) - Q_{c_{k + 1}}(x_k[c_{k + 1}], u_k[c_{k + 1}]; x_k, u_k; V_k)
  \le \Phi_{c_{k + 1}}(x_k, u_k) - \Phi_{c_{k + 1}}(x_{k + 1}, u_{k + 1}) + \nu_k
\end{equation}
and $Q_{c_{k + 1}}(x_k, u_k; x_k, u_k; V_k) - Q_{c_{k + 1}}(x_k[c_{k + 1}], u_k[c_{k + 1}]; x_k, u_k; V_k) \to 0$ 
as $k \to \infty$.
\end{corollary}

\begin{proof}
By Assumption~\ref{assumpt:SubOptSolCorrectness} one has
\[
  Q_{c_{k + 1}}(x_k[c_{k + 1}], u_k[c_{k + 1}]; x_k, u_k; V_k) \le Q_{c_{k + 1}}(x_k, u_k; x_k, u_k; V_k)
\]
for any $k \in \mathbb{N}$. Adding and subtracting expression \eqref{eq:CostFunctionCorrection} and applying the
definition of subgradient (see the proof of Lemma~\ref{lem:PenFunctionDecay}) one obtains that
\[ 
  0 \le Q_{c_{k + 1}}(x_k, u_k; x_k, u_k; V_k) - Q_{c_{k + 1}}(x_k[c_{k + 1}], u_k[c_{k + 1}]; x_k, u_k; V_k) 
  \le \Phi_{c_{k + 1}}(x_k, u_k) - \Phi_{c_{k + 1}}(x_k[c_{k + 1}], u_k[c_{k + 1}]) 
\]
Hence taking into account the fact that according to Step~5 of Algorithmic Pattern~\ref{alg:Boosted_SEP_DCA} one has
\[
  \Phi_{c_{k + 1}}(x_{k + 1}, u_{k + 1}) \le \Phi_{c_{k + 1}}(x_k[c_{k + 1}], u_k[c_{k + 1}]) + \nu_k
\]
one gets that inequality \eqref{eq:MajorantDiffEstimate} holds true. The fact that the corresponding difference
converges to zero follows directly from this inequality and the second statement of Theorem~\ref{thrm:StoppingCriteria1}
\end{proof}

\begin{remark}
From the corollary above it follows that if the sequence of penalty parameters $\{ c_k \}$ is bounded, 
$\sum_{k = 0}^{\infty} \nu_k < + \infty$, and $\varphi(x_k, u_k) \to 0$, then Algorithmic
Pattern~\ref{alg:Boosted_SEP_DCA} necessarily terminates after a finite number of iterations and returns a
generalised $(\varepsilon_k + \varepsilon_f + \nu_k)$-critical point, if the stopping criterion
\eqref{eq:StoppingCriterion1} is used, and a geneeralised $(\varepsilon_f + \varepsilon_k)$-critical point, if the
stopping criterion \eqref{eq:StoppingCriterion2} is employed. Note, however, that if the problem is in some sense
degenerate and either $\liminf_{k \to \infty} \varphi(x_k, u_k) > 0$ or one has $c_k \to + \infty$, these stopping
criteria cannot not be satisfied. Therefore, a practical implementation of Algorithmic Pattern~\ref{alg:Boosted_SEP_DCA}
must include corresponding safeguards that take into account such degenerate cases (see
Remark~\ref{rmrk:MaxPenParamSafeguard}).
\end{remark}

Let us also consider the strongly convex case.

\begin{theorem} \label{thrm:StoppingCriteria2}
Let the sequence $\{ (x_k, u_k) \}$ be generated by Algorithmic Pattern~\ref{alg:Boosted_SEP_DCA}, the sequence of
penalty parameters $\{ c_k \}$ be bounded, and $\sum_{k = 0}^{\infty} \nu_k < + \infty$. Suppose also that the function
$H_0(x, u, t)$ is strongly convex in $(x, u)$ with modulus $\mu > 0$ for a.e. $t \in [0, T]$. 
Then $\sum_{k = 0}^{\infty} \| (x_{k + 1} - x_k, u_{k + 1} - u_k) \|_2^2 < + \infty$. In particular, 
$\| x_{k + 1} - x_k \|_2 \to 0$ and $\| u_{k + 1} - u_k \|_2 \to 0$ as $k \to \infty$.
\end{theorem}

\begin{proof}
As was noted in the proof of Theorem~\ref{thrm:StoppingCriteria1}, from the assumption on the boundedness of the penalty
parameter it follows that there exists $k_0 \in \mathbb{N}$ such that $c_k = c_{k_0}$ for all $k \ge k_0$.

By the second part of Lemma~\ref{lem:PenFunctionDecay} for any $s > k_0$ one has
\[
  \Phi_{c_{k_0}}(x_s, u_s) - \Phi_{c_{k_0}}(x_{k_0}, u_{k_0}) 
  = \sum_{k = k_0}^{s - 1} \big( \Phi_{c_{k_0}}(x_{k + 1}, u_{k + 1}) - \Phi_{c_{k_0}}(x_k, u_k) \big) 
  \le - \sum_{k = k_0}^{s - 1} \left( \frac{\mu}{2} + \alpha_k^2 \right) \rho_k^2 
  + \sum_{k = k_0}^{s - 1} \nu_k.
\]
Now by applying Assumption~\ref{assumpt:PenaltyFuncBoundedBelow} and arguing in the same way as in the proof of 
the first part of Theorem~\ref{thrm:StoppingCriteria1} one can readily check that
\begin{equation} \label{eq:DecreaseEstimSeries}
  \sum_{k = 0}^{\infty} \left( \frac{\mu}{2} + \alpha_k^2 \right) \rho_k^2 < + \infty.
\end{equation}
By definition one has $(x_{k + 1} - x_k, u_{k + 1} - u_k) = (1 + \alpha_k) (x_k[c_{k + 1}] - x_k, u_k[c_{k + 1}] - u_k)$
for any $k \in \mathbb{N}$ (see Step~5 of Algorithmic Pattern~\ref{alg:Boosted_SEP_DCA}). Therefore for any 
$N \in \mathbb{N}$ one gets
\[
  \sum_{k = 0}^N \| (x_{k + 1} - x_k, u_{k + 1} - u_k) \|_2^2 
  = \sum_{k = 0}^N (1 + \alpha_k)^2 \| (x_k[c_{k + 1}] - x_k, u_k[c_{k + 1}] - u_k) \|_2^2
  = \sum_{k = 0}^N (1 + \alpha_k)^2 \rho_k^2.
\]
Applying the obvious inequality $(1 + \alpha_k)^2 \le 2 + 2 \alpha_k^2$ one obtains
\[
  \sum_{k = 0}^N \| (x_{k + 1} - x_k, u_{k + 1} - u_k) \|_2^2 \le 2 \sum_{k = 0}^N \rho_k^2 + 2 \sum_{k = 0}^N
\alpha_k^2 \rho_k^2.
\]
Hence taking into account \eqref{eq:DecreaseEstimSeries} one can conclude that the series 
$\sum_{k = 0}^{\infty} \| (x_{k + 1} - x_k, u_{k + 1} - u_k) \|_2^2$ is summable, which completes the proof of the
theorem.
\end{proof}

\subsection{Convergence to critical points}
\label{subsect:GlobalConvergence}

Let us now study convergence of sequences generated by B-STEP-DCA to generalised critical points of the problem
$(\mathcal{P})$. Namely, let us prove that all limit points of the sequence generated by this method (if such points
exist) are generalised critical points of the problem $(\mathcal{P})$. We will prove this result under an additional
differentiability assumption on the mapping $F(x, u, t)$ and the mixed state-control constraints. For the sake of
completeness, we present the main part of the proof without this technical assumption and then explicitly point out when
and why this assumption is needed.

Recall that B-STEP-DCA is a local search method for the problem $(\mathcal{P})$ in 
the space $X = W^{1, \infty}([0, T]; \mathbb{R}^n) \times L^{\infty}([0, T], \mathbb{R}^m)$. Nevertheless, below we
consider limit points $\{ (x_k, u_k) \}$ of the sequence generated by this method in the topology of the space 
$Y = W^{1, 1}([0, T]; \mathbb{R}^n) \times L^1([0, T]; \mathbb{R}^m)$ that is more natural from the point of view of
convergence analysis. Moreover, since this topology is weaker than the topology of the space $X$, limit points in
this topology exist under less restrictive assumption on the sequence $\{ (x_k, u_k) \}$ (that is, in a more general
case). To make the use of such topology consistent with the problem formulation we will assume that limit points belong
to $X$. Note that all limit point of the sequence $\{ (x_k, u_k) \}$ in the topology of the space $Y$ belong to $X$,
provided the sequence $\{ (x_k, u_k) \}$ is bounded in $X$. Moreover, they also belong to the set $X_0$ due to our
assumption that the set $X_0 \subset X$ is closed in the topology of the space $Y$.

\begin{theorem} \label{thrm:ConvergenceToCriticalPoints}
Let the sequence $\{ (x_k, u_k) \}$ be generated by Algorithmic Pattern~\ref{alg:Boosted_SEP_DCA} and suppose that the
following assumptions hold true:
\begin{enumerate}
\item{the sequences of penalty parameters $\{ c_k \}$ and optimality tolerances $\{ \varepsilon_k \}$ are bounded;}

\item{$\sum_{k = 0}^{\infty} \nu_k < + \infty$ and $\limsup_{k \to \infty} \varepsilon_k = \varepsilon^*$;}

\item{the sequence of controls $\{ u_k \}$ is bounded in $L^{\infty}([0, T]; \mathbb{R}^m)$;}

\item{the functions $G_i(x, u, t)$ and $H_i(x, u, t)$, $i \in \{ 1, \ldots, n \}$, and $q_s(x, u, t)$, 
$s \in \mathcal{M}$, are differentiable in $(x, u)$ for a.e. $t \in [0, T]$ and their partial derivatives in $x$ and $u$
are Carath\'{e}odory mappings.
}
\end{enumerate}
Then any limit point 
$(x_*, u_*) \in X$ of the sequence $\{ (x_k, u_k) \}$ in the Banach space 
$W^{1, 1}([0, T]; \mathbb{R}^n) \times L^1([0, T], \mathbb{R}^m)$ is a generalised $\varepsilon_*$-critical point
of the problem $(\mathcal{P})$. Furthermore, if $(x_*, u_*)$ is feasible for this problem, then it is an
$\varepsilon_*$-critical point of the problem $(\mathcal{P})$.
\end{theorem}

\begin{proof}
For the sake of convenience we divide the proof of the theorem into several steps.

\textbf{1. Modes of convergence.} Let $(x_*, u_*)$ be a limit point of the sequence $\{ (x_k, u_k) \}$ in 
$Y = W^{1, 1}([0, T]; \mathbb{R}^n) \times L^1([0, T], \mathbb{R}^m)$, that is, there exists a subsequence 
$\{ (x_{k_l}, u_{k_l}) \}$ converging to $(x_*, u_*)$ in $Y$. By the Sobolev imbedding theorem
(Adams\cite{Adams}, Thm.~5.4) the subsequence $\{ x_{k_l} \}$ converges to $x_*$ uniformly on $[0, T]$. Moreover,
replacing, if necessary, the sequence $\{ (x_{k_l}, u_{k_l}) \}$ by its subsequence one can suppose that $u_{k_l}$
converges to $u_*$ almost everywhere on $[0, T]$ (see, e.g. Bogachev\cite{Bogachev}, Thms.~2.2.5 and 4.5.4).

\textbf{2. Convergence of subgradients.} Since the sequence $x_{k_l}$ converges to $x_*$ uniformly on $[0, T]$, it is
bounded in $C([0, T]; \mathbb{R}^n)$. In particular, the sequences $\{ x_{k_l}(0) \}$ and $x_{k_l}(T) \}$ are bounded
and converge to $x_*(0)$ and $x_*(T)$ respectively. Therefore the corresponding sequences of subgradients 
$v_{i k_l} \in \partial f_i(x_{k_l}(0), x_{k_l}(T))$ and $w_{j k_l} \in \partial g_j(x_{k_l}(0), x_{k_l}(T))$, 
$l \in \mathbb{N}$, are bounded due to the local boundedness of the subdifferential mapping of a convex function
(Rockafellar\cite{Rockafellar}, Thm.~24.7). Consequently, replacing, if necessary, the sequence 
$\{ (x_{k_l}, u_{k_l}) \}$ by its subsequence one can suppose that the corresponding sequences of subgradients converge
to some $v_i^* \in \mathbb{R}^{2d}$ and $w_j^* \in \mathbb{R}^{2d}$. Moreover, $v_i^* \in \partial f_i(x_*(0), x_*(T))$
for any $i \in \{ 0 \} \cup \mathcal{I} \cup \mathcal{E}$, and $w_j^* \in \partial g_j(x_*(0), x_*(T))$ 
for any $j \in \mathcal{E}$ by virtue of the fact that the graph of the subdifferential mapping of a convex function is
closed (Rockafellar\cite{Rockafellar}, Thm.~24.4).

By our assumption the sequence $\{ u_k \}$ is uniformly essentially bounded. Then there exists $R > 0$ such that
$|x_{k_l}(t)| + |u_{k_l}(t)| \le R$ for a.e. $t \in [0, T]$, which by Assumption~\ref{assumpt:ContDiff} implies that
there exists $\gamma_R > 0$ such that
\begin{equation} \label{eq:FuncSubgrad_UniformlyBounded}
  |V_{k_l i}(t)| \le \gamma_R \quad \forall i \in \{ 0, 1, \ldots, n \}, \quad
  |W_{k_l i}(t)| \le \gamma_R \quad \forall i \in \{ 1, \ldots, n \}, \quad
  |z_{k_l s}(t)| \le \gamma_R \quad \forall s \in \mathcal{M}
\end{equation}
for a.e. $t \in [0, T]$. Therefore
\begin{gather*}
  \int_{|V_{k_l i}| > \gamma_R} |V_{k_l i}(t)| dt = 0 \quad \forall i \in \{ 0, 1, \ldots, n \}, \quad
  \int_{|W_{k_l i}| > \gamma_R} |W_{k_l i}(t)| dt = 0 \quad \forall i \in \{ 1, \ldots, n \},
  \\
  \int_{|z_{k_l s}| > \gamma_R} |z_{k_l s}(t)| dt = 0 \quad \forall s \in \mathcal{M},
\end{gather*}
which by (Bogachev\cite{Bogachev}, Thm.~4.7.20) implies that the sets $\{ V_{k_l i} \}_{l \in \mathbb{N}}$, 
$\{ W_{k_l i} \}_{l \in \mathbb{N}}$, and $\{ z_{k_l s} \}_{l \in \mathbb{N}}$ are weakly relatively compact in 
$L^1([0, T]; \mathbb{R}^{n + m})$. Hence by the Eberlein-\v{S}mulian theorem one can extract weakly convergent
subsequences from these sequences. Replacing, if necessary, the sequence $\{ (x_{k_l}, u_{k_l}) \}$ by its subsequence
one can suppose that the sequences of subgradients $\{ V_{k_l i} \}_{l \in \mathbb{N}}$, 
$\{ W_{k_l i} \}_{l \in \mathbb{N}}$, and $\{ z_{k_l s} \}_{l \in \mathbb{N}}$ weakly converge in 
$L^1([0, T], \mathbb{R}^{n + m})$ to some $V_i^*$, $i \in \{ 0, 1, \ldots, n \}$, $W_j^*$, $j \in \{ 1, \ldots, n \}$,
and $z_s^*$, $s \in \mathcal{M}$.

\textbf{3. Limiting subgradients.} Let us check that $V_i^*(t) \in \partial_{x, u} H_i(x_*(t), u_*(t), t)$ for a.e. 
$t \in [0, T]$ and all $i \in \{ 0, 1, \ldots, n \}$, $W_j^*(t) \in \partial_{x, u} G_j(x_*(t), u_*(t), t)$ for a.e. 
$t \in [0, T]$ and all $i \in \{ 1, \ldots, n \}$, and $z_s^*(t) \in \partial_{x, u} q_s(x_*(t), u_*(t), t)$ for a.e. 
$t \in [0, T]$ and all $s \in \mathcal{M}$. For the sake of brevity we will prove this statement only for $V_0^*$, since
the proof for all other subgradients is exactly the same.

Since the sequence $\{ V_{k_l 0} \}$ weakly converges to $V_0^*$, by Mazur's lemma (see, e.g. Ekeland and
Temam\cite{EkelandTemam}) for any $l \in \mathbb{N}$ there exist $N(l) \in \mathbb{N}$, $N(l) \ge l$ and a set of
nonnegative real numbers $\beta_r(l)$, $r \in \{ l, \ldots, N(l) \}$, such that $\sum_{r = l}^{N(l)} \beta_r(l) = 1$ and
the sequence $V'_l = \sum_{r = l}^{N(l)} \beta_r(l) V_{k_r 0}$ strongly converges to $V_0^*$ in 
$L^1([0, T], \mathbb{R}^{n + m})$.

Denote
\[
  x'_l = \sum_{r = l}^{N(l)} \beta_r(l) x_{k_r}, \quad 
  u'_l = \sum_{r = l}^{N(l)} \beta_r(l) u_{k_r} \quad \forall l \in \mathbb{N}.
\]
Recall that $x_{k_l}$ converges to $x_*$ uniformly on $[0, T]$. Therefore for any $\varepsilon > 0$ there exists 
$l_0 \in \mathbb{N}$ such that for any $l \ge l_0$ one has $\| x_{k_l} - x_* \|_{\infty} < \varepsilon$. Observe that
\[
  \| x'_l - x_* \|_{\infty} = \Big\| \sum_{r = l}^{N(l)} \beta_r(l) (x_{k_r} - x_*) \Big\|_{\infty}
  \le \sum_{r = l}^{N(l)} \beta_r(l) \| x_{k_r} - x_* \|_{\infty} 
  \le \sum_{r = l}^{N(l)} \beta_r(l) \varepsilon = \varepsilon \quad \forall l \ge l_0
\]
(here we used the fact that $\sum_{r = l}^{N(l)} \beta_r(l) = 1$). Thus, the sequence $\{ x'_l \}$ also
converges to $x_*$ uniformly on $[0, T]$. Arguing in the same way one can verify that the sequence $\{ u'_l \}$
converges to $u_*(t)$ almost everywhere on $[0, T]$.

By definition $V_{k_l 0}(t) \in \partial_{x, u} H_0(x_{k_l}(t), u_{k_l}(t), t)$ for a.e. $t \in [0, T]$, which by 
the definition of subdifferential means that
\[
  H_0(x, u, t) - H_0(x_{k_l}(t), u_{k_l}(t), t) \ge \langle V_{k_l 0}(t), (x - x_{k_l}(t), u - u_{k_l}(t)) \rangle
  \quad \forall (x, u) \in \mathbb{R}^{n + m}, \quad \text{for a.e. } t \in [0, T].
\]
Due to the convexity of $H_0(x, u, t)$ in $(x, u)$ one has 
$H_0(x'_l(t), u'_l(t), t)) \le \sum_{r = l}^{N(l)} \beta_r(l) H_0(x_{k_l}(t), u_{k_l}(t), t)$. Hence
with the use of the inequality above one obtains that
\begin{multline} \label{eq:ConvexityMazurSeq}
  H_0(x'_l(t), u'_l(t), t)) \le \sum_{r = l}^{N(l)} \beta_r(l) H_0(x_{k_l}(t), u_{k_l}(t), t)
  \le \sum_{r = l}^{N(l)} \beta_r(l) 
  \Big( H_0(x, u, t) - \langle V_{k_l 0}(t), (x - x_{k_l}(t), u - u_{k_l}(t)) \rangle \Big)
  \\ 
  = H_0(x, u, t) - \langle V'_l(t), (x - x_*(t), u - u_*(t)) \rangle
  - \sum_{r = l}^{N(l)} \beta_r(l) \langle V_{k_l 0}(t), (x_*(t) - x_{k_l}(t), u_*(t) - u_{k_l}(t)) \rangle
\end{multline}
for any $(x, u) \in \mathbb{R}^{n + m}$ and a.e. $t \in [0, T]$. Observe that
\begin{align*}
  \Big| \sum_{r = l}^{N(l)} \beta_r(l) \langle V_{k_l 0}(t), (x_*(t) - x_{k_l}(t), u_*(t) - u_{k_l}(t)) \rangle \Big|
  &\le \sum_{r = l}^{N(l)} \beta_r(l) \gamma_R \big| (x_*(t) - x_{k_l}(t), u_*(t) - u_{k_l}(t))) \big|
  \\
  &= \gamma_R \big| (x_*(t) - x'_l(t), u_*(t) - u'_l(t)) \big|.
\end{align*}
Recall that $x'_l$ converges to $x_*$ uniformly on $[0, T]$, $u'_l$ converges to $u_*$ a.e. on $[0, T]$, while
$V'_l$ converges to $V_0^*$ in $L^1([0, T], \mathbb{R}^{n + m})$. Therefore, as is well known (see, e.g.
Bogachev\cite{Bogachev}, Thms.~2.2.5 and 4.5.4), there exists a subsequence of the sequence $\{ V'_l \}$ that
converges to $V_0^*$ almost everywhere on $[0, T]$. Hence passing to the limit in inequality
\eqref{eq:ConvexityMazurSeq} along the corresponding subsequence one finally gets that
\[
  H_0(x, u, t) - H_0(x_*(t), u_*(t), t)) \ge \langle V_0^*(t), (x - x_*(t), u - u_*(t)) \rangle
  \quad \forall (x, u) \in \mathbb{R}^{n + m}, \quad \text{for a.e. } t \in [0, T]
\]
or, equivalently, $V_0^*(t) \in \partial_{x, u} H_0(x_*(t), u_*(t), t)$, which completes the third step of the proof.

\textbf{4. Convergence of the penalty function values.} Let us show that for any $c > 0$ one has
\begin{align} \label{eq:PenFuncValueConvergence1}
  \lim_{l \to \infty} Q_c(x_{k_l}, u_{k_l}, x_{k_l}, u_{k_l}, V_{k_l}) &= Q_c(x_*, u_*; x_*, u_*; V^*),
  \\
  \lim_{l \to \infty} Q_c(x, u, x_{k_l}, u_{k_l}, V_{k_l}) &= Q_c(x, u; x_*, u_*; V^*) \quad \forall (x, u) \in X_0.
  \label{eq:PenFuncValueConvergence2}
\end{align}
Indeed, fix any $c > 0$ and $(x, u) \in X_0$. For the sake of convenience, recall that by definition
\begin{align} \notag
  Q_c&(x, u, x_{k_l}, u_{k_l}, V_{k_l})  
  = \int_0^T \Big( G_0(x(t), u(t), t) - \langle V_{k_l 0}(t), (x(t), u(t)) \rangle \Big) \, dt 
  + g_0(x(0), x(T)) - \langle v_{k_l 0}, (x(0), x(T)) \rangle
  \\ \notag
  &+ c \sum_{i = 1}^n \int_0^T 
  \begin{aligned}[t]
    \max\Big\{ &\dot{x}^{(i)}(t) + H_i(x(t), u(t), t) 
    - G_i(x_{k_l}(t), u_{k_l}(t), t) - \langle W_{k_l i}(t), (x(t) - x_{k_l}(t), u(t) - u_{k_l}(t)) \rangle,
    \\
    &- \dot{x}^{(i)}(t) + G_i(x(t), u(t), t) - H_i(x_{k_l}(t), u_{k_l}(t), t) 
    - \langle V_{k_l i}(t), (x(t) - x_{k_l}(t), u(t) - u_{k_l}(t)) \rangle \Big\} \, dt
  \end{aligned}
  \\ \notag
  &+ c \sum_{i = 1}^{\ell_I}
  \max\big\{ 0, g_i(x(0), x(T)) - h_i(x_{k_l}(0), x_{k_l}(T)) 
  - \langle v_{k_l i}, (x(0) - x_{k_l}(0), x(T) - x_{k_l}(T)) \rangle \big\}
  \\ \notag
  &+ c \sum_{j = \ell_I + 1}^{\ell_E}
  \begin{aligned}[t]
    \max\big\{ &g_i(x(0), x(T)) - h_i(x_{k_l}(0), x_{k_l}(T)) 
    - \langle v_{k_l i}, (x(0) - x_{k_l}(0), x(T) - x_{k_l}(T)) \rangle,
    \\
    &h_i(x(0), x(T)) - g_i(x_{k_l}(0), x_{k_l}(T)) 
    - \langle w_{k_l i}, (x(0) - x_{k_l}(0), x(T) - x_{k_l}(T)) \rangle \big\}
  \end{aligned}
  \\ \label{eq:PenFuncAtIterateDef}
  &+ c \sum_{s = 1}^{\ell_M} \int_0^T \max\Big\{ 0, 
  p_s(x(t), u(t), t) - q_s(x_{k_l}(t), u_{k_l}(t), t) 
  - \langle z_{k_l s}(t), (x(t) - x_{k_l}(t), u(t) - u_{k_l}(t)) \rangle \Big\} \, dt.
\end{align}
Let us prove equality \eqref{eq:PenFuncValueConvergence1} first. Putting $(x, u) = (x_{k_l}, u_{k_l})$ one gets
\begin{align*}
  Q_c(x_{k_l}, u_{k_l}, x_{k_l}, u_{k_l}, V_{k_l}) 
  &= \int_0^T \Big( G_0(x_{k_l}(t), u_{k_l}(t), t) - \langle V_{k_l 0}(t), (x_{k_l}(t), u_{k_l}(t)) \Big) \, dt 
  \\
  &+ g_0(x_{k_l}(0), x_{k_l}(T)) - \langle v_{k_l 0}, (x_{k_l}(0), x_{k_l}(T)) \rangle 
  + c \varphi(x_{k_l}, u_{k_l}).
\end{align*}
The convergence of $\varphi(x_{k_l}, u_{k_l})$ to $\varphi(x_*, u_*)$ and the equality
\[
  \lim_{l \to \infty} \int_0^T G_0(x_{k_l}(t), u_{k_l}(t), t) \, dt = \int_0^T \Big( G_0(x_*(t), u_*(t), t) \, dt
\]
can be proved with the use of Lebesgue's dominated convergence theorem and Assumption~\ref{assumpt:ContDiff} in
precisely the same way as in the proof of Proposition~\ref{prp:PenFuncContinuity} 
(see also Remark~\ref{rmrk:PenFuncContinuity}. The equality
\[
  \lim_{l \to \infty} 
  \big( g_0(x_{k_l}(0), x_{k_l}(T)) - \langle v_{k_l 0}, (x_{k_l}(0), x_{k_l}(T)) \rangle \big)
  = g_0(x_*(0), x_*(T)) - \langle v^*_0, (x_*(0), x_*(T)) \rangle
\]
holds true due to the uniform convergence of $x_{k_l}$ to $x_*$ and the convergence of the corresponding subgradients.
Finally, by inequalities \eqref{eq:FuncSubgrad_UniformlyBounded} there exists $\gamma_R > 0$ such that
\begin{align*}
  \Big| \int_0^T \Big( \langle V_{k_l 0}(t), (x_{k_l}(t), u_{k_l}(t)) 
  - \langle V^*_0(t), (x_*(t), u_*(t)) \rangle \Big) \, dt \Big|
  &\le \Big| \int_0^T \langle V_{k_l 0}(t) - V^*_0(t), (x_*(t), u_*(t)) \rangle \, dt \Big| 
  \\
  + \Big| \int_0^T \langle V_{k_l 0}(t), (x_{k_l}(t) - x_*(t), u_{k_l}(t) - u_*(t)) \rangle dt \Big|
  &\le \Big| \int_0^T \langle V_{k_l 0}(t) - V^*_0(t), (x_*(t), u_*(t)) \rangle \, dt \Big| 
  \\
  &+ \gamma_R T \| x_{k_l} - x_* \|_{\infty} + \gamma_R \| u_{k_l} - u_* \|_1.
\end{align*}
The right-hand size of this inequality converges to zero as $l \to \infty$, since $x_{k_l}$ uniformly converges to
$x_*$, $u_{k_l}$ converges to $u_*$ in $L^1([0, T], \mathbb{R}^m)$, and $V_{k_l 0}$ weakly converges to $V^*_0$.
Therefore, equality \eqref{eq:PenFuncValueConvergence1} holds true.

Let us now prove equality \eqref{eq:PenFuncValueConvergence2}. To this end, we shall prove that each term in the
expression \eqref{eq:PenFuncAtIterateDef} converges to the corresponding term in the expression for 
$Q_c(x, u; x_*, u_*; V^*)$ as $l \to \infty$. 

The convergence of the boundary terms follows directly from the convergence of $(x_{k_l}(0), x_{k_l}(T))$ to 
$(x_*(0), x_*(T))$ and the convergence of the corresponding subgradients (see Steps 1 and 2 of the proof). Let us now
consider the integral terms. The fact that
\[
  \int_0^T \langle V_{k_l 0}(t), (x(t), u(t)) \rangle \, dt \to 
  \int_0^T \langle V^*_0(t), (x(t), u(t)) \rangle \, dt
\]
follows directly from the weak convergence of $V_{k_l}$ to $V_*$. Let us now prove the convergence of the terms
\begin{equation} \label{eq:SystDynPenTerm}
\begin{split}
  \int_0^T \max\Big\{ &\dot{x}^{(i)}(t) + H_i(x(t), u(t), t) 
  - G_i(x_{k_l}(t), u_{k_l}(t), t) - \langle W_{k_l i}(t), (x(t) - x_{k_l}(t), u(t) - u_{k_l}(t) \rangle,
  \\
  &- \dot{x}^{(i)}(t) + G_i(x(t), u(t), t) - H_i(x_{k_l}(t), u_{k_l}(t), t) 
  - \langle V_{k_l i}(t), (x(t) - x_{k_l}(t), u(t) - u_{k_l}(t) \rangle \Big\} \, dt
\end{split}
\end{equation}
The convergence of the terms corresponding to the mixed state-control constraints can be proved in exactly the same way.

Note that each term under the max operator in expression \eqref{eq:SystDynPenTerm} converges only weakly in the general
case due to the weak convergence of the subgradients $\{ W_{k_l i} \}$ and $\{ V_{k_l i} \}$, but the max operator is
not weakly continuous in $L^p$ spaces (see Chen and Wickstead\cite{ChenWickstead}, Crlr.~2.3). To overcome this
difficulty, we will use the assumption on differentiability of the function $G_i$ and $H_i$. Indeed, under this
assumption one has
\[
  V_{k_l i}(t) = \nabla G_i(x_{k_l}(t), u_{k_l}(t), t), \quad
  W_{k_l i}(t) = \nabla H_i(x_{k_l}(t), u_{k_l}(t), t) \quad \forall t \in [0, T],
\]
where $\nabla G(x, u, t)$ is the gradient of the function $(x, u) \mapsto G(x, u, t)$. Since by our assumption the
gradients of the functions $G_i$ and $H_i$ are Carath\'{e}odory mappings, one has
\[
  \lim_{l \to \infty} V_{k_l i}(t) = \nabla G_i(x_*(t), u_*(t), t) = V_i^*(t), \quad
  \lim_{l \to \infty} W_{k_l i}(t) = \nabla H_i(x_*(t), u_*(t), t) = W_i^*(t)
\]
for a.e. $t \in [0, T]$. Hence by applying the growth conditions from Assumption~\ref{assumpt:ContDiff} and Lebesgue's
dominated convergence theorem one can readily verify that $V_{k_l i}$ strongly converges to $V_i^*$ in 
$L^1([0, T]; \mathbb{R}^{m + n})$, while $W_{k_l i}$ strongly converges to $W_i^*$ in $L^1([0, T]; \mathbb{R}^{m + n})$.
With the use of the strong convergence of subgradients, the growth conditions from Assumption~\ref{assumpt:ContDiff},
and Lebesgue's dominated convergence theorem one can easily check the convergence of the integral terms
\eqref{eq:SystDynPenTerm}.

\textbf{5. Proof of criticality.} Denote
\[
  V^* = (V_0^*(\cdot), V_1^*(\cdot), \ldots, V_n^*(\cdot), W_1^*(\cdot), \ldots, W_n^*(\cdot), 
  v_0^*, v_1^*, \ldots, v_{\ell_E}^*, w_{\ell_I + 1}^*, \ldots, w_{\ell_E}^*, 
  z_1^*(\cdot), \ldots z_{\ell_m}^*(\cdot)).
\]
As was noted in the proof of Theorem~\ref{thrm:StoppingCriteria1}, from the assumption on the boundedness of the penalty
parameter and Assumption~\ref{assumpt:PenaltyIncrease} it follows that there exist $\widehat{k} \in \mathbb{N}$ and 
$c_* > 0$ such that $c_k = c_*$ for all $k \ge \widehat{k}$. Therefore, we can suppose that 
$c_{k_l} = c_{k_l + 1} = c_*$ for all $l \in \mathbb{N}$. 

Arguing by reductio ad absurdum, suppose that $(x_*, u_*)$ is not a generalised $\varepsilon_*$-critical point of 
the problem $(\mathcal{P})$. According to Def.~\ref{def:EpsilonCriticality} it means that $(x_*, u_*)$ is not a
$\varepsilon_*$-optimal solution of the problem
\[
  \minimise_{(x, u) \in X_0} \enspace Q_{c_*}(x, u; x_*, u_*; V^*).
\]
Denote the optimal value of this problem by $Q^*$. Then by definition 
$Q^* + \varepsilon_* < Q_{c_*}(x_*, u_*, x_*, u_*; V^*)$, which implies that there exists $\theta > 0$ and
$(\widehat{x}, \widehat{u}) \in X_0$ such that
\[
  Q_{c_*}(\widehat{x}, \widehat{u}; x_* u_*; V^*)  + \varepsilon_* + \theta \le Q_{c_*}(x_*, u_*; x_*, u_*; V^*).
\]
With the use of equalities \eqref{eq:PenFuncValueConvergence1} and \eqref{eq:PenFuncValueConvergence2}, and the fact
that $\limsup_{k \to \infty} \varepsilon_k = \varepsilon_*$ one obtains that there exists $l_0 \in \mathbb{N}$ such
that $k_{l_0} \ge \widehat{k}$ and
\[
  Q_{c_*}(\widehat{x}, \widehat{u}; x_{k_l}, u_{k_l}; V_{k_l}) + \varepsilon_k + \frac{\theta}{2}
  \le Q_{c_*}(x_{k_l}, u_{k_l}; x_{k_l}, u_{k_l}; V_{k_l})
\]
for any $l \ge l_0$. Recall that $(x_k[c], u_k[c])$ is an $\varepsilon_k$-optimal solution of problem
\eqref{subprob:MainStep} (see Algorithmic Pattern~\ref{alg:Boosted_SEP_DCA}). Therefore
\[
  Q_{c_*}(x_{k_l}[c_{k_l + 1}], u_{k_l}[c_{k_l + 1}]; x_{k_l}, u_{k_l}; V_{k_l}) + \frac{\theta}{2}
  \le Q_{c_*}(x_{k_l}, u_{k_l}; x_{k_l}, u_{k_l}; V_{k_l})
\]
for any $l \ge l_0$. Now, arguing in the same way as in the proof of Lemma~\ref{lem:PenFunctionDecay} and applying 
the inequality above one can readily verify that
\[
  \Phi_{c_*}(x_{k_l + 1}, u_{k_l + 1}) 
  \le \Phi_{c_*}(x_{k_l}, u_{k_l}) - \sigma \alpha_{k_l}^2 \rho_{k_l}^2 + \nu_{k_l} - \frac{\theta}{2}
\]
for any $l \ge l_0$. Hence with the use of Lemma~\ref{lem:PenFunctionDecay} one gets that
\[
  \Phi_{c_*}(x_{k_{l + 1}} u_{k_{l + 1}}) - \Phi_{c_*}(x_{k_l}, u_{k_l}) 
  = \sum_{i = k_l}^{k_{l + 1} - 1} \Big( \Phi_{c_*}(x_{i + 1}, u_{i + 1}) - \Phi_{c_*}(x_i, u_i) \Big)  
  \le - \sigma \sum_{i = k_l}^{k_{l + 1} - 1} \alpha_i^2 \rho_i^2 
  + \sum_{i = k_l}^{k_{l + 1} - 1} \nu_i - \frac{\theta}{2}
\]
for any $l \ge l_0$. By our assumption the sequence $\{ \nu_k \}$ is summable, which by Cauchy's convergence test
implies that there exists $l_1 \ge l_0$ such that $\sum_{i = k_l}^{k_{l + 1} - 1} \nu_i < \theta/4$ for any $l \ge l_1$.
Consequently, one has
\[
  \Phi_{c_*}(x_{k_{l + 1}}, u_{k_{l + 1}}) - \Phi_{c_*}(x_{k_l}, u_{k_l}) \le - \frac{\theta}{4}
  \quad \forall l \ge l_1.
\]
Therefore $\Phi_{c_*}(x_{k_l}, u_{k_l}) \to - \infty$ as $l \to \infty$, which contradicts 
Assumption~\ref{assumpt:PenaltyFuncBoundedBelow}.
\end{proof}

\begin{remark}
From the proof of the previous theorem it follows that the conclusion of the theorem remains hold true in the general
nonsmooth case (that is, without the additional assumption on differentiability of $G_i(x, u, t)$, $H_i(x, u, t)$, and
$q_s(x, u, t)$), if the corresponding subsequences of subgradients $\{ V_{k_l i} \}$, $\{ W_{k_l i} \}$, and 
$\{ z_{k_l s} \}$ converge not only weakly, but also a.e. on $[0, T]$. More generally, for the validity of the theorem
in the fully nonsmooth case it is sufficient to suppose that equality \eqref{eq:PenFuncValueConvergence2} holds true,
that is, all integral terms in expression \eqref{eq:PenFuncAtIterateDef} for $Q_c(x, u, x_{k_l}, u_{k_l}, V_{k_l})$
corresponding to the constraints of the problem $(\mathcal{P})$ converge as $l \to \infty$ for any $(x, u) \in X_0$. 
\end{remark}

\subsection{Convergence of control inputs vs. convergence of trajectories}
\label{subsect:Control_vs_traj}

Let us prove an interesting property of sequences generated by B-STEP-DCA. Namely, we will show that under some natural
assumptions the convergence of a subsequence of controls $\{ u_{k_l} \}$ implies the convergence of the corresponding
subsequence of trajectories. In particular, if the sequence of controls $\{ u_k \}$ converges, then the corresponding
sequence of trajectories $\{ x_k \}$ converges as well. Thus, under some natural assumptions the convergence of the
sequence $\{ (x_k, u_k) \}$ generated by B-STEP-DCA is completely defined by the convergence of the sequence of controls
$\{ u_k \}$, despite the fact that this method treats $x$ and $u$ as independent variables.

\begin{theorem} \label{thrm:TrajectoryConvVsControlConv}
Let the sequence $\{ (x_k, u_k) \}$ be generated by Algorithmic Pattern~\ref{alg:Boosted_SEP_DCA}, the sequence 
$\{ x_k \}$ be bounded in $W^{1, r}([0, T]; \mathbb{R}^n)$ for some $1 < r \le \infty$, and $\varphi(x_k, u_k) \to 0$ as
$k \to \infty$. Suppose also that some subsequence of controls $\{ u_{k_l} \}$ is uniformly bounded and converges in
$L^1([0, T]; \mathbb{R}^m)$ to a function $u_* \in L^{\infty}([0, T]; \mathbb{R}^m)$. Then the corresponding subsequence
of trajectories $\{ x_{k_l} \}$ is relatively compact in $W^{1, 1}([0, T]; \mathbb{R}^n)$, its limit points $x_*$ belong
to $W^{1, \infty}([0, T]; \mathbb{R}^n)$, and all pairs $(x_*, u_*)$ are feasible for the problem $(\mathcal{P})$. If,
in addition, the initial condition $x(0)$ is fixed (e.g. $X_0 = \{ (x, u) \in X \mid x(0) = x^0$ for some fixed 
$x^0 \in \mathbb{R}^n$), then the subsequence $\{ x_{k_l} \}$ converges in $W^{1, 1}([0, T]; \mathbb{R}^n)$.
\end{theorem}

\begin{proof}
Let us prove the relative compactness of the sequence $\{ x_{k_l} \}$. Indeed, choose any subsequence of this sequence.
For the sake of convenience, denote it by $\{ x_{k_l} \}$. By our assumption this sequence is bounded in 
$W^{1, r}([0, T]; \mathbb{R}^n)$ for some $1 < r \le + \infty$. One can obviously suppose that $r < + \infty$. Since the
space $W^{1, r}([0, T]; \mathbb{R}^n)$ is reflexive (see, e.g. Adams\cite{Adams}, Thm.~3.5), one can extract a
subsequence, which we denote again by $\{ x_{k_l} \}$, weakly converging to some 
$x_* \in W^{1, r}([0, T]; \mathbb{R}^n)$. By the Rellich-Kondrachov theorem (Adams\cite{Adams}, Thm.~6.2) the imbedding
of $W^{1, r}([0, T]; \mathbb{R}^n)$ into $C([0, T]; \mathbb{R}^n)$ is compact, which implies that the sequence 
$\{ x_{k_l} \}$ converges to $x_*$ uniformly on $[0, T]$ and, therefore, is uniformly bounded on $[0, T]$.

Replacing, if necessary, the sequence $\{ u_{k_l} \}$ by its subsequence one can suppose that it converges to $u_*$
almost everywhere on $[0, T]$. Therefore, $F(x_{k_l}(\cdot), u_{k_l}(\cdot), \cdot)$ converges to $F(x_*(\cdot),
u_*(\cdot), \cdot)$ almost everywhere on $[0, T]$. Hence by applying Lebesgue's dominated convergence theorem and the
growth conditions on the mapping $F$ from Assumption~\ref{assumpt:ContDiff} one can readily check that
$F(x_{k_l}(\cdot), u_{k_l}(\cdot), \cdot)$ converges to $F(x_*(\cdot), u_*(\cdot), \cdot)$ in $L^1([0, T],
\mathbb{R}^n)$.

In turn, from the assumption that $\varphi(x_k, u_k) \to 0$ as $k \to \infty$ it follows that
\[
  \lim_{l \to \infty} \sum_{i = 1}^n \int_0^T \big| \dot{x}_{k_l i}(t) - F_i(x_k(t), u_k(t), t) \big| = 0,
\]
that is, the function $\dot{x}_{k_l}(\cdot) - F_i(x_{k_l}(\cdot), u_{k_l}(\cdot), \cdot)$ converges to zero in 
$L_1([0, T], \mathbb{R}^n)$. Consequently, by noting that
\[
  \dot{x}_{k_l}(\cdot) = \Big( \dot{x}_{k_l}(\cdot) - F_i(x_{k_l}(\cdot), u_{k_l}(\cdot), \cdot) \Big) 
  + F(x_{k_l}(\cdot), u_{k_l}(\cdot), \cdot)
\]
one can conclude that the sequence $\{ \dot{x}_{k_l} \}$ converges in $L^1([0, T], \mathbb{R}^n)$ to 
$F(x_*(\cdot), u_*(\cdot), \cdot)$. Bearing in mind this fact and the weak convergence of $\{ x_{k_l} \}$ to $x_*$ in
$W^{1, r}([0, T]; \mathbb{R}^n)$ one gets that the sequence $\{ x_{k_l} \}$ converges to $x_*$ in 
$W^{1, 1}([0, T]; \mathbb{R}^n)$ and $\dot{x}_*(t) = F(x_*(t), u_*(t), t)$ for a.e. $t \in [0, T]$. 

By the assumption of the Theorem, $u_* \in L^{\infty}([0, T]; \mathbb{R}^m)$. Consequently, by
Assumption~\ref{assumpt:ContDiff} the function $\dot{x}_*(\cdot) = F(x_*(\cdot), u_*(\cdot), \cdot)$ is essentially
bounded. Thus, the subsequence of trajectories $\{ x_{k_l} \}$ is relatively compact in $W^{1, 1}([0, T]; \mathbb{R}^n)$
and its limit points belong to $W^{1, \infty}([0, T]; \mathbb{R}^n)$.

Let us check the feasibility of $(x_*, u_*)$. Indeed, from the uniform convergence of $\{ x_{k_l} \}$ to $x_*$ and the
fact that $\varphi(x_k, u_k) \to 0$ as $k \to \infty$ it obviously follows that $f_i(x_*(0), x_*(T)) \le 0$ for all 
$i \in \mathcal{I}$, and $f_j(x_*(0), x_*(T)) = 0$ for all $j \in \mathcal{E}$. Moreover, $(x_*, u_*) \in X_0$, since 
the set $X_0$ is closed in the topology of the space 
$W^{1, 1}([0, T]; \mathbb{R}^n) \times L^1([0, T]; \mathbb{R}^m)$ by our assumption.

Clearly, the sequence $\Xi_s(x_{k_l}(\cdot), u_{k_l}(\cdot), \cdot)$ converges to $\Xi_s(x_*(\cdot), u(\cdot), \cdot))$
almost everywhere on $[0, T]$. Hence by Fatou's lemma
\[
  \int_0^T \max\{ 0, \Xi_s(x_*(t), u_*(t), t) \} \, dt 
  \le \liminf_{l \to \infty} \int_0^T \max\{ 0, \Xi_s(x_{k_l}(t), u_{k_l}(t), t) \} \, dt
  \le \liminf_{l \to \infty} \varphi(x_k, u_k) = 0.
\] 
Consequently, $\Xi_s(x_*(t), u_*(t), t) \le 0$ for a.e. $t \in [0, T]$, and the pair $(x_*, u_*)$ is feasible for the
problem $(\mathcal{P})$.

Suppose finally that $x(0)$ is fixed. Let us check that in this case a limit point of the sequence $\{ x_{k_l} \}$ in
$W^{1, 1}([0, T]; \mathbb{R}^n)$ is unique, which due to the relative compactness of this sequence implies that it
converges. To this end, note that, as we have just proved, limit points of this sequence are solutions of the
differential equation 
\begin{equation} \label{eq:DiffEqUniqueness}
  \dot{x}(t) = F(x(t), u_*(t), t), \quad t \in [0, T] 
\end{equation}
The growth condition on subdifferentials of the functions $G_i$ and $H_i$ from Assumption~\ref{assumpt:ContDiff}
along with Thm.~24.7 from Rockafellar\cite{Rockafellar} imply that the map $(x, u) \mapsto F_i(x, u, t)$ is Lipschitz
continuous on the set $\{ (x, u) \in \mathbb{R}^{n + m} \mid |x| + |u| \le R \}$ with Lipschitz constant 
$L = 2 \gamma_R$. Consequently, by the uniqueness theorem (see, e.g. Filippov\cite{Filippov}, Thm.~1.1.2) 
the differential equation \eqref{eq:DiffEqUniqueness} has a unique absolutely continuous solution on $[0, T]$ for any
fixed initial condition. In other words, a limit point of the sequence $\{ x_{k_l} \}$ in 
$W^{1, 1}([0, T]; \mathbb{R}^n)$ is unique.
\end{proof}

\subsection{Convergence of the infeasibility measure}
\label{subsect:InfeasMeas}

Let us finally analyse convergence of the infeasibility measure $\varphi(x_k, u_k)$. To this end, we need to introduce
an auxiliary definition of criticality for this penalty term. Recall that $\Gamma(\cdot; x_k, u_k; V_k)$ is a global
convex majorant of $\varphi$ (see inequalities \eqref{eq:PenTermConvexMajorant}).

\begin{definition} \label{def:PenTermCriticality}
Let $\varepsilon \ge 0$ be given. A point $(x_*, u_*) \in X_0$ is called an $\varepsilon$-\textit{critical} point 
of the penalty term $\varphi$, if there exists a collection of subgradients $V$ of the corresponding convex functions at
$(x_*, u_*)$, defined in Eq.~\eqref{eq:SubgradientCollection}, such that $(x_*, u_*)$ is an $\varepsilon$-optimal
solution of the optimal feasibility problem
\[
  \minimise_{(x, u) \in X_0} \enspace \Gamma(x, u; x_*, u_*; V).
\]
If $\varepsilon = 0$, then $(x_*, u_*)$ is simply called \textit{critical} for the penalty term $\varphi$.
\end{definition}

Since any feasible point of the problem $(\mathcal{P})$ is a global minimiser of the penalty term $\varphi$ (recall that
this function is nonnegative and equal to zero if and only if the corresponding point is feasible), any such point is
$\varepsilon$-critical for $\varphi$ for any $\varepsilon \ge 0$. In the case $\varepsilon > 0$, any point 
$(x_*, u_*) \in X_0$ such that $\varphi(x_*, u_*) \le \varepsilon$ is also $\varepsilon$-critical for the penalty term
$\varphi$ by virtue of inequalities \eqref{eq:PenTermConvexMajorant}. However, if a point $(x_*, u_*)$ is infeasible
and $\varphi(x_*, u_*) > \varepsilon$, then its $\varepsilon$-criticality means that the constraints of the problem
$(\mathcal{P})$ are in some sense degenerate at this point.

Our aim is to show that under some additional assumptions all limit points of the sequence $\{ (x_k, u_k) \}$ generated
by B-STEP-DCA are either feasible for the problem $(\mathcal{P})$ or infeasible
$(\varepsilon_{\varphi} + \varepsilon_*)$-critical for the penalty term, where 
$\varepsilon_* = \limsup_{k \to \infty} \varepsilon_k$. In other words, our aim is, roughly speaking, to show that the
sequence generated by B-STEP-DCA either converges to a feasible point of the problem $(\mathcal{P})$ or gets stuck in a
small neighbourhood of an infeasible point at which the constraints of the problem $(\mathcal{P})$ are degenerate. The
proof of this result is in many ways similar to the proof of Theorem~\ref{thrm:ConvergenceToCriticalPoints}. Note,
however, that for the validity of this result the boundedness of the penalty parameter is \textit{not} necessary.

To prove the approximate criticality for the penalty term $\varphi$ of limit points of the sequence generated by
B-STEP-DCA, we impose the following additional assumption ensuring that the penalty parameter updates on Step~4 are
consistent with Step 3 of this method.

\begin{assumption} \label{assumpt:Step3and4Consistency}
For any $k \in \mathbb{N}$, if Step~3 is executed on iteration $k$ of Algorithmic Pattern~\ref{alg:Boosted_SEP_DCA},
then the point $(x_k[c_{k + 1}], u_k[c_{k + 1}])$ computed on Step~4 of Algorithmic Pattern~\ref{alg:Boosted_SEP_DCA}
satisfies the inequality
\begin{equation} \label{eq:Step3and4Consistency}
  \Gamma(x_k[c_{k + 1}], u_k[c_{k + 1}]; x_k, u_k; V_k) - \Gamma(x_k, u_k; x_k, u_k; V_k)
  \le \eta_1 \Big( \Gamma(\widehat{x}_k, \widehat{u}_k; x_k, u_k; V_k) - \Gamma(x_k, u_k; x_k, u_k; V_k) \Big)
\end{equation}
from Step~3.
\end{assumption}

It should be noted that this assumption is not restrictive. Indeed, if $\varepsilon_k = 0$ (that is, optimal solutions
of corresponding convex subproblems are computed exactly), then Assumption~\ref{assumpt:Step3and4Consistency} is
satisfied by virtue of Lemma~\ref{lem:InfeasMeasBehaviour}. In turn, if $\varepsilon_k > 0$, then from
Theorem~\ref{thrm:MethodCorrectness} it follows that there exists $c_* \ge c_k$ such that inequality
\eqref{eq:InfeasMeasDecay} on Step~3 and inequality \eqref{eq:PenFuncDecay} on Step~4 are satisfied for any 
$c_{k + 1} = c_+ \ge c_*$. In other words, increasing $c_{k + 1}$, if necessary, one can guarantee that
Assumption~\ref{assumpt:Step3and4Consistency} holds true.

Strictly speaking, one can ensure the validity of Assumption~\ref{assumpt:Step3and4Consistency} by checking 
the validity of inequality \eqref{eq:Step3and4Consistency} whenever the penalty parameter $c_{k + 1}$ is updated on
Step~4 and increasing it further, if inequality \eqref{eq:Step3and4Consistency} is violated. However, according to our
numerical experiments, such additional safeguard is completely redundant.

\begin{theorem} \label{thrm:InfeasMeasConvergence}
Let the sequence $\{ (x_k, u_k) \}$ be generated by Algorithmic Pattern~\ref{alg:Boosted_SEP_DCA}, the sequence of
optimality tolerances $\{ \varepsilon_k \}$ be bounded, $\sum_{k = 0}^{\infty} \nu_k < + \infty$, 
$\limsup_{k \to \infty} \varepsilon_k = \varepsilon^*$, and the parameter $\varkappa$ from
Assumption~\ref{assumpt:PenaltyIncrease} satisfies the inequality $\varkappa \ge 2$. Suppose also that for any 
$i \in \{ 1, \ldots, n \}$, and $s \in \mathcal{M}$ the functions $G_i(x, u, t)$, $H_i(x, u, t)$, and $q_s(x, u, t)$ are
differentiable in $(x, u)$ for a.e. $t \in [0, T]$ and their partial derivatives in $x$ and $u$ are Carath\'{e}odory
functions. Let finally one of the following assumption be valid:
\begin{enumerate}
\item{the sequence of penalty parameters $\{ c_k \}$ is bounded;}

\item{the sequence $\{ (x_k, u_k) \}$ converges in $W^{1, 1}([0, T]; \mathbb{R}^n) \times L^1([0, T]; \mathbb{R}^m)$
and the sequence $\{ u_k \}$ is bounded in $L^{\infty}([0, T]; \mathbb{R}^m)$;}

\item{$\sum_{k = 0}^{\infty} \max\{ 0, \varphi(x_k[c_{k + 1}], u_k[c_{k + 1}]) - \varphi(x_k, u_k) \} < + \infty$ and
$\overline{\alpha}_k = 0$ for any $k \in \mathbb{N}$ (that is, line search is not employed).}
\end{enumerate} 
Then limit points $(x_*, u_*) \in X$ of the sequence $\{ (x_k, u_k) \}$ in the topology of the space 
$W^{1, 1}([0, T]; \mathbb{R}^n) \times L^1([0, T]; \mathbb{R}^m)$ (if exist) are either feasible points of the problem
$(\mathcal{P})$ or infeasible $(\varepsilon_{\varphi} + \varepsilon^*)$-critical points for the penalty term
$\varphi$.
\end{theorem}

\begin{remark}
Let us comment on the third assumption in the theorem above. In the general case, the sequence $\{ \varphi(x_k, u_k) \}$
might not be monotone. The increase of the infeasibility measure $\varphi(x_k, u_k)$ can be caused by the employment of
the line search, which is nonmonotone by itself and even in the case $\nu_k = 0$ can increase the value of the
infeasibility measure. The increase of $\varphi(x_k, u_k)$ can also be caused by the fact that the current iterate
$(x_k, u_k)$ is an infeasible approximately critical point of the penalty term $\varphi$. In this case 
Algorithmic  Pattern~\ref{alg:Boosted_SEP_DCA} performs Step~2 and only the inequality 
$\varphi(x_k[c_+], u_k[c_+]) \le \varphi(x_k, u_k) + \varepsilon_{feas}$ can be satisfied without additional
assumptions. Thus, the infeasibility measure $\varphi(x_k, u_k)$ can increase between iterations. The third assumption
of Theorem~\ref{thrm:InfeasMeasConvergence}, roughly speaking, imposes a limitation on how much the infeasibility
measure $\varphi(x_k, u_k)$ can increase due to the criticality of $(x_k, u_k)$ for the penalty term by assuming that
the corresponding sum is bounded. Note that this assumption is satisfied, in particular, if only a finite number of
points in the sequence $\{ (x_k, u_k) \}$ are critical for the penalty term. One can also ensure that the validity of
this assumption by replacing fixed $\varepsilon_{feas}$ with a sequence 
$\{ \varepsilon_{feas}^k \} \subset (0, + \infty)$ such that $\sum_{k = 0}^{\infty} \varepsilon_{feas}^k < + \infty$.
\end{remark}

We divide the proof the theorem into four lemmas, three of which correspond to the three assumptions of
Theorem~\ref{thrm:InfeasMeasConvergence} and one of which is an auxiliary technical result on the penalty term
$\varphi$. For the sake of convenience, denote $Y = W^{1, 1}([0, T]; \mathbb{R}^n) \times L^1([0, T]; \mathbb{R}^m)$. We
start with the lemma corresponding to the first assumption of Theorem~\ref{thrm:InfeasMeasConvergence}.

\begin{lemma} \label{lem:InfeasMeasConvergence_BoundedPenParam}
If under the assumptions of Theorem~\ref{thrm:InfeasMeasConvergence} the sequence of penalty parameters $\{ c_k \}$ is
bounded, then limit points $(x_*, u_*) \in X$ of the sequence $\{ (x_k, u_k) \}$ in the topology of the space $Y$ (if
exist) are either feasible for the problem $(\mathcal{P})$ or $\max\{ \varepsilon_{\varphi}, \varepsilon^* \}$-critical
for the penalty term $\varphi$.
\end{lemma}

\begin{proof}
Let $(x_*, u_*) \in X$ be a limit point of the sequence $\{ (x_k, u_k) \}$ in the topology of the space $Y$. Then there
exists a subsequence $\{ (x_{k_l}, u_{k_l}) \}$ that converges to $(x_*, u_*)$ in the corresponding topology. Arguing in
the same way as in the first part of the proof of Theorem~\ref{thrm:ConvergenceToCriticalPoints} and replacing, if
necessary, the sequence $\{ (x_{k_l}, u_{k_l}) \}$ by its subsequence one can suppose that $x_{k_l}$ converges to $x_*$
uniformly on $[0, T]$, while $u_{k_l}$ converges to $u_*$ almost everywhere.

Almost literally repeating the second and third parts of the proof of Theorem~\ref{thrm:ConvergenceToCriticalPoints} and
replacing, if necessary, the sequence $\{ (x_{k_l}, u_{k_l}) \}$ by its subsequence one can suppose that the sequence
of subgradients $V_{k_l}$ weakly converges to some 
\begin{equation} \label{eq:LimitingCollectedSubgradients}
  V^* = (V_0^*(\cdot), V_1^*(\cdot), \ldots, V_n^*(\cdot), W_1^*(\cdot), \ldots, W_n^*(\cdot), 
  v_0^*, v_1^*, \ldots, v_{\ell_E}^*, w_{\ell_I + 1}^*, \ldots, w_{\ell_E}^*, 
  z_1^*(\cdot), \ldots z_{\ell_M}^*(\cdot))
\end{equation}
with
\begin{gather} \notag
  V_l^*(t) \in \partial_{x, u} H_l(x_*(t), u_*(t), t), \quad l \in \{ 0, 1, \ldots, n \}, \quad
  W_l^*(t) \in \partial_{x, u} G_l(x_*(t), u_*(t), t), \quad l \in \{ 1, \ldots, n \}, 
  \\ \label{eq:LimitingSubgradients}
  v_i^* \in \partial h_i(x_*(0), x_*(T)), \quad i \in \{ 0, 1, \ldots, \ell_E \}, \quad
  w_j^* \in \partial g_j(x_*(0), x_*(T)), \quad j \in \{ \ell_I + 1, \ldots, \ell_E \}, \quad
  \\ \notag
  z_s^*(t) \in \partial_{x, u} q_s(x_*(t), u_*(t), t), \quad s \in \{ 1, \ldots, \ell_M \}.
\end{gather}
for a.e. $t \in [0, T]$.

Next, arguing in the same way as in part 4 of the proof of Theorem~\ref{thrm:ConvergenceToCriticalPoints} and applying
the differentiability assumptions on the functions $G_i(x, u, t)$, $H_i(x, u, t)$, and $q_s(x, u, t)$ one can check that
the following equalities hold true:
\begin{align} \label{eq:LimitingGamma1}
  \lim_{l \to \infty} \Gamma(x_{k_l}, u_{k_l}; x_{k_l}, u_{k_l}; V_{k_l}) &= \Gamma(x_*, u_*; x_*, u_*; V^*),
  \\ \label{eq:LimitingGamma2}
  \lim_{l \to \infty} \Gamma(x, u; x_{k_l} u_{k_l}; V_{k_l}) &= \Gamma(x, u; x_*, u_*; V^*) 
  \quad \forall (x, u) \in X_0.
\end{align}
Now, arguing by reductio ad absurdum, suppose that the statement of the theorem is false. Then $(x_*, u_*)$ is
infeasible for the problem $(\mathcal{P})$ and is not $\max\{ \varepsilon_{\varphi}, \varepsilon^* \}$-critical for the
penalty term $\varphi$. Therefore $\varphi(x_*, u_*) > \varepsilon_{\varphi}$ and $(x_*, u_*)$ is not an
$\varepsilon^*$-optimal solution of the problem
\begin{equation} \label{prob:OptimalInfeasAtLimPoint}
  \minimise_{(x, u) \in X_0} \enspace \Gamma(x, u; x_*, u_*; V^*).
\end{equation}
In other words, $\Gamma_* + \varepsilon^* < \Gamma(x_*, u_*; x_*, u_*; V^*)$, where $\Gamma_*$ is the optimal value of
problem~\eqref{prob:OptimalInfeasAtLimPoint}. Therefore, there exist $\theta > 0$ and a feasible point of this
problem $(\widehat{x}, \widehat{u})$ such that
\begin{equation} \label{eq:NonCriticalForPenTerm}
  \Gamma(\widehat{x}, \widehat{u}; x_*, u_*; V^*)  + \varepsilon^* + \theta \le \Gamma(x_*, u_*; x_*, u_*, V^*).
\end{equation}
Hence with the use of relations~\ref{eq:LimitingGamma1} and \eqref{eq:LimitingGamma2} and the definition of
$\varepsilon_*$ one obtains that there exists $l_0 \in \mathbb{N}$ such that
\[
  \Gamma(\widehat{x}, \widehat{u}; x_{k_l}, u_{k_l}; V_{k_l}) + \varepsilon_{k_l} + \frac{\theta}{2}
  \le \Gamma(x_{k_l}, u_{k_l}; x_{k_l}, u_{k_l}; V_{k_l}) \quad \forall l \ge l_0.
\]
Consequently, one has 
\begin{equation} \label{eq:InfeasMeasDecaySubseq}
  \Gamma(\widehat{x}_{k_l}, \widehat{u}_{k_l}; x_{k_l}, u_{k_l}; V_{k_l}) + \frac{\theta}{2}
  \le \Gamma(x_{k_l}, u_{k_l}; x_{k_l}, u_{k_l}; V_{k_l}) \quad \forall l \ge l_0
\end{equation}
by the definition of $(\widehat{x}_k, \widehat{u}_k)$ (see Step~2 of Algorithmic Pattern~\ref{alg:Boosted_SEP_DCA}).

Recall that by our assumption $(x_*, u_*)$ is infeasible point of the problem $(\mathcal{P})$ and
$\varphi(x_*, u_*) > \varepsilon_{\varphi}$. Fix any $\gamma \in (\varepsilon_{\varphi}, \varphi(x_*, u_*))$. As was
noted in Remark~\ref{rmrk:PenFuncContinuity}, the function $\varphi$ is lower semicontinuous in the topology of the
space $Y$. Therefore, increasing $l_0$, if necessary, one can suppose that 
$\varphi(x_{k_l}, u_{k_l}) \ge \gamma$ for any $l \ge l_0$. Recall also that by our assumption the sequence of penalty
parameters is bounded, which due to Assumption~\ref{assumpt:PenaltyIncrease} implies that there exists $c_* > 0$ such
that $c_k = c_*$ for any sufficiently large $k$. In particular, we can suppose that $c_{k_l} = c_*$ for any $l \ge l_0$.

Fix any $l \in \mathbb{N}$ and consider $k_{l}$-th iteration of Algorithmic Pattern~\ref{alg:Boosted_SEP_DCA}. Note that
from inequality \eqref{eq:InfeasMeasDecaySubseq} it follows that if Algorithmic Pattern~\ref{alg:Boosted_SEP_DCA}
executes Step 2, then Step 3 is executed as well. Therefore, there are only two possibilities.

\textbf{Case 1. Steps 2 and 3 are not executed.} In this case, 
$\Gamma(x_{k_l}[c_{k_l}], u_{k_l}[c_{k_l}]; x_{k_l}, u_{k_l}; V_{k_l}) \le \varepsilon_{\varphi}$
according to Step~1 of Algorithmic Pattern~\ref{alg:Boosted_SEP_DCA}. Consequently,
according to Step~4 of this algorithmic pattern (see inequality \eqref{eq:PenFuncDecay}) one has
\begin{align*}
  Q_{c_{k_l + 1}}(x_{k_l}[c_{k_l + 1}], u_{k_l}[c_{k_l + 1}]; x_{k_l}, u_{k_l}; V_{k_l}) 
  - Q_{c_{k_l + 1}}(x_{k_l}, u_{k_l}; x_{k_l}, u_{k_l}; V_{k_l})
  &\le c_{k_l + 1} \eta_2 \big( \varepsilon_{\varphi} - \Gamma(x_{k_l}, u_{k_l}; x_{k_l}, u_{k_l}; V_{k_l}) \big)
  \\
  &\le c_{k_l + 1} \eta_2 \big( \varepsilon_{\varphi} - \varphi(x_{k_l}, u_{k_l}) \big),
\end{align*}
where the last inequality follows from the fact that $\Gamma(\cdot; x_k, u_k; V_k)$ is a global convex majorant of
$\varphi(\cdot)$. Applying the inequality $\varphi(x_{k_l}, u_{k_l}) \ge \gamma > \varepsilon_{\varphi}$ and arguing in
the same way as in the proof of Lemma~\ref{lem:PenFunctionDecay} one can verify that the inequality above implies that
\begin{equation} \label{eq:PenFuncDecayInfeasCritical1}
  \Phi_{c_{k_l + 1}}(x_{k_l + 1}, u_{k_l + 1}) \le \Phi_{c_{k_l + 1}}(x_{k_l}, u_{k_l}) 
  - \sigma \alpha_{k_l}^2 \rho_{k_l}^2 - c_{k_l + 1} \eta_2 (\gamma - \varepsilon_{\varphi}) + \nu_{k_l}.
\end{equation}

\textbf{Case 2. Step 3 is executed.} In this case, according to the description of 
Algorithmic Pattern~\ref{alg:Boosted_SEP_DCA} (see inequality \eqref{eq:InfeasMeasDecay}) and inequality
\eqref{eq:InfeasMeasDecaySubseq} one has
\begin{multline} \label{eq:PenTermMajorantDecay_NonCritLimit}
  \Gamma(x_{k_l}[c_+], u_{k_l}[c_+]; x_{k_l}, u_{k_l}; V_{k_l}) 
  - \Gamma(x_{k_l}, u_{k_l}; x_{k_l}, u_{k_l}; V_{k_l})
  \le \eta_1 
  \Big( \Gamma(\widehat{x}_{k_l}, \widehat{u}_{k_l}; x_{k_l}, u_{k_l}; V_{k_l}) 
  - \Gamma(x_{k_l}, u_{k_l}; x_{k_l}, u_{k_l}; V_{k_l}) \Big)
  \\
  \le - \eta_1 \frac{\theta}{2}.
\end{multline}
Hence by Step~4 of Algorithmic Pattern~\ref{alg:Boosted_SEP_DCA} one has
\begin{multline*}
  Q_{c_{k_l + 1}}(x_{k_l}[c_{k_l + 1}], u_{k_l}[c_{k_l + 1}]; x_{k_l}, u_{k_l}; V_{k_l}) 
  - Q_{c_{k_l + 1}}(x_{k_l}, u_{k_l}; x_{k_l}, u_{k_l}; V_{k_l})
  \\
  \le c_{k_l + 1} \eta_2 \Big( \Gamma(x_{k_l}[c_{k_l + 1}], u_{k_l}[c_{k_l + 1}]; x_{k_l}, u_{k_l}; V_{k_l})
  - \Gamma(x_{k_l}, u_{k_l}; x_{k_l}, u_{k_l}; V_{k_l}) \Big) 
  \le - c_{k_l + 1} \eta_1 \eta_2 \frac{\theta}{2}.
\end{multline*}
Consequently, arguing in the same way as in the proof of Lemma~\ref{lem:PenFunctionDecay} one obtains that
\begin{equation} \label{eq:PenFuncDecayInfeasCritical2}
  \Phi_{c_{k_l + 1}}(x_{k_l + 1}, u_{k_l + 1}) \le \Phi_{c_{k_l + 1}}(x_{k_l}, u_{k_l}) 
  - \sigma \alpha_{k_l}^2 \rho_{k_l}^2 - c_{k_l + 1} \eta_2 \frac{\theta}{2} + \nu_{k_l}.
\end{equation}

Combining inequalities \eqref{eq:PenFuncDecayInfeasCritical1} and \eqref{eq:PenFuncDecayInfeasCritical2} and
Lemma~\ref{lem:PenFunctionDecay} one obtains that 
\begin{align*}
  \Phi_{c_*}(x_{k_{l + 1}}, u_{k_{l + 1}}) - \Phi_{c_*}(x_{k_l}, u_{k_l})
  &= \sum_{s = k_l}^{k_{l + 1}} \Big( \Phi_{c_*}(x_{s + 1}, u_{s + 1}) - \Phi_{c_*}(x_s, u_s) \Big)
  \\
  &= - \sigma \sum_{s = k_l}^{k_{l + 1}}  \alpha_s^2 \rho_s^2 + \sum_{s = k_l}^{k_{l + 1}} \nu_s
  - c_* \eta_1 \eta_2 \frac{\min\{ \gamma - \varepsilon_{\varphi}, \theta \}}{2} 
\end{align*}
for any $l \ge l_0$. By our assumptions $\sum_{k = 0}^{\infty} \nu_k < + \infty$. Therefore, for any sufficiently large
$l \in \mathbb{N}$ one has
\[
  \Phi_{c_*}(x_{k_{l + 1}}, u_{k_{l + 1}}) - \Phi_{c_*}(x_{k_l}, u_{k_l})
  \le - \sigma \sum_{s = k_l}^{k_{l + 1}} \alpha_s^2 \rho_s^2 
  - c_* \eta_1 \eta_2 \frac{\min\{ \gamma - \varepsilon_{\varphi}, \theta \}}{4}.
\]
Consequently, $\Phi_{c_*}(x_{k_l}, u_{k_l}) \to - \infty$ as $l \to \infty$, which contradicts
Assumption~\ref{assumpt:PenaltyFuncBoundedBelow}.
\end{proof}

Let us now prove an auxiliary result on the penalty term $\varphi$.

\begin{lemma} \label{lem:PenTermLipschitz}
For any set $K \subset X$ that is bounded in $L^{\infty}([0, T]; \mathbb{R}^n \times \mathbb{R}^m)$ there exists $L > 0$
such that
\[
  \big| \varphi(x_1, u_1) - \varphi(x_2, u_2) \big| \le L \big\| (x_1 - x_2, u_1, - u_2) \big\|_Y 
  \quad \forall (x_1, u_1), (x_2, u_2) \in K,
\]
that is, $\varphi$ is Lipschitz continuous with respect to the norm of the space $Y$ on any subset of the space $X$
that is bounded in the $L^{\infty}$-norm.
\end{lemma}

\begin{proof}
Fix any set $K \subset X$ that is bounded in $L^{\infty}([0, T]; \mathbb{R}^n \times \mathbb{R}^m)$. Let $R > 0$ be
such that $\| (x, u) \|_{\infty} \le R$. Note that for any $(x, u) \in K$ one has $|x(t)| \le R$ for all 
$t \in [0, T]$, since the function $x$ is absolutely continuous. We will split the rest of the proof into three parts.

\textbf{Part 1.} By Thm.~24.7 from Rockafellar\cite{Rockafellar} the functions $f_i$, 
$i \in \mathcal{I} \cup \mathcal{E}$ are Lipschitz continuous on the set 
$\{ (x, y) \in \mathbb{R}^{2n} \mid \max\{ |x|, |y| \} \le R \}$ as the differences of convex functions. Hence taking
into account the fact that by the Sobolev imbedding theorem $\max\{ |x(0)|, |x(T)| \} \le C \| x \|_{1, 1}$ for any 
$x \in W^{1, 1}([0, T]; \mathbb{R}^n)$ and some $C > 0$ independent of $x$ one obtains that the function
$\sum_{i = 1}^{\ell_I} \max\{ 0, f_i(x(0), x(T) \} + \sum_{j = \ell_I + 1}^{\ell_E} |f_j(x(0), x(T))|$ from 
the definition of $\varphi$ (see equality \eqref{eq:PenTerm}) is Lipschitz continuous with respect to the norm
$\| \cdot \|_Y$ on $K$.

\textbf{Part 2.} Fix any $s \in \mathcal{M}$. Assumption~\ref{assumpt:ContDiff} and Thm.~24.7 from
Rockafellar\cite{Rockafellar} imply that there exists $L_s > 0$ such that
\[
  \big| \Xi_s(x_1, u_1, t) - \Xi_s(x_2, u_2, t) \big| \le L_s \big( |x_1 - x_2| + |u_1 - u_2| \big)
  \quad \forall (x_i, u_i) \in \mathbb{R}^n \times \mathbb{R}^m \colon \max\{ |x_i|, |u_i| \} \le R.
\]
As one can readily see, the inequality above implies that for any $(x_i, u_i) \in \mathbb{R}^n \times \mathbb{R}^m$ with
$\max\{ |x_i|, |u_i| \} \le R$ one has
\[
  \big| \max\{ 0, \Xi_s(x_1, u_1, t) \} - \max\{ 0, \Xi_s(x_2, u_2, t) \} \big| 
  \le L_s \big( |x_1 - x_2| + |u_1 - u_2| \big),
\]
from which it obviously follows that the terms $\sum_{s = 1}^{\ell_M} \int_0^T \max\{ 0, \Xi_s(x(t), u(t), t) \} \, dt$
from the definition $\varphi$ corresponding to the mixed constraints are Lipschitz continuous with respect to the norm
$\| \cdot \|_Y$.

\textbf{Part 3.} Let us finally consider the terms $\int_0^T |\dot{x}_i(t) - F_i(x(t), u(t), t)| \, dt$ (see equality
\eqref{eq:PenTerm}). Fix any $i \in \{ 1, \ldots, n \}$. Note that Assumption~\ref{assumpt:ContDiff} along with
Thm.~24.7 from Rockafellar\cite{Rockafellar} imply that
\[
  \big| F_i(x_1, u_1, t) - F_i(x_2, u_2, t) \big| \le L \big( |x_1 - x_2| + |u_1 - u_2| \big)
  \quad \forall (x_i, u_i) \in \mathbb{R}^n \times \mathbb{R}^m \colon \max\{ |x_i|, |u_i| \} \le R.
\]
for some $L > 0$. Hence by applying the reverse triangle inequality first and the trianle inequality second one obtains
that for any such $(x_i, u_i)$ and all $y_1, y_2 \in \mathbb{R}$ the following inequalities hold true:
\begin{align*}
  \Big| \big| y_1 - F_i(x_1, u_1, t) \big| - \big| y_2 - F_i(x_2, u_2, t) \big| \Big|
  &\le \Big| y_1 - F_i(x_1, u_1, t) - \big( y_2 - F_i(x_2, u_2, t) \big) \Big|
  \\
  &\le |y_1 - y_2| + \big| F_i(x_1, u_1, t) - F_i(x_2, u_2, t) \big|
  \\
  &\le |y_1 - y_2| + L \big( |x_1 - x_2| + |u_1 - u_2| \big).
\end{align*}
Consequently, for any $(x_1, u_1) \in K$ and $(x_2, u_2) \in K$ one has
\[
  \big| \dot{x}_{1i}(t) - F_i(x_1(t), u_1(t), t) \big| - \big| \dot{x}_{2i}(t) - F_i(x_2(t), u_2(t), t) \big|
  \le |\dot{x}_{1i}(t) - \dot{x}_{2i}(t)| + L \big( |x_1(t) - x_2(t)| + |u_1(t) - u_2(t)| \big)
\]
for a.e. $t \in [0, T]$. Integrating this inequality from $0$ to $T$ one obtains that
\[
  \int_0^T \big| \dot{x}_{1i}(t) - F_i(x_1(t), u_1(t), t) \big| \, dt
  - \int_0^T \big| \dot{x}_{2i}(t) - F_i(x_2(t), u_2(t), t) \big| \, dt 
  \le \max\{ L, 1 \} \| (x_1, u_1) - (x_2, u_2) \|_Y.
\]
Swapping $(x_1, u_1)$ and $(x_2, u_2)$, one can conclude that the terms 
$\int_0^T |\dot{x}_i(t) - F_i(x(t), u(t), t)| \, dt$ are indeed Lipschitz continuous with respect to the norm 
$\| \cdot \|_Y$ on the set $K$.
\end{proof}

Next we consider the lemma corresponding to the second assumption of Theorem~\ref{thrm:InfeasMeasConvergence}.

\begin{lemma}
If under the assumptions of Theorem~\ref{thrm:InfeasMeasConvergence} the sequence $\{ (x_k, u_k) \}$ converges in $Y$ to
some $(x_*, u_*) \in X$ and the sequence of controls $\{ u_k \}$ is essentially bounded, then the point $(x_*, u_*)$ is
either feasible for the problem $(\mathcal{P})$ or $(\varepsilon_{\varphi} + \varepsilon^*)$-critical 
for the penalty term $\varphi$.
\end{lemma}

\begin{proof}
Arguing by reductio ad absurdum, suppose that the point $(x_*, u_*)$ is infeasible for the problem $(\mathcal{P})$, but
is not $(\varepsilon_{\varphi} + \varepsilon^*)$-critical for the penalty term $\varphi$. Then, in particular,
\begin{equation} \label{eq:VarphiVsVarespilon}
  \varphi(x_*, u_*) > \varepsilon_{\varphi} + \varepsilon^*.
\end{equation}
For the sake of convenience, we divide the rest of the proof into two parts.

\textbf{Part~1.} Let us show that there exists $k_1 \in \mathbb{N}$ such that for any $k \ge k_1$ Algorithmic
Pattern~\ref{alg:Boosted_SEP_DCA} executes Step~2 on iteration $k$. Indeed, suppose that this claim is false. Then 
there exists a subsequence $\{ (x_{k_l}, u_{k_l}) \}$ such that Step~2 is not executed on iteration $k_l$ for any 
$l \in \mathbb{N}$. According to Step~1 it means that 
$\Gamma(x_{k_l}[c_{k_l}], u_{k_l}[c_{k_l}]; x_{k_l}, u_{k_l}; V_{k_l}) \le \varepsilon_{\varphi}$.
Note that if $c_{k_l + 1} > c_{k_l}$ (that is, the penalty parameter is increased on Step~4), then
$c_{k_l + 1} \ge c_{k_l} + 2$ by our assumption on the parameter $\varkappa$ from 
Assumption~\ref{assumpt:PenaltyIncrease}. Therefore, applying the first statement of Lemma~\ref{lem:InfeasMeasBehaviour}
and inequalities \eqref{eq:PenTermConvexMajorant} one gets that
\begin{equation} \label{eq:SmallInfeasMeasure}
  \varphi(x_{k_l}[c_{k_l + 1}], u_{k_l}[c_{k_l + 1}]) \le 
  \Gamma(x_{k_l}[c_{k_l + 1}], u_{k_l}[c_{k_l + 1}]; x_{k_l}, u_{k_l}; V_{k_l})
  \le \Gamma(x_{k_l}[c_{k_l}], u_{k_l}[c_{k_l}]; x_{k_l}, u_{k_l}; V_{k_l}) + \varepsilon_{k_l}
  \le \varepsilon_{\varphi} + \varepsilon_{k_l}
\end{equation}
for any $l \in \mathbb{N}$.

Recall that by our assumption the sequence $\{ u_k \}$ is essentially bounded, while the essential boundedness of the
sequence $\{ x_k \}$ follows from the fact that it converes to $x_*$ in $W^{1, 1}([0, T]; \mathbb{R}^n)$ and the
Sobolev imbedding theorem. By definition one has 
$(x_k[c_{k + 1}], u_k[c_{k + 1}]) \in \co\{ (x_k, u_k), (x_{k + 1}, u_{k + 1}) \}$ (see Step~5 of Algorithmic
Pattern~\ref{alg:Boosted_SEP_DCA}). Therefore, the sequences $\{ x_k[c_k] \}$ and $\{ u_k[c_k] \}$ are essentially
bounded as well. Hence with the use of Lemma~\ref{lem:PenTermLipschitz} and inequality \eqref{eq:SmallInfeasMeasure} 
one gets that there exists $L > 0$ such that for any $l \in \mathbb{N}$ one has
\begin{align*}
  \varphi(x_{k_l + 1}, u_{k_l + 1}) 
  &\le L \big\| (x_{k_l + 1} - x_{k_l}[c_{k_l + 1}], u_{k_l + 1} - u_{k_l}[c_{k_l + 1}]) \big\|_Y 
  + \varphi(x_{k_l}[c_{k_l + 1}], u_{k_l}[c_{k_l + 1}])
  \\
  &= L \big\| (x_{k_l + 1} - x_{k_l}[c_{k_l + 1}], u_{k_l + 1} - u_{k_l}[c_{k_l + 1}]) \big\|_Y 
  + \varepsilon_{\varphi} + \varepsilon_{k_l}
\end{align*}
Observe that by definition
\begin{multline} \label{eq:LineSearchStepNorm}
  \big\| (x_{k_l + 1} - x_{k_l}[c_{k_l + 1}], u_{k_l + 1} - u_{k_l}[c_{k_l + 1}]) \big\|_Y
  = \alpha_{k_l} \big\| (x_{k_l}[c_{k_l + 1}] - x_{k_l}, u_{k_l}[c_{k_l + 1}] - u_{k_l}) \big\|_Y
  \\
  \le (1 + \alpha_{k_l}) \big\| (x_{k_l}[c_{k_l + 1}] - x_{k_l}, u_{k_l}[c_{k_l + 1}] - u_{k_l}) \big\|_Y
  = \big\| (x_{k_l + 1} - x_{k_l}, u_{k_l + 1} - u_{k_l}) \big\|_Y.
\end{multline}
(see Step~5 of Algorithmic Pattern~\ref{alg:Boosted_SEP_DCA}). Thus, for any $l \in \mathbb{N}$ one has 
$\varphi(x_{k_{l + 1}}, u_{k_{l + 1}}) \le L \| (x_{k_l + 1} - x_{k_l}, u_{k_l + 1} - u_{k_l}) \|_Y +
\varepsilon_{\varphi} + \varepsilon_{k_l}$. Hence taking into account the fact that the sequence $\{ (x_k, u_k) \}$ (and
therefore any its subsequence) converges in $Y$ to $(x_*, u_*)$ and $\varphi$ is l.s.c. with respect to the topology of
the space $Y$ by Remark~\ref{rmrk:PenFuncContinuity} one obtains that 
$\varphi(x_*, u_*) \le \varepsilon_{\varphi} + \varepsilon_*$, which contradicts \eqref{eq:VarphiVsVarespilon}. Thus,
there exists $k_1 \in \mathbb{N}$ such that for any $k \ge k_1$ Algorithmic Pattern~\ref{alg:Boosted_SEP_DCA} executes
Step~2 on iteration $k$.

\textbf{Part~2.} Let us consider two cases corresponding to two possible types of behaviour of B-STEP-DCA.

\textbf{Case 1.} The inequality 
$\Gamma(\widehat{x}_k, \widehat{u}_k; x_k, u_k; V_k) \ge \Gamma(x_k, u_k; x_k, u_k; V_k)$ is
satisfied for an infinite number of points in the sequence $\{ (x_k, u_k) \}_{k \ge k_1}$ (that is, Step~3 is
\textit{not} executed on an infinite number of iterations). In other words, there exists a subsequence 
$\{ (x_{k_l}, u_{k_l}) \}$ such that $\Gamma(\widehat{x}_{k_l}, \widehat{u}_{k_l}; x_{k_l}, u_{k_l}; V_{k_l}) 
\ge \Gamma(x_{k_l}, u_{k_l}; x_{k_l}, u_{k_l}; V_{k_l})$
for any $l \in \mathbb{N}$, which by the definition of $(\widehat{x}_{k_l}, \widehat{u}_{k_l})$ means that
\begin{equation} \label{eq:CriticalPenTermSubsequence}
  \Gamma(x_{k_l}, u_{k_l}; x_{k_l}, u_{k_l}; V_{k_l}) 
  \le \inf_{(x, u) \in X_0} \Gamma(x, u; x_{k_l}, u_{k_l}; V_{k_l}) + \varepsilon_{k_l}
  \le  \Gamma(x, u; x_{k_l}, u_{k_l}; V_{k_l}) + \varepsilon_{k_l}
  \quad \forall (x, u) \in X_0.
\end{equation}
Now, arguing in the same way as in the proof of Theorem~\ref{thrm:ConvergenceToCriticalPoints} (see also
Lemma~\ref{lem:InfeasMeasConvergence_BoundedPenParam}) and replacing, if necessary, 
the sequence $\{ (x_{k_l}, u_{k_l}) \}$ by its subsequence one can suppose that the corresponding sequence of 
subgradients $V_{k_l}$ weakly converges to some $V^*$ satisfying conditions \eqref{eq:LimitingCollectedSubgradients} and
\eqref{eq:LimitingSubgradients}. Furthermore, equalities \eqref{eq:LimitingGamma1} and \eqref{eq:LimitingGamma2} also
hold true. Passing to the limit superior in \eqref{eq:CriticalPenTermSubsequence} with the use of these equalities one
obtains that
\[
  \Gamma(x_*, u_*; x_*, u_*; V^*) \le \Gamma(x, u; x_*, u_*; V^*) + \varepsilon^* \quad \forall (x, u) \in X_0,
\]
which by definition means that $(x_*, u_*)$ is an $\varepsilon^*$-critical point of $\varphi$.

\textbf{Case 2.} There exists $k_2 \ge k_1$ such that for any $k \ge k_2$ Algorithmic Pattern~\ref{alg:Boosted_SEP_DCA}
executes Step~2 and $\Gamma(\widehat{x}_k, \widehat{u}_k; x_k, u_k; V_k) < \Gamma(x_k, u_k; x_k, u_k; V_k)$ (that is,
for any $k \ge k_2$ both Step 2 and Step 3 are executed).

Once again, arguing in the same way as in the proof of Theorem~\ref{thrm:ConvergenceToCriticalPoints} and replacing, if
necessary, the sequence  $\{ (x_k, u_k) \}$ by its subsequence $\{ (x_{k_l}, u_{k_l}) \}$ one can suppose that 
the corresponding sequence of subgradients $V_{k_l}$ weakly converges to some $V^*$ satisfying
\eqref{eq:LimitingCollectedSubgradients} and \eqref{eq:LimitingSubgradients}, and, furthermore, equalities
\eqref{eq:LimitingGamma1} and \eqref{eq:LimitingGamma2} hold true.

By our assumption $(x_*, u_*)$ is not an $(\varepsilon_{\varphi} + \varepsilon^*)$-critical point of the penalty term
$\varphi$. Hence, in particular, $(x_*, u_*)$ is not an $\varepsilon^*$-optimal solution of problem
\eqref{prob:OptimalInfeasAtLimPoint}. Therefore,
there exist $(\widehat{x}, \widehat{u}) \in X_0$ and $\theta > 0$ such that inequality \eqref{eq:NonCriticalForPenTerm}
holds true. Then applying equalities \eqref{eq:LimitingGamma1} and \eqref{eq:LimitingGamma2} one gets that there exists
$l_0 \in \mathbb{N}$ such that inequality \eqref{eq:InfeasMeasDecaySubseq} is satisfied for any $l \ge l_0$. Hence
taking into account the facts that Step~3 is executed on iteration $k$ for any $k \ge k_2$ and
Assumption~\ref{assumpt:Step3and4Consistency} holds true, one obtains that the inequality 
\begin{multline*} \label{eq:PenTermMajorantDecay_NonCritLimit_Mod}
  \Gamma(x_{k_l}[c_{k_l + 1}], u_{k_l}[c_{k_l + 1}]; x_{k_l}, u_{k_l}; V_{k_l}) 
  - \Gamma(x_{k_l}, u_{k_l}; x_{k_l}, u_{k_l}; V_{k_l})
  \\
  \le \eta_1 
  \Big( \Gamma(\widehat{x}_{k_l}, \widehat{u}_{k_l}; x_{k_l}, u_{k_l}; V_{k_l}) 
  - \Gamma(x_{k_l}, u_{k_l}; x_{k_l}, u_{k_l}; V_{k_l}) \Big)
  \le - \eta_1 \frac{\theta}{2}.
\end{multline*}
is satisfied for any $l \ge l_0$. Therefore, by inequalities \eqref{eq:PenTermConvexMajorant} one has
\[
  \varphi(x_{k_l}[c_{k_l + 1}], u_{k_l}[c_{k_l + 1}]) - \varphi(x_{k_l}, u_{k_l}) \le - \eta_1 \frac{\theta}{2} 
  \quad \forall l \ge l_0.
\]
Hence with the use of Lemma~\ref{lem:PenTermLipschitz} one gets that
\[
  \varphi(x_{k_l + 1}, u_{k_l + 1}) - \varphi(x_{k_l}, u_{k_l}) \le - \eta_1 \frac{\theta}{2} + 
  L \big\| (x_{k_l + 1} - x_{k_l}[c_{k_l + 1}], u_{k_l + 1} - u_{k_l}[c_{k_l + 1}]) \big\|_Y.
\]
for any $l \ge l_0$. Consequently, applying inequalities \eqref{eq:LineSearchStepNorm} and
Lemma~\ref{lem:PenTermLipschitz} once again one obtains that
\begin{align*}
  \varphi(x_{k_{l + 1}}, u_{k_{l + 1}}) - \varphi(x_{k_l}, u_{k_l})
  &= \big( \varphi(x_{k_{l + 1}}, u_{k_{l + 1}}) - \varphi(x_{k_l + 1}, u_{k_l + 1}) \big)
  + \big( \varphi(x_{k_l + 1}, u_{k_l + 1}) - \varphi(x_{k_l}, u_{k_l})\big)
  \\
  &\le L \big\| (x_{k_{l + 1}} - x_{k_l + 1}, u_{k_{l + 1}} - u_{k_l + 1} \big\|_Y
  + L \big\| (x_{k_l + 1} - x_{k_l}, u_{k_l + 1} - u_{k_l}) \big\|_Y - \eta_1 \frac{\theta}{2}
\end{align*}
for any $l \ge l_0$. By our assumption the sequence $\{ (x_k, u_k) \}$ converges in $Y$ to $(x_*, u_*)$. Therefore, 
the subsequence $\{ (x_{k_l}, u_{k_l}) \}$ also converges in $Y$, which implies that it is a Cauchy sequence. Therefore,
there exists $l_1 \ge l_0$ such that 
\[
  \big\| (x_{k_{l + 1}} - x_{k_l + 1}, u_{k_{l + 1}} - u_{k_l + 1} \big\|_Y \le \frac{\eta_1 \theta}{L 8}, \quad
  \big\| (x_{k_l + 1} - x_{k_l}, u_{k_l + 1} - u_{k_l}) \big\|_Y \le \frac{\eta_1 \theta}{L 8} 
  \quad \forall l \ge l_1.
\]
Therefore, for any $l \ge l_1$ one has 
$\varphi(x_{k_{l + 1}}, u_{k_{l + 1}}) - \varphi(x_{k_l}, u_{k_l}) \le - \eta_1 \theta / 4$. Thus,
$\varphi(x_{k_l}, u_{k_l}) \to - \infty$ as $l \to \infty$, which is obviously impossible.
\end{proof}

Finally, let us consider the lemma corresponding to the third assumption of Theorem~\ref{thrm:InfeasMeasConvergence}.

\begin{lemma}
If under the assumptions of Theorem~\ref{thrm:InfeasMeasConvergence} 
$\sum_{k = 0}^{\infty} \max\{ 0, \varphi(x_k[c_{k + 1}], u_k[c_{k + 1}]) - \varphi(x_k, u_k) \} < + \infty$ and
$\overline{\alpha}_k = 0$ for all $k \in \mathbb{N}$, then limit points $(x_*, u_*) \in X$ of the sequence 
$\{ (x_k, u_k) \}$ in the topology of the space $Y$ (if exist) are either feasible for the problem $(\mathcal{P})$ or
$(\varepsilon^* + \varepsilon_{\varphi})$-critical for the penalty term $\varphi$.
\end{lemma}

\begin{proof}
Arguing by reductio ad absurdum, suppose that there exists a limit point $(x_*, u_*) \in X$ of the sequence 
$\{ (x_k, u_k) \}$ in the topology of the space $Y$ that is infeasible for the problem $(\mathcal{P})$, but is not 
$(\varepsilon^* + \varepsilon_{\varphi})$-critical for the penalty term $\varphi$. Then one can find $\gamma > 0$ such
that
\[
  \varphi(x_*, u_*) \ge \varepsilon_{\varphi} + \varepsilon^* + \gamma
\] 
and there exists a subsequence $\{ (x_{k_l}, u_{k_l}) \}$ converging to $(x_*, u_*)$ in $Y$. Hence by 
Remark~\ref{rmrk:PenFuncContinuity} there exists $l_0 \in \mathbb{N}$ such that
\begin{equation} \label{eq:InfeasMeasureNoncrit_LowerBound}
  \varphi(x_{k_l}, u_{k_l}) \ge \varepsilon_{\varphi} + \varepsilon_{k_l} + \frac{\gamma}{2}
\end{equation}
for any $l \ge l_0$.

Arguing in the same way as in the proof of Theorem~\ref{thrm:ConvergenceToCriticalPoints} (see also
Lemma~\ref{lem:InfeasMeasConvergence_BoundedPenParam}) and replacing, if necessary, 
the sequence $\{ (x_{k_l}, u_{k_l}) \}$ by its subsequence one can suppose that the corresponding sequence of 
subgradients $V_{k_l}$ weakly converges to some $V^*$ satisfying conditions \eqref{eq:LimitingCollectedSubgradients} and
\eqref{eq:LimitingSubgradients}. Furthermore, equalities \eqref{eq:LimitingGamma1} and \eqref{eq:LimitingGamma2} also
hold true.

From the fact that $(x_*, u_*)$ is not $(\varepsilon^* + \varepsilon_{\varphi})$-critical for $\varphi$ it follows that
there exists $(\widehat{x}, \widehat{u}) \in X_0$ and $\theta > 0$ such that inequality
\eqref{eq:NonCriticalForPenTerm} holds true. Hence with the use of equalities \eqref{eq:LimitingGamma1} and
\eqref{eq:LimitingGamma2} one obtains that there exists $l_1 \ge l_0$ such that inequality
\eqref{eq:InfeasMeasDecaySubseq} is satisfied for any $l \ge l_1$. Fix any $l \ge l_1$ and consider two cases.

\textbf{Case 1. Steps 2 and 3 are not executed on iteration $k_l$.} According to Step~1 of Algorithmic 
Pattern~\ref{alg:Boosted_SEP_DCA} in this case one has 
$\Gamma(x_{k_l}[c_{k_l}], u_{k_l}[c_{k_l}]; x_{k_l}, u_{k_l}; V_{k_l}) < \varepsilon_{\varphi}$ and the method jumps to
Step 4. Recall that if the penalty parameter is increased on Step 4, then by our assumption one has 
$c_{k_l + 1} > c_{k_l} + 2$. Consequently, with the use of the first statement of Lemma~\ref{lem:InfeasMeasBehaviour}
and inequalities \eqref{eq:PenTermConvexMajorant} one obtains that
\[
  \varphi(x_{k_l + 1}, u_{k_l + 1}) 
  \le \Gamma(x_{k_l}[c_{k_l + 1}], u_{k_l}[c_{k_l + 1}]; x_{k_l}, u_{k_l}; V_{k_l})
  \le \Gamma(x_{k_l}[c_{k_l}], u_{k_l}[c_{k_l}]; x_{k_l}, u_{k_l}; V_{k_l}) + \varepsilon_{k_l}
  \le \varepsilon_{\varphi} + \varepsilon_{k_l}.
\]
(here we used the fact that $(x_{k_l + 1}, u_{k_l + 1}) = (x_{k_l}[c_{k_l + 1}], u_{k_l}[c_{k_l + 1}])$, since the line
search is not employed). Hence taking into account inequality \eqref{eq:InfeasMeasureNoncrit_LowerBound} one gets that
$\varphi(x_{k_l + 1}, u_{k_l + 1}) - \varphi(x_{k_l}, u_{k_l}) \le \gamma/2$.

\textbf{Case 2. Step 2 is executed on iteration $k_l$.} Note that from the validity of inequality
\eqref{eq:InfeasMeasDecaySubseq} it follows that in this case Step 3 is executed as well. Consequently, due to
Assumption~\ref{assumpt:Step3and4Consistency} inequality \eqref{eq:PenTermMajorantDecay_NonCritLimit} holds true,
which by virtue of inequalities \eqref{eq:PenTermConvexMajorant} and the fact that the line search is not employed
implies that
\[
  \varphi(x_{k_l + 1}, u_{k_l + 1}) - \varphi(x_{k_l}, u_{k_l}) \le 
  \Gamma(x_{k_l}[c_{k_l + 1}], u_{k_l}[c_{k_l + 1}]; x_{k_l}, u_{k_l}; V_{k_l}) 
  - \Gamma(x_{k_l}, u_{k_l}; x_{k_l}, u_{k_l}; V_{k_l})
  \le - \eta_1 \frac{\theta}{2}.
\]
for any $l \ge l_0$.

Combining the two cases together one gets that for any $l \ge l_1$ the following inequalities holds true:
\begin{align*}
  \varphi(x_{k_{l + 1}}, u_{k_{l + 1}}) - \varphi(x_{k_l}, u_{k_l})
  &= \sum_{s = {k_l + 1}}^{k_{l + 1} - 1} \big( \varphi(x_{s + 1}, u_{s + 1}) - \varphi(x_s, u_s) \big)
  + \varphi(x_{k_l + 1}) - \varphi(x_{k_l})
  \\
  &\le \sum_{s = k_l + 1}^{\infty} \max\big\{ 0, \varphi(x_{s + 1}, u_{s + 1}) - \varphi(x_s, u_s) \big\}
  - \frac{\min\{ \gamma, \eta_1 \theta \}}{2}.
\end{align*}
By our assumption $\sum_{k = 0}^{\infty} \max\{ 0, \varphi(x_{k + 1}, u_{k + 1}) - \varphi(x_k, u_k) \} < + \infty$.
Therefore, for any sufficiently large $l \in \mathbb{N}$ one has
$\varphi(x_{k_{l + 1}}, u_{k_{l + 1}}) - \varphi(x_{k_l}, u_{k_l}) \le - \min\{ \gamma, \eta_1 \theta \}/4$.
Consequently, $\varphi(x_{k_l}, u_{k_l}) \to - \infty$ as $l \to \infty$, which is impossible.
\end{proof}

\section{Numerical experiments}
\label{sect:NumericalExperiments}

Let us present some results of numerical experiments illustrating the overall performance of several different version
of B-STEP-DCA for nonsmooth optimal control problems. Algorithmic Pattern~\ref{alg:Boosted_SEP_DCA} was implemented in
\textsc{Matlab} on a 3.7 GHz Intel(R) Core(TM) i3 machine with 16 GB of RAM. The parameters of B-STEP-DCA were chosen as
follows:
\[
  c_0 = 10, \quad \eta_1 = \eta_2 = \sigma = 0.1, \quad \zeta = 0.5, \quad
  \varepsilon_{\varphi} = 0.1, \quad \varepsilon_{feas} = 0.01.
\]
The penalty parameter was increased by the factor $\rho = 10$, each time the corresponding inequality was not satisfied.
The line search tolerances $\nu_k$ were chosen according to the strategy $(S3)$ as 
$\nu_k = 0.1 \| (x_k[c_{k + 1}] - x_k, u_k[c_{k + 1}] - u_k) \|_2^2 / (k + 1)$ for all $k \in \mathbb{N}$. The trial
step sizes $\overline{\alpha}_k$ were chosen according to the self-adaptive rule with $\gamma = 0.5$ and
$\overline{\alpha}_0 = 1$ (see Subsect.~\ref{subsect:TrialStepSize}). Finally, we used the inequalities
\[
  \Big| \Phi_{c_{k + 1}}(x_{k + 1}, u_{k + 1}) - \Phi_{c_{k + 1}}(x_k, u_k) \Big| < 0.001, \quad 
  \varphi(x_{k + 1}, u_{k + 1}) < 0.1
\]
as a stopping criterion for the algorithm. 

To numerically solve convex subproblems on Steps~1--4 of B-STEP-DCA, the interval $[0, T]$ was discretised into $N$
subintervals of equal length, the corresponding integrals were approximated with the use of the left Riemann sums,
while the derivative $\dot{x}(t)$ was approximated by finite forward differences. The corresponding discretised
problems were then solved with the use of \texttt{cvx}, a {\sc Matlab} package for specifying and solving convex
programs \cite{CVXPackage,GrantBoyd_CVX}. cvx package was used with default settings.

We applied B-STEP-DCA to a slightly modified version of the nonsmooth optimal control problem with a
nonsmooth nonconvex state constraint from Outrata\cite{Outrata83} that has the form:
\begin{align*}
  &\minimise \enspace J(x, u) = \int_0^T x_2(t) \max\{ 0, u(t) \} \, dt
  \\
  &\text{subject to} \enspace 
  \begin{aligned}[t]
    &\dot{x}_1(t) = x_2(t), \quad
    \dot{x}_2(t) = \frac{1}{m} u(t) - P x_2(t) |x_2(t)| - Q x_2(t), \quad t \in [0, T],
    \\
    &x_1(0) = x_2(0) = 0, \quad x_1(T) = 200, \quad x_2(T) = 0, \quad
    \underline{u} \le u(t) \le \overline{u} \quad t \in [0, T],
    \\
    & x_2(t) - \min\Big\{ 7, \max\big\{ 7 - 0.3(x_1 - 90), 4, 4 + 0.3(x_1 - 120) \big\} \Big\} \le 0, 
    \quad t \in [0, T].
  \end{aligned}
\end{align*}
The problem consists in driving a train from point $0$ to point $200$ and stopping there in time $T > 0$, while
obeying a certain speed limit described by the state constraint. Here $x_1(t)$ is the position of the train at time $t$,
$x_2(t)$ is its speed, $m$ is the mass of the train, while parameters $P > 0$ and $Q > 0$ describe its dynamic
behaviour. We approximated the discontinuous speed limit
\[
  x_2(t) \le \begin{cases}
    7, & \text{if } x_1(t) \notin [100, 120], 
    \\
    4, & \text{if } x_1(t) \in [100, 120]
  \end{cases}
\]
from Outrata\cite{Outrata83}, Example~5.1 by the corresponding continuous nonsmooth state constraint. Note that
\[
  \min\Big\{ 7, \max\big\{ 7 - 0.3(x - 90), 4, 4 + 0.3(x - 120) \big\} \Big\}
  = \begin{cases}
    7, & \text{if } x_1(t) \notin [90, 130],
    \\
    4, & \text{if } x_1(t) \in [100, 120],
    \\
    7 - 0.3(x - 90), &\text{if } x_1(t) \in [90, 100],
    \\
    4 + 0.3(x - 120), &\text{if } x_1(t) \in [120, 130].
  \end{cases}
\]
We used the same parameters of the problem as given in Example~5.1 of Ref.\cite{Outrata83}:
\[
  \underline{u} = - \frac{2}{3} 10^5, \quad \overline{u} = \frac{2}{3} 10^5, \quad
  P = 0.78 \cdot 10^{-4}, \quad Q = 0.28 \cdot 10^{-3}, \quad T = 48,
\]
but changed the value of the mass of the train $m$ from $3 \cdot 10^5$ to $10^5$, since for $m = 3 \cdot 10^5$ the
problem turned out to be infeasible. To improve numerical stability and avoid ill-conditioning, we rescaled the control
as follows: $\widehat{u} = (1/m) u$. This transformation corresponds to putting $m = 1$, $\underline{u} = - 2/3$, and 
$\overline{u} = 2/3$ for the original problem. Finally, the interval $[0, T]$ was discretised into $N = 480$
subintervals of equal length and we chose the functions $x_{01}(t) \equiv 0$, $x_{02}(t) \equiv 0$, and 
$u_0(t) \equiv 0$ as initial guess.

DC decompositions of the cost functional and nonlinear constraints were constructed with the use of the following
equalities
\begin{align*}
  &x_2 [u]_+ = \frac{1}{2} \Big( \big( [x_2]_+ + [u]_+ \big)^2 + [-x_2]_+^2 \Big)
  - \frac{1}{2} \Big( \big( [-x_2]_+ + [u]_+ \big)^2 + [x_2]_+^2 \Big), \quad
  x_2 |x_2| = [x_2]_+^2 - [- x_2]_+^2,
  \\
  &\min\Big\{ 7, \max\big\{ 7 - 0.3(x_1 - 90), 4, 4 + 0.3(x_1 - 120) \big\} \Big\}
  \begin{aligned}[t]
    &= \max\big\{ 7 - 0.3(x_1 - 90), 4, 4 + 0.3(x_1 - 120) \big\}
    \\
    &- \max\big\{ 0, - 0.3(x_1 - 90), - 3 + 0.3(x_1 - 120)  \big\},
  \end{aligned}
\end{align*}
which can be easily verified directly. Here $[t]_+ = \max\{ 0, t \}$.

To compare different versions of the method and better understand the effect of the line search and the difference
between $L_1/L_{\infty}$ penalty terms, we considered two different sets $X_0$:
\begin{align*}
  X_0 &= \Big\{ (x, u) \in X \Bigm| \dot{x}_1 = x_2, \quad x_1(0) = x_2(0) = 0, \quad x_1(T) = s, \quad x_2(T) = 0, 
  \quad
  \underline{u} \le u(t) \le \overline{u}, \quad t \in [0, T] \Big\}
  \\
  X_0 &= \Big\{ (x, u) \in X \Bigm| \dot{x}_1 = x_2, \quad 
  x_1(0) = x_2(0) = 0, \quad x_1(T) = s, \quad x_2(T) = 0 \Big\}
\end{align*}
In the first case, the presence of inequality constraints on control does not allow one to use line search. We denote
the corresponding method by STEP-DCA${}_0$. In the second case, the constraints 
$\underline{u} \le u(t) \le \overline{u}$ were included directly into the penalty function with the use of either the
$L_1$ or $L_{\infty}$ penalty terms:
\[
  \int_0^T \max\big\{ u(t) - 2/3, 0, - u(t) - 2/3 \big\} \, dt, \quad
  \max\big\{ 0, \| u \|_{\infty} - 2/3 \big\}.
\]
In the case when the line search is employed, we denote the corresponding methods by B-STEP-DCA${}_1$ and
B-STEP-DCA${}_{\infty}$, respectively, while in the case when the line search is not used, we denote these methods by
STEP-DCA${}_1$ and STEP-DCA${}_{\infty}$. Thus, we compared five different versions of Algorithmic
Pattern~\ref{alg:Boosted_SEP_DCA}. The results of our numerical experiments are presented in Table~\ref{tab:Outrata},
containing the computation time in seconds, the number of iterations $k$, the final value of the penalty parameter
$c_k$, as well as the the value of the cost functional and the infeasibility measure $\varphi$ at the last computed
point. The ``optimal'' control and the ``optimal''\footnote{In the context of nonsmooth DC optimisation, the terms
\textit{critical} control and \textit{critical} tachogram seem to be more appropriate, but such terms might be
incorrectly understood in the context of optimal control theory.} tachogram that assigns to each position of the train
the computed ``optimal'' velocity corresponding to the best computed value of the cost functional are given in
Figure~\ref{fig:Outrata}.

\begin{figure}[t!]
\centering
\includegraphics[width=0.45\linewidth]{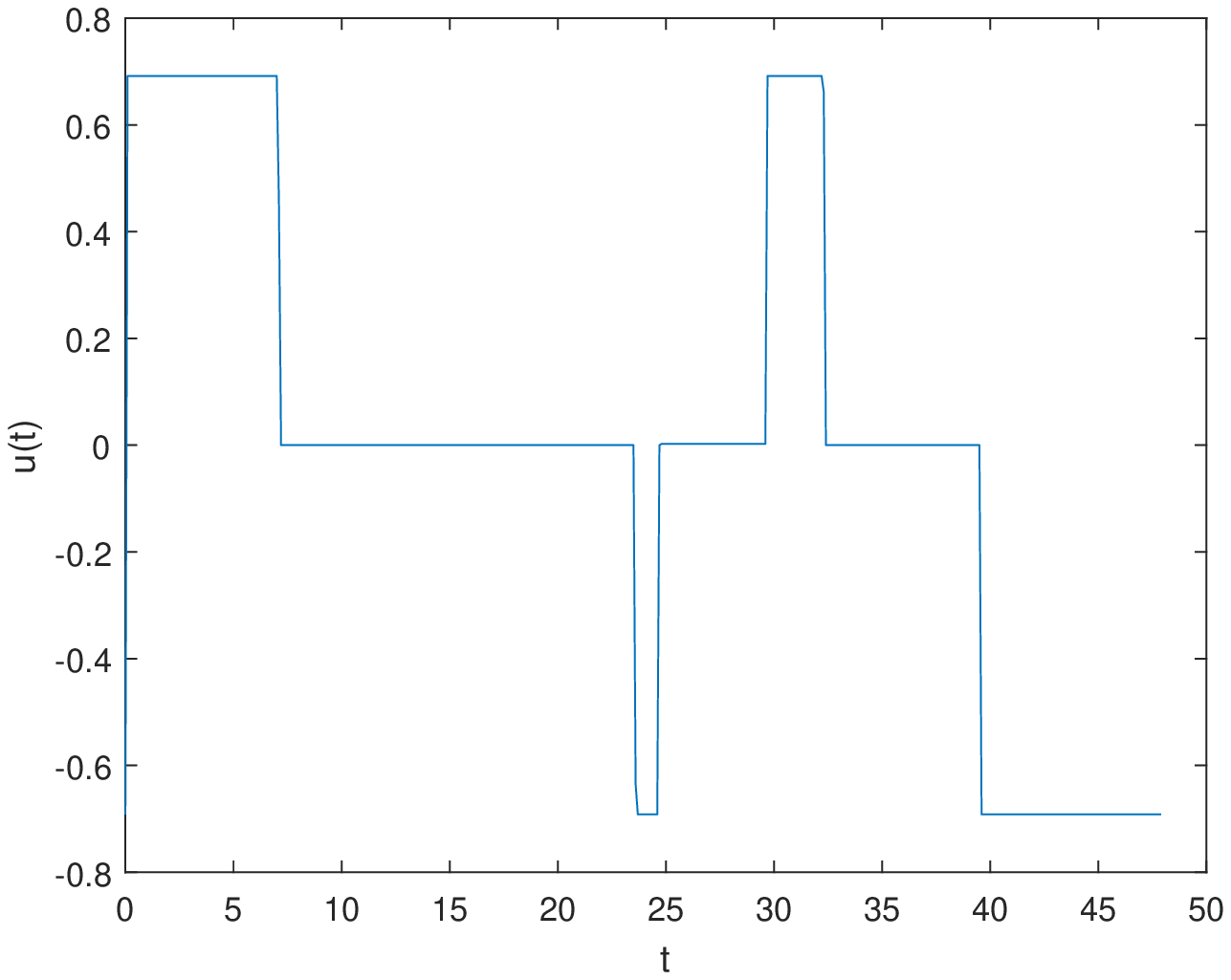}
\includegraphics[width=0.45\linewidth]{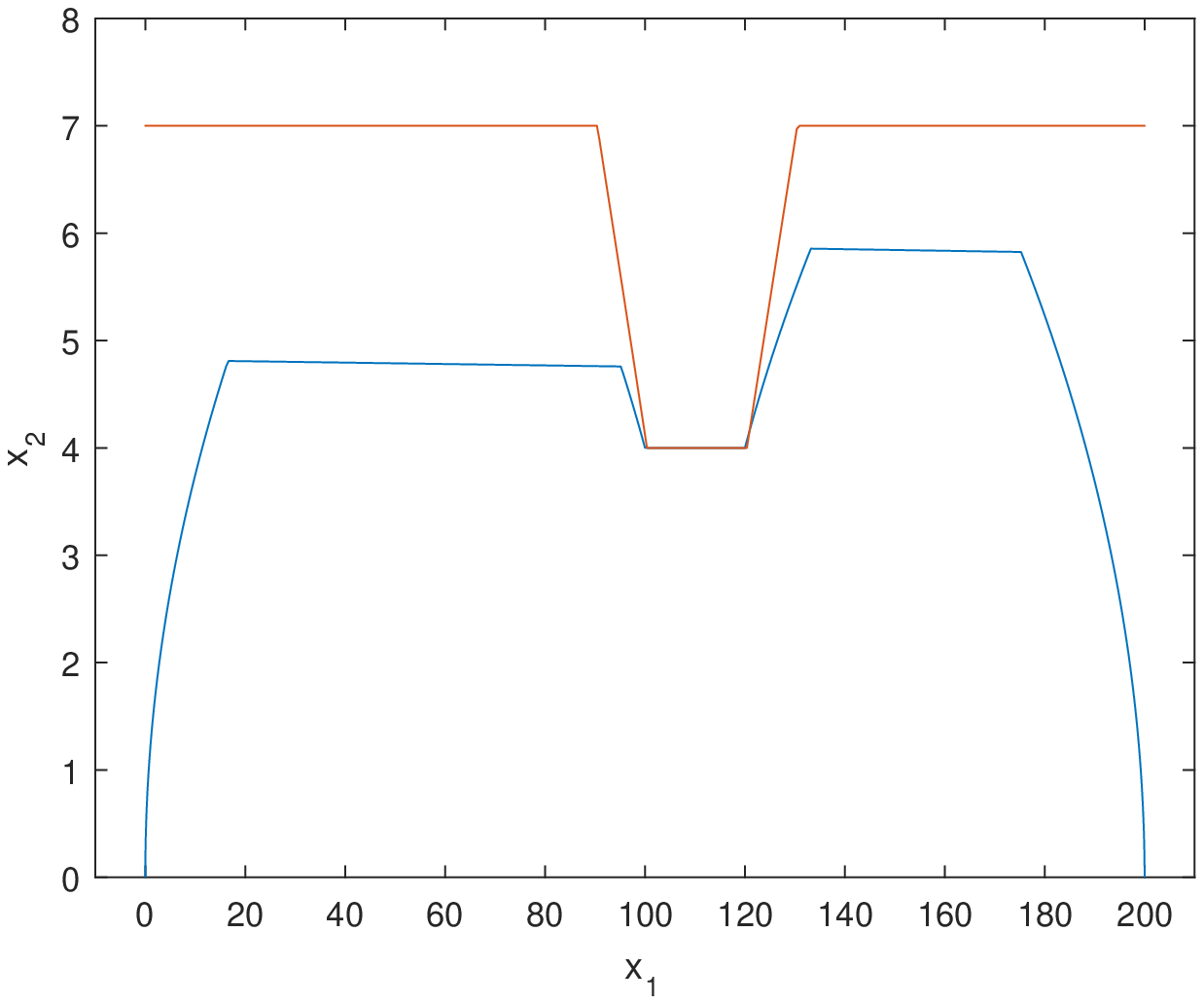}
\caption{The ``optimal'' control (left figure) and the ``optimal'' tachogram along with the graph of the state
constraint (right figure).}
\label{fig:Outrata}
\end{figure}

\begin{table} [th!]
\caption{The results of numerical experiments for 5 different versions of Algorithmic
Pattern~\ref{alg:Boosted_SEP_DCA}.}
\begin{tabular}{| c | c | c | c | c | c |} 
  \hline
  {} & $time$ & $k$ & $c_k$ & $J(x_k, u_k)$ & $\varphi(x_k, u_k)$ \\ 
  \hline
  STEP-DCA${}_0$ & $489.03$ & $44$ & $100$ & $21.8549$ & $0.0064$ \\
  \hline
  STEP-DCA${}_1$ & $489.95$ & $45$ & $100$ & $21.9936$ & $0.006$ \\
  \hline
  B-STEP-DCA${}_1$ & $491.9$ & $40$ & $100$ & $21.9936$ & $0.0064$ \\
  \hline
  STEP-DCA${}_{\infty}$ & $607.67$ & $58$ & $100$ & $20.5023$ & $0.0195$ \\
  \hline
  B-STEP-DCA${}_{\infty}$ & $574.89$ & $55$ & $100$ & $20.5988$ & $0.0255$ \\
  \hline
\end{tabular}
\label{tab:Outrata}
\end{table}

Let us comment on the results of our numerical experiments. Firstly, it should be noted that in all 5 cases the
method increased the penalty parameter only once, on Step 3 to ensure sufficient decay of the infeasibility measure
closer towards the end of the optimisation process. In other words, the method had to solve the convex subproblem
\eqref{subprob:MainStep} the second time during the same iteration for a larger value of the penalty parameter only once
throughout the optimisation process.

Recall that we solved all convex subproblems with the use of \texttt{cvx} package with default settings. This package
reports whether the problem was solved successfully, inaccurately or the method failed to converge to an optimal
solution. In all 5 cases, several convex subproblems \eqref{subprob:MainStep} on Step~1 were solved inaccurately
by \texttt{cvx}, while some of the optimal infeasibility subproblems \eqref{subprob:OptInfeasMeas} on Step~2 were
either solved inaccurately or \texttt{cvx} failed to find an optimal solution of these subproblems. Note that such cases
are allowed by B-STEP-DCA due to the presence of nonzero optimality tolerances $\varepsilon_k$, which can be arbitrarily
large, as long as they remain bounded throughout iterations. Despite the fact that some convex subproblems were solved
inaccurately, (B-)STEP-DCA successfully found a point satisfying the stopping criterion and in the case when the line
search was not employed preserved the monotonicity of the sequence $\{ \Phi_{c_{k + 1}}(x_k, u_k) \}$ (see
Lemma~\ref{lem:PenFunctionDecay}), which demonstrates the notable robustness of
this method with respect to computational errors and inaccurate solution of convex optimisation subproblems.

For both $L_1$ and $L_{\infty}$ penalty terms, the number of iterations of the method before termination reduced when
the nonmonotone line search was employed. Moreover, in the case of $L_{\infty}$ penalty term the use of the
nonmonotone line search also allowed one to slightly reduce the computation time. One can expect that for more complex
problems the effect of the line search will be greater. In addition, let us note that the line search was indeed
nonmonotone, in the sense that on some iterations the value of the penalty function decreased (sometimes very
significantly) after the line search was performed, while on others the value of the penalty function increased as the
result of the line search.

Finally, it is worth mentioning that although the values of the cost functional are different for different versions of
Algorithmic Pattern~\ref{alg:Boosted_SEP_DCA}, the controls and the trajectories computed in all 5 cases had exactly the
same structure depicted in Fig.~\ref{fig:Outrata}. The smaller value of the cost functional corresponded to slightly
longer period of acceleration of the train during the first phase, that is, before decelerating to satisfy the speed
limit. Let us note that the value of the cost functional $20.5023$ is not optimal. By carefully choosing an initial
point with the same structure as depicted in Fig.~\ref{fig:Outrata}, but having a longer initial acceleration period, we
were able to find an admissible control corresponding to the smaller value of the cost functional 
$J(x, u) \approx 19.33$. This fact clearly demonstrates that one can guarantee convergence of B-STEP-DCA only to
critical points of an optimal control problem, which might be non-optimal, as well as highlights the well-known issue of
convergence of \textit{nonsmooth} DC optimisation methods to non-optimal critical
points, which might be even locally non-optimal (see
Refs.\cite{deOliveiraTcheou,deOliveira2019,PangRazaviyaynAlvarado}). Therefore, for practical problems it
is advisable to combine B-STEP-DCA with global search techniques for DC optimal control problems developed by
Strekalovsky et al\cite{Strekalovsky2013,StrekalovskyYanulevich2008,StrekalovskyYanulevich2013,StrY2016}.

\section{Conclusions}

We presented an exact penalty-type method based on the boosted DCA, called B-STEP-DCA, for solving nonsmooth optimal
control problems with nonsmooth state and control constraints having DC structure, that is, optimal control problems
whose cost functional and constraints can be represented as the difference of smooth or nonsmooth convex functions. The
method is based on solving a sequence of penalised convex optimisation subproblems obtained from the original problem
via linearisation of the concave part of the cost functional and constraints. The method adaptively adjusts the penalty
parameter following the steering exact penalty rules and employs a nonmonotone line search technique with adaptive
rules for choosing the trial step sizes and nonmonotone line search tolerances. 

We proved the correctness of the method in the general case (that is, we proved the fact that only a finite number of
optimisation subproblems are solved on each iteration of the method) and provided conditions ensuring correctness of the
stopping criteria under which the method necessarily terminates after a finite number of iterations. We also presented a
detailed convergence analysis of the method, including the proof of its global convergence to approximately critical
points of the optimal control problem under consideration and some results on convergence of the infeasibility measure.

In the end of the paper, the performance of the method is illustrated by means of a numerical example. This example
demonstrated robustness of the proposed method with respect to computational errors and inaccurate solution of convex
optimisation subproblems, as well as the fact that the employment of nonmonotone line search leads to reduction of the
number of iteration of the method before termination and, in some cases, reduction of the computation time.

Our convergence analysis and numerical experiments show that in the general case B-STEP-DCA converges only to a critical
point of an optimal control problem, which is not necessarily locally optimal. Therefore, one important direction of 
future research consists in an extension of existing finite dimensional unconstrained nonsmooth DC optimisation
techniques ensuring convergence to d-stationary points (that is, points satisfying optimality conditions from
Prp.~\ref{prp:OptCond}) to the case of nonsmooth optimal control problems. Future research can also be concentrated on
an integration of B-STEP-DCA, as a local search method, into global search methods for DC optimal control problems.

\section*{Acknowledgments}

The author is sincerely grateful to his colleague A.V. Fominyh for many valuable discussions on nonsmooth optimal
control and exact penalty methods that played a significant role in the preparation of this paper. Also, the
author wishes to express his thanks to prof. A.S. Strekalovsky for many critical comments on the author's earlier papers
on  DC optimisation that helped to improve the quality of this article.

\bibliographystyle{abbrv}  
\bibliography{DC_OptimalControl_bibl}

\end{document}